\documentclass[11pt,a4paper]{article}

\usepackage[hyphens]{url}
\usepackage{dsfont}
\usepackage[utf8]{inputenc}
\usepackage{hyperref}
\hypersetup{linktocpage,
  colorlinks   = true, 
  urlcolor     = blue, 
  linkcolor    = blue,
  citecolor   = blue 
}
\usepackage[T1]{fontenc}
\usepackage[english]{babel}
\usepackage{amsmath}
\usepackage{amssymb}
\usepackage{amsthm}
\usepackage{sansmath}
\usepackage{calligra}
\usepackage{mathtools}
\usepackage[titletoc,toc,title]{appendix}
\usepackage{appendix}
\usepackage{ae}
\usepackage{icomma}
\usepackage{units}
\usepackage{color}
\usepackage{graphicx}
\usepackage{caption}
\usepackage{subcaption}
\usepackage{bbm}
\usepackage[square, numbers, sort]{natbib}
\usepackage{float}
\usepackage{multirow}
\usepackage{array}
\usepackage{geometry}
\usepackage{fancyhdr}
\usepackage{lettrine}
\usepackage{theoremref}
\usepackage{centernot}
\usepackage[normalem]{ulem}

\usepackage[shortlabels]{enumitem}
\setenumerate[0]{label=\alph*)}

\newcommand{\be}{\mathcal{I} }
\newcommand{\capac}{\mathrm{cap}}

\newcommand{\Pm}{\mathbb{P}}

\newcommand{\M}{\mathcal{M}}

\newcommand{\E}{ \mathbb{E}}

\newcommand{\I}{\mathbb{I}} 
\renewcommand{\tilde}{\widetilde}														
\newcommand{\N}{\mathbb{N}}
\newcommand{\Z}{\mathbb{Z}}
\newcommand{\Q}{\mathbb{Q}}
\newcommand{\R}{\mathbb{R}}
\newcommand{\ch}{\mathsf{c}}

\newcommand{\V}{\mathcal{V}}

\newcommand{\id}{\mathrm{d}}

\newcommand{\im}{\mathrm{i}}
\newcommand{\e}{\mathrm{e}}

\newcommand{\supp}{\mathrm{supp}\,}

\newcommand{\BC}{{\mathbb{C}}}
\newcommand{\D}{{\mathbb{D}}}

\newcommand{\BI}{{\mathbb{I}}}

\newcommand{\BR}{{\mathbb{R}}}

\newcommand{\CH}{{\mathcal{H}}}
\newcommand{\CI}{{\mathcal{I}}}

\newcommand{\CV}{{\mathcal{V}}}

\newcommand{\HH}{\mathbb{H}}

\newcommand{\intens}{u}

\renewcommand{\H}{\ensuremath{\mathbb{H}}}
\renewcommand{\Pm}{\ensuremath{\mathbb{P}}}

\let \le \leqslant
\let \leq \leqslant
\let \ge \geqslant
\let \geq \geqslant
\let \epsilon \varepsilon

\let \el k

\newtheorem{theorem}{Theorem}[section]

\newtheorem{proposition}[theorem]{Proposition}
\newtheorem{lemma}[theorem]{Lemma}

\newtheorem{corollary}[theorem]{Corollary}

\theoremstyle{definition}
\newtheorem{remark}[theorem]{Remark}

\newcommand{\tend}[2]{\displaystyle\mathop{\longrightarrow}_{#1\rightarrow#2}}
\newcommand{\eps}{\varepsilon}

\definecolor{Red}{rgb}{1,0,0}
\definecolor{Blue}{rgb}{0,0,1}
\definecolor{Olive}{rgb}{0.41,0.55,0.13}
\definecolor{Yarok}{rgb}{0,0.5,0}
\definecolor{Green}{rgb}{0,1,0}
\definecolor{MGreen}{rgb}{0,0.8,0}
\definecolor{DGreen}{rgb}{0,0.55,0}
\definecolor{Yellow}{rgb}{1,1,0}
\definecolor{Cyan}{rgb}{0,1,1}
\definecolor{Magenta}{rgb}{1,0,1}
\definecolor{Orange}{rgb}{1,.5,0}
\definecolor{Violet}{rgb}{.5,0,.5}
\definecolor{Purple}{rgb}{.75,0,.25}
\definecolor{Brown}{rgb}{.75,.5,.25}
\definecolor{Grey}{rgb}{.7,.7,.7}
\definecolor{Black}{rgb}{0,0,0}
\definecolor{lightgrey}{gray}{0.65}

\usepackage{constants}
\newcommand{\parenthezises}[1]{\arabic{#1}}
\numberwithin{equation}{section}
\newconstantfamily{Const}{
	symbol=C,
	format=\parenthezises,
}
\newconstantfamily{const}{
	symbol=c,
	format=\parenthezises,
}
\usepackage{titletoc}

\dottedcontents{section}[4em]{}{2.9em}{0.7pc}

\usepackage{titling}
\begin{document}

\pagenumbering{arabic}
\title{Percolation for two-dimensional excursion clouds and the discrete Gaussian free field}
\thanksmarkseries{arabic}
\author{A. Drewitz\thanks{University of Cologne. E-mail: \protect\url{adrewitz@uni-koeln.de}}, O. Elias\thanks{University of Cologne. E-mail: \protect\url{khallqvi@uni-koeln.de}}, A. Pr\'evost\thanks{University of Geneva. E-mail: \protect\url{alexis.prevost@unige.ch}},  J. Tykesson\thanks{Chalmers University of Technology and Gothenburg University. E-mail: \protect\url{johant@chalmers.se}}, F. Viklund\thanks{KTH Royal Institute of Technology. E-mail: \protect\url{frejo@kth.se}}}

\date{\today}

\maketitle

\thispagestyle{empty}
\begin{abstract}
We study percolative properties of excursion processes and the discrete Gaussian free field (dGFF) in the planar unit disk. We consider discrete excursion clouds, defined using random walks as a two-dimensional version of random interlacements, as well as its scaling limit, defined using Brownian motion. We prove that the critical parameters associated to vacant set percolation for the two models are the same and equal to $\pi/3.$ The value is obtained from a Schramm-Loewner evolution (SLE) computation. Via an isomorphism theorem, we use a generalization of the discrete result that also involves a loop soup (and an SLE computation) to show that the critical parameter associated to level set percolation for the  dGFF is strictly positive and smaller than $\sqrt{\pi/2}.$ In particular this entails a strict inequality of the type $h_*<\sqrt{2u_*}$ between the critical percolation parameters of the  dGFF and the two-dimensional excursion cloud. Similar strict inequalities are conjectured to hold in a general transient setup.

\end{abstract}
\setcounter{tocdepth}{1}
\vspace{8mm}
\renewcommand{\contentsname}{\centering {\small 
Contents}}

\begin{minipage}{0.9\textwidth}
{\small
\tableofcontents
}
\end{minipage}
\vspace{1cm}

\pagestyle{fancy}
\setlength{\headheight}{14pt} 
\fancyhf{}
\lhead{Percolation for 2D excursion clouds and the dGFF}
\cfoot{\thepage}

\newpage
\section{Introduction}

\subsection{Background and main result}
\label{sec:background}

The Brownian excursion cloud on the unit disk $\D,$ introduced by Lawler and Werner in \cite{lawler2000universality}, is a conformally invariant Poissonian cloud of planar Brownian motion trajectories, starting and ending on the boundary of $\D$. Roughly speaking, the number of trajectories is controlled by an intensity parameter $\intens>0.$ One can view the Brownian excursion cloud as a natural  version of Brownian interlacements in the hyperbolic disc. Varying the intensity parameter, one may consider percolative properties of the corresponding vacant set, that is, the set of points visited by no trajectory in the cloud. Given $r\in{[0,1)},$  consider the event that the disc of radius $r$ about $0$ can be connected to $\partial \mathbb{D}$ within the vacant set, and denote by $\intens_*^c(r)$ the associated percolation critical parameter, see \eqref{def:u*c}.
We will verify that percolation occurs with positive probability if and only if $u<\pi/3$ independently of $r$, that is, $\intens_*^c(r)$ equals $\pi/3$ for all $r\in{[0,1)}$, cf.\ Theorem~\ref{t.mainthm}. (This follows from an SLE computation and is certainly known to experts; an essentially equivalent statement for $u_*^c$ is contained in the informal discussion of \cite{werner-qian}, Section~5.)

There is an analogous discrete model in which one considers a random walk excursion cloud on $\D_n:=n^{-1}\Z^2\cap\D,$  see for instance \cite{ArLuSe-20a}. The corresponding question about vacant set percolation may then be formulated as follows. For fixed $r\in{[0,1)}$ and $u>0,$ we say that percolation occurs if the discrete ball $B_n(r)$ of radius $r$ (see below \eqref{eq:setDiscreteApprox}) can be connected to the boundary of $B_n(1-\eps)$ in the vacant set of the discrete excursion process at intensity $u$ on $\D_n$, as $n\rightarrow\infty$ and $\eps\rightarrow0$ in that order. (We will show that one may take $\eps \sim n^{-1/7}$ and only let $n \to \infty$.) In Theorem~\ref{the:maindisexc}, we prove that percolation in this sense occurs with positive probability if and only if $u<\pi/3,$ by comparing with the continuum scaling limit, the Brownian excursion cloud. In other words, the critical parameter $\intens_*^d(r)$ for the discrete excursion cloud, see \eqref{def:u*}, is also equal to $\pi/3$ for all $r\in{[0,1)},$ and similarly percolation does not occur at criticality.

The third model we analyze is the discrete Gaussian free field (dGFF) on the disk $\D_n$ with zero boundary condition on $\D_n^\ch.$ In this case, we are interested in percolative properties of its \emph{level sets} or \emph{excursion sets} as $n\rightarrow \infty,$ that is, the set of vertices where the field is larger than $h$ for some $h\in{\R}.$ Via isomorphism theorems, see for instance \cite[Proposition~2.4]{ArLuSe-20a}, which are a reformulation of similar theorems \cite{MR2892408,MR2932978,sznitman2013scaling,MR3492939,Lu-14} for interlacements, percolative properties of the level sets can be related to those of discrete excursion clouds.
Given $r\in{[0,1)},$ we study the event that $B_{n}(r)$ is connected to $B_n(1-\eps)$ for the level sets of the dGFF above level $h$ on $\D_n$ as $n\rightarrow\infty$ and $\eps\rightarrow0,$ and denote by $h_*^d(r)$ the associated critical level, see \eqref{def:h*}. This problem was for instance studied in \cite{DiLi-18,DiWiWu-20}, and it was proved in the last reference that percolation occurs above level zero for a slightly different percolation event. In Theorem~\ref{the:maindisexc} below, we adapt their argument to prove that there is percolation above some small positive level, that is $h_*^d(r)>0$ for all $r\in{(0,1)}.$  Moreover, we also prove that $h_*^d(r)\leq \sqrt{\pi/2}$ for all $r\in{[0,1)}.$

We thus obtain the following series of inequalities between the critical percolation  parameters $u_*^c(r)$ and $u_*^d(r)$ associated to the vacant set of the continuous and discrete excursion cloud, and $h_*^d(r)$ associated to level sets of the dGFF, which is a combination of Theorems~\ref{the:maindisexc}, \ref{the:maindisgff} and \ref{t.mainthm}.

\begin{theorem}
\label{the:main}
For all $r\in{(0,1)},$
\begin{equation}
\label{eq:strictinequalitycritparaintro}
   0<h_*^d(r)\leq\sqrt{\frac{\pi}2}<\sqrt{\frac{2\pi}3}=\sqrt{2\intens_*^d(r)}=\sqrt{2\intens_*^c(r)}.
\end{equation}
\end{theorem}

Note that the exact values $\sqrt{\pi/2}$ and $\sqrt{2\pi/3}$ in Theorem~\ref{the:main} depend on our choice of normalization in the definition of the excursion clouds and the dGFF, and we refer to Remark~\ref{rk:othercontperco},\ref{rk:differentnormalization} for more details.

\subsection{Relation to other models and further motivation}
We now explain how Theorem~\ref{the:main} is related to existing and conjectured results for other models, in particular random interlacements and the dGFF in dimension larger than three.

By Proposition~\ref{prop:localdescription} below, the problem of vacant set percolation for the Brownian excursion cloud on $\D$ can be seen as a natural two-dimensional equivalent of percolation for the vacant set of Brownian interlacements \cite{sznitman2013scaling,li2016percolative} on $\R^d,$ $d\geq3.$ But it is different from the model introduced in \cite{MR4125109}. In \cite{elias2017visibility}, a related property of the two-dimensional Brownian excursion cloud was studied, namely that of "visibility to infinity'', that is, the event that the origin can be connected to $\partial \mathbb{D}$ by a line segment contained in the vacant set started from $0$. (Infinity here is understood in the sense of hyperbolic distance.) It was proved that there is visibility to infinity from the origin if and only if $\intens<\pi/4,$ which directly implies that the associated critical parameter for vacant set percolation satisfies $\intens_*^c(0)\geq\pi/4.$ It was moreover proved in \cite{eliaslic2018} that $\intens_*^c(0)\leq \pi/2$ by comparing with a Poisson process of hyperbolic geodesics. The fact that $\intens_*^c(0)=\pi/3$ shows that this critical parameter sits strictly between these two previously obtained bounds.

In the discrete setting, the excursion cloud can be identified with random interlacements on $\D_n$ with infinite killing measure on $\D_n^\ch,$ see below \eqref{eq:muexcasinter}. In dimension $d\geq3,$ random interlacements on $\Z^d$ have been introduced in \cite{sznitman2010vacant}, but their definition can be extended to general transient weighted graphs, see \cite{teixeira2009interlacement}, even with a non-zero killing measure, see for instance \cite[Section~3]{Pre1}. The percolative properties of the vacant set associated to random interlacements on transient graphs have been studied intensely, see, e.g., \cite{sznitman2010vacant,teixeira2009interlacement,MR2891880}.

In dimension two, one cannot define random interlacements directly on the whole recurrent graph $\Z^2,$ and a possible alternative definition has been introduced in \cite{MR3475663}, where one essentially conditions on the origin being avoided by the walk to obtain a transient graph. In this context the percolation question is perhaps less natural, 
since the connected components of the associated vacant set at level $\alpha$ are always bounded, but depending on $\alpha$ its cardinality can be either finite or infinite. It is shown in \cite{MR3737923}, that this vacant set is infinite if and only if $\alpha\leq 1$.

When blowing up the set $\D_n$ to $B(n):=\{x\in{\Z^2}:\,|x|_2\leq n\},$ one can identify percolation for the vacant set of the two-dimensional discrete excursion process with percolation for the vacant set of random interlacements on $B(n)$ with infinite killing on $B(n)^\ch,$  which, as $n\rightarrow\infty,$ can be considered as another natural definition of percolation for the vacant set of interlacements in dimension two. Indeed, in dimension $d\geq3,$ random interlacements on the ball of radius $n,$ killed outside this ball, converge to random interlacements on $\Z^d,$ but this limit does not seem to be well-defined in dimension two.

Let us now turn to the dGFF on $\D_n,$ which is linked to the discrete excursion cloud on $\D_n$ via an isomorphism \cite{ArLuSe-20a}. This isomorphism can alternatively be seen from the point of view of random interlacements \cite{MR2892408, Lu-14,MR3492939,DrePreRod3}, or from the point of view of the random walk on finite graphs \cite{MR1813843,MR3978220}, see Remark~\ref{rk:otheriso}. A direct consequence of this isomorphism is the weak inequality $h_*^d(r)\leq\sqrt{2\intens_*^d(r)}$ between the critical parameter for the vacant set of random interlacements and the critical parameter for the level sets of the dGFF.

On transient graphs, level set percolation for the dGFF has also received significant attention in recent years \cite{MR3053773, MR3492939}. The isomorphism theorem relating its law with random interlacements has been a powerful tool for the study of the percolation of its level sets \cite{Sz-16,DrePreRod,DrePreRod2}, and also implies the weak inequality $h_*\leq\sqrt{2u_*}$ between the respective critical parameters of the level sets of the dGFF and the vacant set of random interlacements, see \cite[Theorem~3]{Lu-14} on $\Z^d,$ $d\geq3,$ or \cite[Corollary~2.5]{MR3940195} and \cite{DrePreRod3} on more general transient graphs. This inequality is strict on a large class of trees \cite{MR3492939,MR3765885}, and conjectured \cite[Remark~4.6]{MR3492939} to be strict on $\Z^d$ for all $d\geq3.$

The isomorphism between the discrete excursion cloud and the dGFF also involves a third process, which corresponds to an independent loop soup, see Proposition~\ref{prop:BEisom}. On $\D_n,$ taking advantage of this additional loop soup leads to the previously mentioned inequality $h_*^d(r)\leq \sqrt{\pi/2}.$  What is more, combined with our percolation results for the discrete excursion cloud, this entails the strict inequality $h_*^d(r)<\sqrt{2\intens_*^d(r)}$ in dimension two -- note that the weak version $h_*^d(r)\le\sqrt{2\intens_*^d(r)}$ of this inequality follows at once from the isomorphism. It is in general a challenge to establish this strong version of it, which so far has only been established on certain trees \cite{Sz-16,MR3765885}. In the continuous setting, one can also deduce from a similar isomorphism \cite{ArLuSe-20a} that the critical parameter associated to the complement of the first passage set of the two-dimensional dGFF is $\sqrt{\pi/2},$ see Corollary~\ref{cor:percocontGFF} for details. 

The inequality corresponding to $h_*^d(r)>0$ has also been obtained for the dGFF on $\Z^d,$ $d\geq3$ in \cite{DrePreRod2}, and on more general transient graphs in \cite{DrePreRod}. It also corresponds to a strict inequality between critical parameters, as we now explain. Denote by $\tilde{h}_*^d(r)$ the critical parameter associated to the percolation for the level sets of the dGFF on the cable system, or metric graph, associated to $\D_n,$ studied in dimension two in \cite{DiWi-18,DiWiWu-20,ArLuSe-20a}. This critical parameter is equal to zero for our notion of two-dimensional percolation, see \cite{ArLuSe-20a} (this can be proved by combining Lemma~4.13 and Corollary~5.1 therein), see also \cite{DiWiWu-20} for a similar result in a slightly different context. This equality $\tilde{h}_*=0$ has also been obtained for the dGFF on the cable system of a large class of transient graphs, see \cite{Lu-14,MR3492939,DrePreRod3}. Combining this with the positivity of $h_*^d(r),$ we obtain the following strict inequality between critical parameters: $\tilde{h}_*^d(r)<h_*^d(r).$

 Summarizing, Theorem~\ref{the:main} contains several strict inequalities between the critical parameters $\tilde{h}_*^d(r),$ $h_*^d(r)$ and $u_*^d(r),$ which complement known and conjectured results on transient graphs. See the discussion below \eqref{eq:strictinequalitycritpara} for more details.

\subsection{Comments on the proofs and outline of the paper}

Let us now discuss the ideas for the proof of Theorem~\ref{the:main}. The parameter $u_*^c(r),$ associated to percolation for the vacant set of Brownian excursions exploits the relation between excursion clouds and (variants of) Schramm-Loewner evolution processes, specifically the SLE$_{8/3}(\rho)$ process. Here, the weight $\rho$ is an explicit function of the intensity of the excursion process, see \eqref{eq:rhokappaalpha}. By the theory of conformal restriction  \cite{lawler2003conformal,werner2005conformal}, the excursions which start and end on the bottom half of $\partial\D$ produce an interface which can be described by an $\mathrm{SLE}_{8/3}(\rho)$ curve in $\mathbb{D}$ from $-1$ to $1$, see Lemma~\ref{l.brownianexcursionandlooprestriction}. Moreover, it is well-known \cite{lawler2003conformal} for which $\rho$ an $\mathrm{SLE}_{8/3}(\rho)$ curve hits $\partial\D,$ see Lemma~\ref{lem:SLE-ka-r}. This directly implies that the vacant set of the excursions that start and end on the bottom half of $\partial\D$ are connected to the bottom half of $\partial\D$ with positive probability if and only if the Brownian excursions have intensity less than $\pi/3.$ (See also Section~5 of \cite{werner-qian}.) Combining this with a symmetric result for the top part of $\partial\D$ and an argument involving the restriction property, see Lemma~\ref{lem:resexc}, the equality $u_*^c(r)=\pi/3$ follows easily. 

We now comment on the fact that the value of $u_*^c(r)$ does not depend on the choice $r \in [0,1).$ Any ball of radius $r \in [0,1)$ centred at the origin is almost surely hit by only finitely many excursions of the Brownian excursion process. These, however, could be removed at finite probabilistic cost due to an insertion-deletion tolerance property, as can be argued by use of the explicit formulas for the underlying Poisson point process, see Remark~\ref{rk:othercontperco},\ref{finiteenergy}. Hence the question of occurrence of percolation can be answered by looking at the excursions that are entirely contained in the complement of this ball within the disk. This heuristic idea is also consistent with an interpretation of the unit disk as a model for hyperbolic space, where, in fact, each Brownian excursion can be interpreted as an ``infinitely long loop''. From this point of view, the infinitely long part corresponds to that part of the excursion which is infinitesimally close to the boundary of the unit disk, which hence is the crucial part for the investigation of percolation.

We now turn to the proof that the discrete critical parameter $u_*^d(r)$ is equal to $\pi/3.$ The Brownian excursions cloud is the scaling limit of the random walk excursion cloud on $\D_n,$ see \cite{kozdron_scaling_2006,ArLuSe-20a}. In particular, one can compare the probabilities of connection events for the two models, see Lemma~\ref{lem:approxconnection}, which yields the equality $u_*^d(r)=u_*^c(r).$ We work with the KMT coupling of random walk with Brownian motion. The proof relies on counting the number of trajectories hitting a large ball, see \eqref{e.eqmeasexpr1} and Proposition~\ref{prop:localdescription}, and Beurling type-estimates, see Lemma~\ref{Lem:simpleRWresult}, to ensure that two random walk excursions which are asymptotically (as $n\rightarrow\infty$) arbitrarily close to each other eventually intersect. This will allow us to conclude that the random walk excursion cloud will eventually form an interface between $0$ and $\partial\D$ whenever the Brownian excursion cloud does.

Actually, the statement from \cite{ArLuSe-20a} about the scaling limit of random excursion cloud only lets us prove our discrete percolation result when considering the event that $0$ is connected to a ball of radius $1-\eps,$ for a fixed $\eps>0.$ It is natural to wonder if one can in fact connect $0$ to the boundary of the discrete lattice $\D_n$ in the supercritical regime for the vacant set of random walk excursions. We partially answer this question by proving that $0$ is connected to a point in the vacant set at a mesoscopic distance of order $n^{-1/7}$ to the boundary, see Remark~\ref{rk:connectiontoboundary} for details. This is achieved by improving the coupling between discrete and continuous excursions, see Theorem~\ref{the:couplingdiscontexc1}. 

Let us now turn to the proof of the inequality $h_*^d(r)\leq \sqrt{\pi/2}$ for the percolation of the level sets of the dGFF. As we have already mentioned, the dGFF can be coupled via an isomorphism theorem to excursion clouds with intensity related to the height of the field and loop soups with intensity $1/2,$ see Proposition~\ref{prop:BEisom}. In particular, the level sets of the dGFF are contained in the vacant set associated to the union of the loop clusters which intersect the excursion cloud. It was proved in \cite{werner-wu-cle} that the interfaces of this union are related to $\mathrm{SLE}_{4}(\rho)$ processes in a similar way as the interfaces of the excursion cloud were related to $\mathrm{SLE}_{8/3}(\rho).$ Using again explicit conditions on $\rho$ so that the $\mathrm{SLE}_{\kappa}(\rho)$ curve hit $\partial\D,$ but now for $\kappa=4,$ we deduce $h_*^d(r)\leq \sqrt{\pi/2}.$ Actually, the level $\sqrt{\pi/2}$ corresponds to the critical parameter for the vacant set associated to the union of the clusters of the level sets of the cable system GFF which intersect the boundary of $\D_n,$ see Remark~\ref{rk:exactcritparaGFF}. A similar result can be proved for the continuous GFF, see Corollary~\ref{cor:percocontGFF}. We also refer to Theorem~\ref{the:maindisexcloops} for an extension of the previous percolation result on excursion clouds plus loop soup when the intensity of the loop soup is between $0$ and $1/2$ (i.e.\ for $\kappa$ between $8/3$ and $4$).

Next, we comment on the proof of the strict inequality $h_*^d(r)>0.$  In \cite{DiWiWu-20}, it is proved that there is percolation for the level sets of the dGFF above level $0$ when the domain is a rectangle, in the sense that one can connect the left-hand side of this rectangle to its right hand-side with non-trivial probability. We adapt their technique to our context, that is when the domain is the unit disk $\D$, considering the event that $B(r)$ is connected to the boundary of the disk at a small, but positive, level for the dGFF. Note that contrary to usual Bernoulli percolation on $\Z^2,$  existence of a large component of $B(r)$ is not clearly equivalent to their left to right crossing event. Although very similar in spirit to \cite{DiWiWu-20}, our proof of $h_*^d(r)>0$ exhibits numerous technical differences, see for instance our definition of the exploration martingale below \eqref{eq:Mdef} or Lemma~\ref{Lem:bigcaponboundary}. Moreover, the strict inequality $h_*^d(r)>0$ implies a phase coexistence result for the level sets of the dGFF contrary to the percolation of the dGFF above level $0$ from \cite{DiWiWu-20}, see Remark~\ref{rk:mainthmgff},\ref{rk:phasecoexistence} for details.

Finally note that we expect that our results can be extended to other sufficiently regular Jordan domains besides the unit disk $\D$, for instance rectangles as studied in \cite{DiWiWu-20}. This is clearly the case for the Brownian excursion clouds by conformal invariance. For our results in the discrete setting some additional work would be required to prove convergence of discrete excursions to Brownian excursions in Section~\ref{s.coupling}, as well as the inequality $h_*^d(r)>0;$ see for instance Lemma~\ref{Lem:bigcaponboundary}. Since our main motivation for this article is the comparison with similar percolation results on $\Z^d,$ $d\geq3,$ which can be interpreted as finding large clusters connected to the boundary of the ball $B(n)$ as $n\rightarrow\infty,$ we have chosen to focus on the domain $\D$ here for the sake of exposition.

The rest of this article is organized as follows. In Section~\ref{s.prelimin} we fix notation, define the discrete objects studied in this paper (the dGFF, random walk excursion cloud and loop soup), and state the main theorems~in the discrete setting. Similarly, we define the objects and state our theorems~in the continuous setting in Section~\ref{s.contdef}, as well as recall some standard results from SLE theory.

Section~\ref{s.coupling} is dedicated to the construction of various couplings between the discrete and continuous excursion clouds, which are used in Section~\ref{sec:discretecritpara} to compute the critical value associated to the vacant set of the excursion cloud, plus loop soups, in the discrete setting.

Section~\ref{s:gfflvlperc} is centred around percolation for the Gaussian free field, and in particular the proof of the positivity of the associated critical parameter in the discrete setting. We also comment on the percolation for the Gaussian free field on the cable system.

Finally, Appendix~\ref{sec:contperco} proves our main result on the critical value associated to the vacant set of the excursion cloud, plus loop soups, in the continuum setting.  Appendix~\ref{sec:KMT} establishes some results on the KMT coupling (sometimes also referred to as the dyadic coupling) between a simple random walk and Brownian motion and Appendix~\ref{app:convcap} proves convergence of discrete capacities to continuous capacities for compact sets.

We use the following convention for constants. With $c,c
',c''$  we denote strictly positive and finite constants that do not depend on anything unless explicitly stated otherwise. Their values might change from place to place. Numbered constants  $c_1$ and $c_2$ follow a similar convention, except that they are fixed throughout the paper.

\subsection*{Acknowledgements}
The research of AD has been supported by the Deutsche Forschungsgemeinschaft (DFG) grants DR 1096/1-1 and DR 1096/2-1. OE is supported by the ERC Consolidator Grant 772466 ``NOISE''. AP was supported by the Engineering and Physical Sciences Research Council (EPSRC) grant EP/R022615/1, Isaac Newton Trust (INT) grant G101121, European Research Council (ERC) starting grant 804166 (SPRS) and the Swiss NSF. JT is supported by the Swedish research council. FV was supported by the Knut and Alice Wallenberg Foundation, the Swedish research council, and the Ruth and Nils-Erik Stenb\"ack Foundation. We are grateful to Juhan Aru and Titus Lupu for helpful discussions.

\section{Definitions and results for the discrete models}\label{s.prelimin}

\subsection{Notation} \label{sec:notation}

Throughout this paper we let $\D \subset \BC$ denote the open unit disk, endowed with the Euclidean distance $d(\cdot,\cdot).$ We denote by $\partial\D$ its boundary and by $\overline{\D}$ its closure, the closed unit disk. We define $B(x,r):=\{y\in{\D}:\,d(x,y)< r\}$ for all $x\in{\D}$ and $r\in{(0,1)},$ $\overline{B}(x,r):=\{y\in{\D}:\,d(x,y)\leq r\}$ and for each $K\subset\D,$ define $B(K,r):=\{x\in{\D}:\,\inf_{y\in{K}}d(x,y)<r\}.$ We abbreviate $B(r):=B(0,r)$ and $\overline{B}(r):=\overline{B}(0,r),$  and also take the convention $B(0):=\{0\}.$ We also define the distance $d(A,B)$ between two sets $A$ and $B$ as the smallest Euclidean distance between points in $A$ and points in $B.$  We write $K \Subset \D$ to indicate that $K$ is a compact subset of $\D,$ and let let $\I_A$ denote the indicator function of a set $A$.

Let $\D_n$ be the graph with vertex set $n^{-1} \Z^2 \cap \D$ and edge set given by the nearest-neighbor edges in $n^{-1}\Z^2$ that are included in $\D.$ For $K \subset \D,$ we write
	\begin{equation} \label{eq:setDiscreteApprox}
	K_n := \D_n \cap K.
	\end{equation}
	 Furthermore, denote by $\partial{\D}_n$ the neighboring vertices of $\D_n$  in $n^{-1}\Z^2 \setminus \D_n,$ and let $\overline{\D}_n:=\D_n\cup\partial\D_n.$ Let $\widetilde \D_n$ be the cable system associated to the graph $ \D_n$ with unit weights and infinite killing on $\partial\D_n,$ also sometimes called the metric graph. It is constructed by gluing together intervals $I_{\{x,y\}}$ of length $1/2$ (independent of $n$) through their endpoints for all $x\sim y\in{\overline \D_n},$ ($x\sim y$ means that $x$ and $y$ are neighbors in $n^{-1}\Z^2$) where $I_{\{x,y\}}$ is open at $x$ if $x\in{\partial\D_n}$ and closed at $x$ if $x\in{\D_n},$ and similarly at $y.$  We refer to \cite[Section~2]{Lu-14} for a more general setting. For each $x\in{\D_n}$ and $r\in{[0,1)},$ we define  $B_n(x,r):=(B(x,r))\cap \D_n,$ abbreviate $B_n(r):=B_n(0,r),$ let $\overline{B}_n(r):=\overline{B}(r)\cap\D_n,$ and $\tilde{B}_n(r)$ be the union of $I_e$ over all edges $e$ between two vertices of $B_n(r),$ and let $B_n(K,r):=B(K,r)\cap\D_n$ for each $K\subset\D_n.$   For $A,B \subset \D$ and $F\subset\overline{\D}$ we write
\begin{equation} \label{eq:discreteConnect}
A \stackrel{F}{\longleftrightarrow}B,
\end{equation}
or sometimes $A\leftrightarrow B$ in $F,$ to express one of the following facts: if $F\cap(\D\setminus{\D}_n)\neq\varnothing,$ then there is a continuous path $\pi\subset F$ starting in $A$ and ending in $B;$ otherwise, it means that there is a nearest neighbor path $(x_1, x_2, \ldots, x_m)$ in $F=F_n$ with $x_1 \in A_n$ and $x_m \in B_n.$ In other words, if $F$ is a non-discrete set $\longleftrightarrow$ means continuous connections, whereas if $F$ is discrete it means discrete connections, and we see subsets of the cable system as non-discrete sets. We write $A\stackrel{F}{\centernot \longleftrightarrow}B$ for the complement of the event \eqref{eq:discreteConnect}. Moreover, we say that a set $F\subset \D$ is connected if $x\leftrightarrow y$ in $F$ for all $x,y\in{F}.$ Note that the notations $\longleftrightarrow$ and $\centernot \longleftrightarrow,$ as well as the notion of connectedness, can depend on the choice of $n,$ which should always either be clear from context, or not depend of the choice of $n$ (if $F\cap(\D\setminus\D_n)\neq\varnothing$ for all $n\in\N$).  

We let $(X_k^{(n)})_{k \geq 0}$ denote the simple random walk on $n^{-1} \Z^2$ killed upon hitting the outer boundary $\partial \D_n$, which we abbreviate by $X:=(X_k)_{k\ge 0}$ whenever the dependency on $n$ is clear. 
We denote by $\Pm_x^{(n)}$ the law of  $(X_k^{(n)})_{k\geq0}$ on $\D_{n}$ starting at $x\in{\D_n}.$
For $K \subset \D_n,$ denote for $x,y \in \D_n$ the Green function 
\begin{equation} \label{eq:killedGreen}
G_{K}^{(n)}(x,y) := \frac14\E_x^{(n)} \Big[\sum_{k = 0 }^{\tau_{\partial K\cup\partial\D_n} } \I \{ X_k^{(n)} =y \} \Big]
\end{equation}
killed outside $K.$  We abbreviate $G^{(n)}:=G^{(n)}_{\D_n}.$ We define the hitting and return time of $X^{(n)}$ for a set $K \subset \overline{\D}_n$ as
\[
\tau_K^{(n)} := \inf \left\{ k \geq 0:\, X_k^{(n)} \in K\right\}\text{ and }\widetilde{\tau}_{K}^{(n)}:=\inf\left\{k\geq1:\,X_k^{(n)}\in{K}\right\}, 
\]
with the convention $\inf\varnothing=+\infty,$ and the last exit time is defined by 
\begin{equation}
\label{eq:deftauLn}
L_K^{(n)} := \sup \left\{k \geq 0 :\, X_k^{(n)} \in K\right\},
\end{equation}
with the convention $\sup\varnothing=-\infty.$  We define the equilibrium measure and capacity of a set $K\subset \D_{n}$ as follows
\begin{align}
\begin{split}
    \label{defequiandcap}
    e_{K}^{(n)}(x)&:=4\Pm_x^{(n)}(\tau_{\partial \D_n}^{(n)}<\widetilde{\tau}_{K}^{(n)}), \text{ for all }x \in K,\text{ and } \\
    \mathrm{cap}^{(n)}(K)&:=\sum_{x\in{K}}e^{(n)}_{K}(x).
    \end{split}
\end{align}
We also write $\overline{e}_K^{(n)}:=e_K^{(n)}/\mathrm{cap}^{(n)}(K)$ for the normalized equilibrium measure of $K.$ In view of \cite[(1.56)]{MR2932978} we have
\begin{equation}
\label{eq:lastexitdis}
\Pm_x^{(n)} \big(L_K^{(n)}>0,X_{L_K^{(n)}}=y\big) = G^{(n)}(x,y) e_{K}^{(n)}(y)\text{ for all }K\subset \D_n,\ x\in{\D_n}\text{ and }y\in{K}.
\end{equation}
Denote for $K_n \subset \D_n$ its ``discrete outer"\ part visible from $\partial \D_n $ by  
\begin{equation} \label{eq:outer}
\partial K_n:=\big \{ x \in \overline \D_n \setminus K_n \, : \, \exists\, y \in K_n \text{ with } x \sim y, \, \Pm_x(\tau_{\partial \D_n } < \tau_{K_n}) > 0 \big \}
\end{equation}
as well as ``the support of its equilibrium measure as seen from $\partial \D_n$"\ by 
\begin{equation}\label{eq:interior}
\widehat{\partial} K_n:=\big\{ x \in K_n \, : e_{K_n}^{(n)}(x)>0\big \}.
\end{equation}
Note that the notation $\partial\D_n$ introduced below \eqref{eq:setDiscreteApprox} and in \eqref{eq:outer} are consistent, and that $\widehat{\partial}\D_n$ are just the neighbors in $\D_n$ of $\partial\D_n.$ On the other hand, if $K\subset \D$ is such that $K\cap(\D\setminus\D_n)\neq\varnothing$ for all $n\in{\N},$ we denote by $\partial K$ the topological boundary of $K.$

\subsection{Discrete Gaussian Free Field} \label{sec:dGFF}
The discrete Gaussian free field (dGFF) on $\D_n$ with zero boundary condition is the centred Gaussian vector $(\varphi_x)_{x \in \D_n} $ with covariance given by 
\begin{equation}
\label{eq:defgff}
\E \left[\varphi_x\varphi_y\right]= G^{(n)}(x,y)\text{ for all }x,y\in{\D_n},
\end{equation}
cf.\ \eqref{eq:killedGreen} for notation. We write 
\begin{equation}
\label{eq:deflevelsets}
E^{\ge h}_n:=  \left\{ x \in \D_n : \varphi_x \geq h \right\}
\end{equation}
for the excursion set above level $h$ of $\varphi$ in $\D_n.$

\subsection{Discrete excursion process and loop soup}\label{s.de}
The discrete excursion measure is a measure which is supported on nearest neighbor paths that start and end in $\partial \D_n,$ and otherwise are contained in $\D_n.$ Let 
\[
W^{(n)} := \left\{ e \in \bigcup_{m \geq1} \overline{\D}_n^{m+1} : [e] \subset \D_n, e(0), e(m)\in \partial \D_n ,\,e(i)\sim e(i+1)\,\forall i<m \right\}
\] 
denote the set of discrete excursions in $\D_n,$ where $[e]:=\{e(1),\dots,e(m-1)\}$ denotes the trace of the discrete excursion on $\D_n.$  Given an excursion $e=(e_0,e_1,\ldots,e_m)$ we let $t_e=m$ denote its lifetime.

Given $K \subset \D_n,$ let $W_K^{(n)} := \left\{ e \in W^{(n)}: [e] \cap K \neq \varnothing \right\}$ denote all excursions that intersect $K$. Moreover, we denote by ${\Pm}^{(n)}_x$ the law under which $X$ is  simple random walk on $\frac1n\Z^2,$ starting at $x\in{\overline{\D}_n},$ and killed after the first return time $\widetilde{\tau}_{\partial \D_n}$ to $\partial\D_n$ (note that $\widetilde{\tau}_{\partial \D_n}=\tau_{\partial\D_n}$ if $x\in{\D_n}$). For $x,y \in \overline{\D}_n$  we let  
\begin{equation}
    \label{eq:defdisPoisson}
    \CH_n (x,y) := 4{\Pm}_x^{(n)}( X_1 \in \D_n, X_{\widetilde{\tau}_{\partial \D_n}} =y)
\end{equation}
denote the discrete boundary Poisson kernel and note that $\CH$ is symmetric. We let 
\[
{\Pm}_{x,y}^{(n)} \left(\,  \cdot \, \right) := {\Pm}_x^{(n)} \left(  \cdot \, |\,  X_1 \in \D_n, X_{\widetilde{\tau}_{\partial \D_n}} =y \right)
\]
denote the law of the random walk excursion from $x$ to $y.$ The discrete excursion measure is then defined by 
\begin{equation}
\label{eq:defmun}
    \mu^{(n)}_{\rm exc} := \sum_{x,y \in \partial \D_n} \CH_n(x,y) {\Pm}^{(n)}_{x,y}.
\end{equation}
Alternatively, one can write the measure $\mu^{(n)}_{\rm exc}$ as follows
\begin{equation*}
    \mu^{(n)}_{\rm exc} = \sum_{x \in \partial \D_n}4{\Pm}_x^{(n)}(\, \cdot \, , X_1\in{\D_n}).
\end{equation*}
 Writing $\pi^{(n)}$ for the map that sends
\(  (e_0,e_1,...,e_m) \in W^{(n)} \) 
to $(e_1,...,e_{m-1})$ in the space of nearest-neighbor trajectories on $\D_n$ starting and ending in $\widehat \partial \D_n$ (and hence
forgets about the first and last visited vertex),  we get
\begin{equation}
\label{eq:muexcasinter}
    \mu_{{\rm exc}}^{(n)} \circ (\pi^{(n)})^{-1}
    =\sum_{x\in{\widehat{\partial} \D_n}} \kappa_x \Pm_x^{(\kappa,n)},
\end{equation}
where $\kappa_x:=|\{y\in{\partial \D_n }:\,y\sim x\}|,$ and $\Pm_x^{(\kappa,n)}$ denotes the law of the random walk on the weighted graph $(\D_n,1,\kappa),$ that is $\D_n$ with unit conductance and killing measure $\kappa.$  In view of \cite[Theorem~3.2]{Pre1} (with $F=\D_n$), see also \cite[(2.12)]{MR2892408}, modulo time shift, the measure in \eqref{eq:muexcasinter} corresponds to the interlacements measure for the weighted graph $(\D_n,1,\kappa).$

The corresponding discrete excursion process is a Poisson point process $\omega^{(n)}$ on $W^{(n)}\times\R^+$ with intensity measure $\mu^{(n)}_{\rm exc}\otimes\mathrm{d}\intens,$ under some probability which we will denote by $\Pm.$ We write $\omega^{(n)}_{\intens}$ for the point process which consists of the trajectories in $\omega^{(n)}$ with label at most $\intens,$  and the associated occupied and vacant set are given by 

\begin{equation}
\label{eq:VunDef}
\CI^{\intens}_n := \bigcup_{e \in \supp(\omega^{(n)}_{\intens})} [e]\text{ and }\CV^{\intens}_n := \D_n \setminus \CI^{\intens}_n.
\end{equation}

In view of \eqref{eq:muexcasinter} and below, one can describe the restriction of $\omega_{\intens}^{(n)}$ to the trajectories hitting  a compact $K\subset \D_n$ as follows. 

\begin{proposition}
\label{prop:interoncompacts}
For $K\subset \D_n,$ let $N _K^{(n)}\sim\text{Poi}(\intens \, \mathrm{cap}^{(n)}(K)).$ Conditionally on $N_K^{(n)},$ we let $(X^{(n),i})_{i=1}^{N_K^{(n)}}$ be a collection of independent random walks on $\D_n$ with law $\Pm_{\overline{e}_K^{(n)}}^{(n)}.$ Then $\sum_{i=1}^{N_K^{(n)}}\delta_{X^{(n),i}}$ has the same law as the point process of forward trajectories in $\omega_{\intens}^{(n)}$ hitting $K,$ started at their first hitting time of $K.$ 
\end{proposition}

\begin{remark}
\phantomsection\label{rk:defexcursion}
\begin{enumerate}[label=\arabic*)]
\item \label{rk:excursionviaoneRW}
One can also describe directly the law of all the excursions in $\omega_{\intens}^{(n)}$ with a single random walk. Using network equivalence, one can collapse $\partial \D_n$ into a single vertex $x_n$  and we denote by $(Y_t)_{t\geq0}$ under $\overline{\Pm}^{(n)}_{x_n}$ the random walk on $\D_n \cup\{x_n\}$ starting from $x_n,$ which jumps along any edge between $x_n$ and  $x\in \widehat \partial \D_n$ at rate $1,$ and along any edge between $x\in{\D_n}$ and $y\in{\D_n},$ $y\sim x,$ at rate $1.$ Let $\tau_{\intens}:=\inf\{t\geq0:\,\ell_{x_n}(t)\geq \intens\},$ where $\ell_{x_n}(t)$ denotes the total time spent by $Y$ in $x_n$ at time $t.$ Then the point process consisting of the excursions in $\D_n$ of $(Y_t)_{t\leq\tau_{\intens}}$ under $\overline{\Pm}^{(n)}_{x_n}$ has the same law as $\omega_{\intens}^{(n)}.$ We refer for instance to \cite[(2.8)]{MR2892408} for a proof.
\item \label{rk:choiceofweight}We now comment on the reason for the choosing to include the multiple of $4$ in the definition \eqref{eq:defdisPoisson} of the discrete boundary Poisson kernel or the equilibrium measure \eqref{defequiandcap}. The reason is to ensure the convergence of the discrete excursion process to the continuous excursion process at the same level, see Theorem~\ref{the:couplingdiscontexc1}. Similarly, the factor $1/4$ in the definition \eqref{eq:killedGreen} of the discrete Green function is to ensure convergence to the continuous Green function, see \eqref{e.greenest}. One can interpret this choice of constants as considering $\D_n$ as a weighted graph with weight one between two neighbors (and thus total weight $4$ at each vertex) and infinite killing on $\partial\D_n,$ or equivalently the graph $(\D_n,1,\kappa)$ from below \eqref{eq:muexcasinter}. Similarly, the length of the cables $I_e$ on the cable system, see Section~\ref{sec:notation}, is chosen to be $1/2$ since it corresponds to our choice of unit conductances, see for instance \cite{Lu-14}.
\end{enumerate}
\end{remark}
In a similar fashion one can define the random walk loop soup. We say that  $\ell =(\ell_0, \ell_1,...,\ell_k)\in \D_n^{k+1}$ is a rooted loop if it is a nearest neighbor path such that $\ell_0=\ell_k,$ whereas an unrooted loop is a rooted loop modulo time shift, that is an equivalence class of rooted loops where two loops $\ell$ and $\ell'$ are equivalent if $\ell = \ell'(\cdot+t)$.

The (rooted) loop measure, see for instance \cite{LJ-11}, on $\D_n$ is defined by 
\begin{equation}
	\label{eq:defdisloop}
	\nu_{{\rm loop}}^{r,(n)}(\cdot) := \sum_{x \in \D_n} \int_0^\infty \Pm_{x,x}^{t,(n)}(\, \cdot \,, {\tau}_{\partial \D_n} > t) p_t^{(n)}(x,x) \frac{1}{t}\, {\rm d}t, 
\end{equation}
where $p_t^{(n)}(x,y)$ denotes the family of transition probabilities for the continuous time random walk induced by the (constant $1$) conductances on $n^{-1}\Z^2$ and $\Pm_{x,y}^{t,(n)}$ is the corresponding bridge probability measure. The unrooted loop measure $\nu_{{\rm loop}}^{(n)}$ is given by the image of $\nu_{{\rm loop}}^{r,(n)}$ via the canonical projection on unrooted loops modulo time shift. The random walk loop soup on $\D_n$ with parameter $\lambda>0$ is defined as a Poisson point process with intensity $\lambda\nu_{{\rm loop}}^{(n)}.$

\subsection{Cable system} \label{sec:cableSystemExc}

Recall the definition of the cable system $\widetilde{\D}_n$ below \eqref{eq:setDiscreteApprox}. Under some probability $\tilde{\Pm}^{(n)}_x,$ $x\in{\widetilde \D_n},$ we denote by $\widetilde{X}^{(n)}$ the canonical diffusion on $\widetilde{\D}_n$ starting in $x,$ which behaves like a standard Brownian motion inside $I_{\{y,z\}},$ $y\sim z\in{\overline{\D}_n},$ killed when reaching the open end of $I_{\{y,z\}}$ if either $y\in{\partial\D_n}$ or $z\in{\partial\D_n},$ and such that the discrete time process which corresponds to the successive visits of $\widetilde{X}^{(n)}$ in $\D_n$ is the random walk $(X_k^{(n)})_{k\in{\N}}$ on $\D_n$ under $\Pm_x^{(n)}.$ We refer to \cite[Section~2]{Lu-14} for details. Let us remark that we chose to work with segments $I_{\{x,y\}}$ of length $1/2$ and standard Brownian motion, instead of segments $I_{\{x,y\}}$ of length $1$ and non-standard Brownian motion as in \cite{ArLuSe-20a}. This simply corresponds to a time change of the Brownian motion, and since our results will be independent of this time change we will from now on use the results of \cite{ArLuSe-20a} without paying attention to this different convention. 

For the cable graph $\widetilde{\D}_n$ we denote the analogously defined quantities by putting a tilde $\widetilde{\cdot}$ on them.  In particular, the Gaussian free field (GFF) on the cable system, denoted by $(\tilde{\varphi}_x)_{x\in{\widetilde{\D}_n}}$, is defined as in \eqref{eq:defgff} but for all $x,y\in{\widetilde{\D}_n},$ where $G^{(n)}$ from \eqref{eq:killedGreen} can be extended consistently to the cable system $\widetilde{\D}_n$ as the Green function associated with the diffusion $\tilde{X}.$ There is a simpler construction of the GFF $\tilde{\varphi}$ on the cable system $\widetilde{\D}_n$ from the dGFF on $\D_n$: conditionally on the dGFF $\varphi$ on $\D_n,$ $\tilde{\varphi}_{|I_e}$ for each edge $e=\{x,y\}$ can be obtained by running a (conditionally) independent length-$1/2$ Brownian bridge (for a non-standard Brownian bridge obtained from a Brownian motion of variance $2$ at time $1$) on the interval $I_e$ starting from $\varphi_x$ and ending at $\varphi_y$. In particular, $\tilde{\varphi}_x=\varphi_x$ for all $x\in{\D_n},$ and, denoting by $\widetilde{E}^{\ge h}_n:=\{x\in{\widetilde{\D}_n}:\,\varphi_x\geq h\}$ the excursion set in the cable system at level $h\in\R,$ we have ${E}^{\ge h}_n=\widetilde{E}^{\ge h}_n\cap\D_n.$ 

The intensity $\tilde{\mu}_{\text{exc}}^{(n)}$ of the excursion process on the cable system can be defined by extending the definition \eqref{eq:defmun} to the cable system, see \cite[Section~2.2]{ArLuSe-20a} for $u(x)=\sqrt{2}$ for details (replacing non-standard Brownian motions by standard ones). Under some probability measure $\tilde{\Pm}^{(n)},$ for each $u>0$ we then define the excursion process $\tilde{\omega}_u^{(n)}$ as a Poisson point process with intensity measure $u\tilde{\mu}_{\text{exc}}^{(n)}.$ We furthermore let $\tilde{\be}_n^u\subset\widetilde{\D}_n$ be the set of points visited by a trajectory in $\tilde{\omega}_u^{(n)},$ and denote by $\tilde{\V}_n^u:=(\tilde{\be}_n^u)^\ch$ its complement in $\widetilde{\D}_n.$ Alternatively, one could equivalently define $\tilde{\omega}_u^{(n)}$ as the random interlacements process at level $u$ on the cable system of the weighted graph defined below \eqref{eq:muexcasinter}, see  \cite[Section~3]{Pre1}; or, yet another equivalent way is to define it as the excursions on $\widetilde{\D}_n$ of the diffusion starting from $x_n,$ and run until spending total time $u$ in $x_n,$ for the cable system associated to the graph $\D_n\cup\{x_n\}$ from Remark~\ref{rk:defexcursion}, \ref{rk:excursionviaoneRW}. In particular, extending the definition of the equilibrium measure $e_K^{(n)}$ and capacity $\mathrm{cap}^{(n)}(K)$ of compacts $K\subset\widetilde{\D}$ similarly as in \cite[(2.16) and (2.19)]{DrePreRod3}, one can extend Proposition~\ref{prop:interoncompacts} to compacts $K\subset\widetilde{\D}_n,$ see \cite[Theorem~3.2]{Pre1}. Similarly as for the diffusion $\tilde{X}^{(n)},$ one can obtain $\tilde{\omega}_u^{(n)}$ by adding standard Brownian motion excursions on the cables $I_{\{x,y\}},$ $x\sim y$ to the excursions in $\omega_u^{(n)},$ and we will from now on always assume that $\tilde{\omega}_u^{(n)}$ and $\omega_u^{(n)}$ are coupled in that way. In particular $\be^u_n=\tilde{\be}^u_n\cap\D_n$ and $\V_u^n=\tilde{\V}_u^n\cap\D_n.$

We finally introduce the loop intensity measure $\tilde{\nu}_{\rm{loop}}$   on loops in the cable system $\widetilde{\D}_n,$ and refer to \cite{Lu-14} for a rigorous definition. The projection of this measure on the discrete part of the loops which hit $\D_n$ simply corresponds to the measure introduced in \eqref{eq:defdisloop}. The cable system loop soup with parameter $\lambda>0$ is a Poisson point process with intensity $\lambda\tilde{\nu}_{\rm{loop}}.$ We collect these loops in connected components which we will refer to as the loop clusters on the cable system. We always assume that under the probability $\Pm,$ the loop soup and the excursion process are independent. For each $u,\lambda\ge 0,$ we denote by $\tilde{\be}_n^{u,\lambda}$ the closure of the union of all the loop clusters on the cable system of the loop soup at intensity $\lambda,$ which intersect $\tilde{\be}_n^u.$ It is then consistent to set
\begin{equation} \label{eq:BuuVuuDef}
\tilde{\be}_n^{u,0}:=\tilde{\be}_n^u \quad \text{ as well as } \quad  \tilde{\V}_n^{u,\lambda}:=\widetilde{\D}_n\setminus\tilde{\be}_n^{u,\lambda}\text{ for all } \lambda\geq0.
\end{equation}

\subsection{Isomorphism theorem}
\label{sec:iso}

A key tool in our investigations is provided by isomorphism theorems~on the cable system relating a Gaussian free field on the one hand, and independent Brownian excursions as well as an independent Gaussian free field on the other hand. Such results have a long history, dating back to Dynkin's isomorphism theorem~\cite{MR693227,MR734803}, which found its motivation in earlier work by Brydges, Fröhlich and Spencer \cite{BrFrSp-82} and Symanzik \cite{Sy-69}.  More recent developments include the second Ray-Knight isomorphism for random walks \cite{MR1813843} or the isomorphism between random interlacements and the Gaussian free field \cite{MR2892408}. For our purposes, we will only need a relatively soft result, which is a direct consequence of this isomorphism for the excursion sets of the GFF as it can be found in \cite{ArLuSe-20a}.

\begin{proposition}
\label{prop:BEisom}
For each $\intens>0$ and $n\in\N,$ there exists a coupling between $\widetilde{E}^{\geq \sqrt{2\intens}}_n$ and $\tilde{\V}^{\intens,1/2}_n$ such that almost surely,
\begin{equation}
\label{eq:iso}
    \widetilde{E}^{\geq \sqrt{2\intens}}_n\subset\tilde{\V}^{\intens,1/2}_n.
\end{equation}
\end{proposition} 
\begin{proof}
We use \cite[Proposition~2.4]{ArLuSe-20a} applied with the constant boundary condition equal to $\sqrt{2u},$ and note that the excursion process with boundary condition $\sqrt{2u}$ defined therein has the same law as our excursion process at level $u$ (compare the normalization in \cite[(2.3)]{ArLuSe-20a} to the one in \eqref{eq:defmun}). On $\tilde{\be}_n^{u,1/2},$ the sign $\sigma$ from  \cite[Proposition~2.4]{ArLuSe-20a} is equal to $1,$ and thus $\tilde{\be}_n^{u,1/2}$ is stochastically dominated by $\widetilde{E}^{\geq-\sqrt{2u}}_n.$ We can conclude by taking complement and by symmetry of the GFF.
\end{proof}

Note that to prove Proposition~\ref{prop:BEisom} one actually does not need the full strength of \cite[Proposition~2.4]{ArLuSe-20a}, but only a cable system version of  \cite[Proposition~2.3]{ArLuSe-20a}, by proceeding similarly as in the proof of  \cite[Theorem~3]{Lu-14}. We will assume from now that the field, the loop soup and the excursion process on the cable system are coupled under $\Pm$ so  that \eqref{eq:iso} holds. Note that \eqref{eq:iso} implies that ${E}_n^{\geq\sqrt{2u}}\subset\tilde{\V}_n^{u,1/2}\cap\D_n\subset\V^u_n.$ One immediately deduces the weak inequality $h_*^d(r)\leq\sqrt{2\intens_*^d(r)}$ for all $r\in{(0,1)},$ see \eqref{def:u*} and \eqref{def:h*}, which will be strengthened as a consequence of our results, see \eqref{eq:strictinequalitycritpara}. 

\begin{remark}
\label{rk:otheriso}
As we explained below \eqref{eq:muexcasinter}, see also Proposition~\ref{prop:interoncompacts}, the excursion process $\omega_u^{(n)}$ can be seen as a random interlacements process on $\D_n$ with infinite killing on $\partial\D_n.$ An isomorphism theorem~between the GFF and random interlacements was first proved in \cite{MR2892408} on discrete graphs, and extended in \cite{Lu-14} to the cable system. Even if these references only consider random interlacements on graphs with zero killing measure, their results can easily be extended to graphs with any killing measure, see \cite[Remark~2.2]{DrePreRod3}. Combining this isomorphism between the GFF and random interlacements with the isomorphism between loop soups and random interlacements from \cite{LJ-11,Lu-14}, one easily obtains an alternative proof of the inclusion \eqref{eq:iso}. Using Remark~\ref{rk:defexcursion}, \ref{rk:excursionviaoneRW}, one could alternatively see the inclusion \eqref{eq:iso} as a consequence of the second Ray-Knight Theorem~between Markov processes and the GFF from \cite{MR1813843}. The stronger version of the isomorphism found in  \cite[Proposition~2.4]{ArLuSe-20a} can also be seen as an isomorphism for random interlacements, see \cite[Theorem~2.4]{MR3492939} or \cite[Theorem~1.1, 2)]{DrePreRod3}, or for Markov processes, see the proof of \cite[Theorem~8]{MR3978220}.
\end{remark}

\subsection{Discrete results}

Recall the notation  \eqref{eq:discreteConnect}. For each $r\in{[0,1)}$ we define
\begin{equation} \label{def:u*}
	    \intens_*^d(r):=\inf\Big\{\intens\geq0:\,\lim\limits_{\eps\rightarrow0}\liminf_{n\rightarrow\infty}\Pm\Big(B_n(r)\stackrel{\mathcal{V}^{\intens}_n}{\longleftrightarrow}
		\partial B_n(1-\eps)\Big)=0\Big\}.
\end{equation}
Our first main result, proved at the end of Section~\ref{sec:discretecritpara}, identifies the parameter $\intens_*^d(r),$ and actually prove that in the supercritical phase one can have connection at polyonomial distance from the boundary, see \eqref{eq:percowithoutloops}. 

\begin{theorem}
\label{the:maindisexc}
For each $r\in{[0,1)}$ and $u\geq\pi/3$
\begin{equation}
\label{eq:percowithoutloopssub}
    \lim\limits_{\eps\rightarrow0}\lim\limits_{n\rightarrow\infty}\Pm\Big(B_n(r)\stackrel{{\mathcal{V}}^{\intens}_n}{\longleftrightarrow}
		\partial B_n(1-\eps)\Big)=0,
\end{equation}
whereas for all $u<\pi/3$ and $r\in{[0,1)}$
\begin{equation}
\label{eq:percowithoutloops}
        \liminf_{n\rightarrow\infty}\Pm\big(B_n(r)\stackrel{{\mathcal{V}}_n^{\intens}}{\longleftrightarrow}\partial B_n(1-n^{-1/7})\big)>0.
\end{equation}
In particular, $u_*^d(r)=\pi/3.$
\end{theorem}

We now present an extension of Theorem~\ref{the:maindisexc} to the setting when there is an additional loop soup on top of  the discrete excursion process. Recall the notation $\tilde{\V}_n^{u,\lambda}$ introduced in \eqref{eq:BuuVuuDef}, and define the corresponding parameter which equals $1/2$ times the central charge appearing in conformal field theory:
\begin{equation}
\label{eq:defckappa}
    \lambda(\kappa) := \frac{(8-3\kappa)(\kappa-6)}{4\kappa}\text{ for all }\kappa\in{[8/3,4]}.
\end{equation}
Note that $\tilde{\mathcal{V}}_n^{u,\lambda(8/3)}\cap\D_n={\mathcal{V}}_n^u,$ and that the connection probabilities for $\tilde{\mathcal{V}}_n^{u,\lambda(8/3)}$ are the same as for $\V_n^u,$ since for each $x\sim y\in{\D_n},$ $\{x,y\}\subset{\V_n^u}$ if and only if $I_{\{x,y\}}\subset \tilde{\mathcal{V}}_n^{u,\lambda(8/3)}.$  Our second main result concerns the percolation of $\tilde{\mathcal{V}}_n^{u,\lambda(\kappa)},$ and is also proved at the end of Section~\ref{sec:discretecritpara}.

\begin{theorem}
\label{the:maindisexcloops}
For each $r\in{[0,1)}$ and $\kappa\in{[8/3,4]},$
\begin{equation}
\label{eq:percowithloops}
    \lim\limits_{\eps\rightarrow0}\liminf_{n\rightarrow\infty}\Pm\Big(B_n(r)\stackrel{\tilde{\mathcal{V}}^{\intens,\lambda(\kappa)}_n}{\longleftrightarrow}
		\partial B_n(1-\eps)\Big)=0\text{ if and only if }u\geq\frac{(8-\kappa)\pi}{16}.
\end{equation}
\end{theorem}
Note that the statement \eqref{eq:percowithoutloops} for $\kappa=8/3$ is stronger in the supercritical phase than \eqref{eq:percowithloops} for other values of $\kappa,$ and we refer to Remark~\ref{rk:polycloseboundary} for more details on this.

We now turn to our results for the percolation of the level sets of the dGFF. For $r\in{[0,1)}$ we define the critical level
\begin{equation} \label{def:h*}
h_*^d (r):= \inf\Big\{ h \in \R \, : \, \lim\limits_{\eps\rightarrow0}\liminf_{n \to \infty} \Pm \big (B_n(r) \stackrel{E^{\ge h}_n}{\longleftrightarrow}\partial B_n(1-\eps) \big) = 0 \Big\}.
\end{equation}
It follows from \eqref{eq:percowithloops} with $\kappa=4$ that there is no percolation in $\tilde{\mathcal{V}}^{\intens,1/2}_n$ for all $\intens\geq {\pi/4}.$ One can directly deduce from the isomorphism \eqref{eq:iso} that $\widetilde{E}^{\ge h}_n$ also does not percolate for all $h\geq\sqrt{\pi/2}.$ We refer to Remark~\ref{rk:exactcritparaGFF} as to why the parameter $\sqrt{\pi/2}$ can also be seen as a critical percolation parameter for the dGFF.  Combining this observation with the methods from \cite{DiWiWu-20} to prove percolation of the dGFF at small positive levels, we obtain the following theorem, proved in Section~\ref{s:gfflvlperc}.

\begin{theorem}
\label{the:maindisgff}
For each $r\in{(0,1)}$ we have
\begin{equation} \label{eq:h*ineq}
    0<h_*^d(r)\leq\sqrt{\frac{\pi}2}.
\end{equation}
Moreover, there exists $h>0$ such that
\begin{equation}
\label{eq:GFFconnectionBds}
    0 < \liminf_{n \to \infty} \Pm \Big (B_n(r) \stackrel{E^{\ge h}_n}{\longleftrightarrow} \widehat{\partial}\D_n \Big) \le \limsup_{n \to \infty} \Pm \Big (B_n(r) \stackrel{E^{\ge h}_n}{\longleftrightarrow}\widehat{\partial} \D_n \Big) < 1.
\end{equation}
\end{theorem}
In view of \eqref{eq:GFFconnectionBds}, one could also prove positivity of the critical parameter for the dGFF when replacing $1-\eps$ by $1$ in \eqref{def:h*}, and removing the limit as $\eps\rightarrow0.$ We refer to Remark~\ref{rk:connectiontoboundary} for more details on this alternative choice of the definition of the critical parameter. 

\begin{remark}
\phantomsection\label{rk:mainthmgff}
\begin{enumerate}[label=\arabic*)]
\item This is drastically different from Bernoulli percolation where the density at criticality for bond percolation is $1/2,$ and is strictly larger than $1/2$ for site percolation (see \cite{Ke-80} for the former, which in combination with \cite[Theorem~3.28]{grimmett1999percolation} yields the latter). On the other hand, in our setting of site percolation for the dGFF level sets the density at criticality is strictly smaller than $1/2$ by Theorem~\ref{the:maindisgff}; our results thus support the heuristics that 'positive correlations help percolation'. 
    
\item \label{rk:phasecoexistence} Theorem~2.7 entails a phase coexistence result for the components of the level sets of the Gaussian free field connected to $\partial\D.$ Indeed fix some $h\in{(0,h_*^d(1/2))},$ then by monotonicity, for each $r\in{[1/2,1)},$ the limit as $n\rightarrow\infty$ of the probability that $B_n(r)$ is connected to $\widehat{\partial}\D_n$ in $E^{\ge h}_n$ is larger than $c,$ for some positive constant $c$ not depending on $r.$ Moreover, using the isomorphism~\eqref{eq:iso}, the symmetry of the dGFF and the fact that the limit as $n\rightarrow\infty$ of the probability that the random walk excursion cloud hits $B_n(r)$ converges to $1$ as $r\rightarrow1,$ see \eqref{capball} and  Lemma~\ref{l.capconv}, one can easily show that there exists $r<1$ so that the probability that $B_n(r)$ is connected to $\widehat{\partial}\D_n$ in $\{x\in{\D_n}:\varphi_x<h\}$ is larger than $1-c/2.$ Therefore, for this choice of $r,$ there is with positive probability as $n\rightarrow\infty$ coexistence of a path in  $\{x\in{\D_n}:\varphi_x<h\}$ and a path in $\{x\in{\D_n}:\varphi_x\geq h\},$ both connecting $B_n(r)$ to $\widehat{\partial}\D_n.$ Note that one could not easily deduce phase coexistence from the percolation of the sign clusters of the dGFF from \cite{DiWiWu-20}.
\end{enumerate}
\end{remark}

A consequence of \eqref{eq:GFFconnectionBds} and the isomorphism theorem, Proposition~\ref{prop:BEisom}, is the following result for the excursion process, proved at the end of Section~\ref{s:gfflvlperc}.
\begin{corollary} \label{cor:VunPercBds}
For all $r\in{[0,1)},$ there exists $\intens > 0$ such that 
\begin{equation} \label{eq:ineqsVacantSetPerc}
		0< \liminf_{n \to \infty} \Pm\big(B_n(r) \stackrel{\mathcal{V}_{n}^{\intens}}{\longleftrightarrow}
		\widehat{\partial}\D_n \big) \le 
		\limsup_{n \to \infty} \Pm\big(B_n(r) \stackrel{\mathcal{V}_{n}^{\intens}}{\longleftrightarrow}
		\widehat{\partial}\D_n \big) < 1.
\end{equation}
\end{corollary}
Compared to Theorem~\ref{the:maindisexc}, Corollary~\ref{cor:VunPercBds} is stronger in the sense that we do have a full connection up to the boundary; however, it is weaker in the sense that it only holds for some $\intens> 0$ small enough, whereas the connectivity of  Theorem~\ref{the:maindisexc} is valid throughout the entire supercritical phase.

On the cable system one can define a critical parameter $\tilde{h}_*^d(r)$ similarly to $h_*^d(r)$ in \eqref{def:h*} but with ${E}_n^{\ge h}$ replaced by $\widetilde{E}_n^{\ge h}.$  Combining \cite[Lemma~4.13 and Corollary~5.1]{ArLuSe-20a} with \cite[Theorem~1]{Lu-14}, we have the following result. 

\begin{proposition}[\cite{ArLuSe-20a}]
\label{the:htilde=0}
For all $r\in{(0,1)},$ $\tilde{h}_*^d(r)=0,$ and there is no percolation above level $h=0.$
\end{proposition}
Although \cite[Corollary~5.1]{ArLuSe-20a} is stated on the square $\{-n,\dots,n\}^d,$  its proof could easily be adapted to the disk $\D_n,$ and we give an alternative proof in Remark~\ref{rk:cablegff}, \ref{rk:tilh*>0}. Moreover, contrary to what is stated therein, \cite[Corollary~5.1]{ArLuSe-20a} does not hold for $r=0,$ see Remark~\ref{rk:cablegff}, \ref{rk:r=0cablegff}.

Therefore, combining Theorems~\ref{the:maindisexc} and \ref{the:maindisgff} with Proposition~\ref{the:htilde=0}, we obtain for each $r\in{(0,1)}$
\begin{equation}
\label{eq:strictinequalitycritpara}
    \tilde{h}_*^d(r)=0<h_*^d(r)<\sqrt{2\intens_*^d(r)}<\infty.
\end{equation}
Note that when $r=0,$ the situation for the Gaussian free field is different, see Remark~\ref{rk:cablegff}, \ref{rk:r=0cablegff} and \cite{2dGFFperc}. As explained in Section~\ref{sec:background}, the inequalities \eqref{eq:strictinequalitycritpara} were already proved on certain transient graphs for random interlacements and the Gaussian free field: $\tilde{h}_*^d=0$ is proved in \cite{Lu-14,MR3492939,DrePreRod3,Pre1} on all vertex-transitive transient graphs, and in particular on $\Z^d,$ $d\geq3,$ or for the the pinned dGFF in dimension $2$ ; $h_*^d>0$ is proved in \cite{DrePreRod,DrePreRod2,MR3940195,MR3765885,MR4169171} on $\Z^d,$ $d\geq3,$ a class of fractal or Cayley graphs with polynomial growth, a large class of trees, and expander graphs; $h_*^d<\sqrt{2\intens_*^d}$ is proved in \cite{MR3492939,MR3765885} on a large class of trees; $\intens_*^d<\infty$ is proved in \cite{sznitman2010vacant,teixeira2009interlacement,MR2891880,DrePreRod2} on $\Z^d,$ $d\geq3,$ the same class of fractal or Cayley graphs with polynomial growth as before, and non-amenable graphs. However, the question of the strict inequality $h_*^d<\sqrt{2\intens_*^d}$ is still open on $\Z^d,$ $d\geq3,$ see \cite[Remark~4.6]{MR3492939}, and we refer to \cite{MR3940195} for an extension of this question.

\begin{remark}
\phantomsection\label{rk:connectiontoboundary}
\begin{enumerate}[label=\arabic*)]
\item \label{rk:othercritexp} Let us comment on our definition of the critical parameters \eqref{def:u*} and \eqref{def:h*}. An alternative possible definition would be to inverse the order of the limits $\eps\rightarrow0$ and $n\rightarrow\infty$ in \eqref{def:u*} and \eqref{def:h*}. More generally, for all $r\in{[0,1)}$ and any sequence $\boldsymbol{\eps}=(\eps_n)_{n\in\N}\in{[0,\infty)^{\N}}$ decreasing to $0,$ we define
\begin{equation}
\label{eq:defu*depsr}
    u_*^d(\boldsymbol{\eps},r):=\inf\Big\{\intens\geq0:\,\liminf_{n\rightarrow\infty}\Pm\Big(B_n(r)\stackrel{\mathcal{V}^{\intens}_n}{\longleftrightarrow}
		\widehat{\partial}B_n(1-\eps_n)\Big)=0\Big\}.
\end{equation}
Note that if $\eps_n\leq \eps'_n$ for all $n\in\N,$ then $u_*^d(\boldsymbol{\eps},r)\leq u_*^d(\boldsymbol{\eps}',r),$ and that $u_*^d(r)=\sup u_*^d(\boldsymbol{\eps},r),$ where the supremum is taken over all sequences $\boldsymbol{\eps}=(\eps_n)_{n\in\N}\in{[0,\infty)^{\N}}$ decreasing to $0.$ By Theorem~\ref{the:maindisexc} and Corollary~\ref{cor:VunPercBds} we have that for all $r\in{[0,1)}$ and any sequences $\boldsymbol{\eps}=(\eps_n)_{n\in\N}$ decreasing to $0$ 
\begin{equation*}
    u_*^d(\boldsymbol{\eps},r)\begin{cases}
    =\pi/3&\text{if }\eps_n\geq n^{-1/7}\text{ for each }n\in\N,
    \\\in{(0,\pi/3]}&\text{otherwise.}
    \end{cases}
\end{equation*}
and there is no percolation at $\pi/3.$ One can define similarly the critical parameter $h_*^d(\boldsymbol{\eps},r)$ by replacing $\V_n^u$ by $E_n^{\geq u}$ in \eqref{eq:defu*depsr}. Using Theorem~\ref{the:maindisgff}, see in particular \eqref{eq:GFFconnectionBds} for the lower bound, we have for all $r\in{(0,1)}$ and any sequences $\boldsymbol{\eps}=(\eps_n)_{n\in\N}$ decreasing to $0$ 
\begin{equation*}
    0<h_*^d(\boldsymbol{\eps},r)\leq \sqrt{\frac{\pi}{2}}.
\end{equation*}
Finally, for the GFF on the cable system,  \cite[Lemma~4.13 and Corollary~5.1]{ArLuSe-20a} actually imply that $\tilde{h}_*^d(\boldsymbol{\eps},r)=0$ for all $r\in{(0,1)}$ and for any sequences $\boldsymbol{\eps}$ decreasing to $0,$ where $\tilde{h}_*^d(\boldsymbol{\eps},r)$ is defined similarly as \eqref{eq:defu*depsr}, replacing $\V_n^u$ by $\widetilde{E}^{\geq u}_n.$ In particular, the inequalities \eqref{eq:strictinequalitycritpara} still holds with these new definitions of the critical parameters, as long as the sequence $\boldsymbol{\eps}$ satisfies $\eps_n\geq n^{-1/7}$ for all $n\in\N.$

\item We conjecture that $u_*^d(\boldsymbol{\eps},r)=\pi/3$ for any sequence $\boldsymbol{\eps}=(\eps_n)_{n\in\N}$ decreasing to $0,$ and in particular for the choice $\eps_n=0$ for all $n\in\N,$ which seems to be another natural definition of the critical parameter. The difficulty in proving this statement is that the equality $u_*^d(r)=\pi/3$ from Theorem~\ref{the:maindisexc} follows from the coupling between discrete and continuous excursion processes, and a similar result for continuous excursions, see Theorem~\ref{t.mainthm}. However, this coupling fails near the boundary of the disk $\D,$ see Theorem~\ref{the:couplingdiscontexc1}, and it seems that proving the equality $u_*^d(0,r)=\pi/3,$ that is finding a large cluster of $\V_n^u$ reaching the boundary of $\D_n,$ would require an additional purely discrete argument. The result \eqref{eq:percowithoutloops} shows that one can at least have a large cluster of $\V_n^u$ at polynomial distance from the boundary, and can thus be seen as a first step to obtain $u_*^d(0,r)=\pi/3.$ On the other hand, it is not clear if the critical parameter $h_*^d(\boldsymbol{\eps},r)$ depends on the sequence $\boldsymbol{\eps}$ or not, or in fact on $r\in{(0,1)}.$
\end{enumerate}
\end{remark}

\section{Definitions and results for the continuum models}
\label{s.contdef}
\subsection{The Brownian excursion cloud and loop soups}\label{s.be}
The Brownian excursion measure $\mu$ originated as a tool for studying intersection exponents of planar Brownian motion, see for instance \cite{MR2883393} and is the continuum analogue of the discrete excursion measure. It is a $\sigma$-finite measure on trajectories that spend their life time in the unit disk with endpoints on the boundary $\partial \D$. We now briefly recall one way to construct this measure. See also for instance \cite{lawler2000universality}, \cite{virag2003beads}, \cite{lawler2005conformally} and \cite{lawler2004soup}. Let
\begin{equation*}
W_\D := \left\{ w \in C([0,T_w],\overline{\D}) : w(t) \in \D, \forall t \in (0,T_w)  \right\},
\end{equation*} 
and for $K \Subset \D$ we let $W_{K}$ be the set of trajectories in $W_{\D}$ that hit $K$. For $w$ in $W_{\D}$ we write $Z_t(w)=w(t)$.

For $x\in \D$, let $\Pm_x$ denote the law under which  $(Z_t)_{t \geq0}$ is a complex Brownian motion started at $x$ killed upon hitting $\partial \D$. For each probability measure $m$ on $\D,$ let $\Pm_{m}:=\int_{x\in \D} \Pm_x\, m({\rm d} x).$ The Brownian excursion measure $\mu$ on general domains is for instance defined in \cite[Section~5.2]{lawler2005conformally}. When the domain is the unit disk $\D,$ it corresponds to the limit
\begin{equation}\label{e.bemeas}
\mu  = \lim_{\epsilon \to 0} \frac{2 \pi}{\epsilon} \Pm_{\sigma_{1-\epsilon}},
\end{equation} 
where $\sigma_r$ is the uniform probability measure on $\partial B(r)$ for each $r\in{(0,1)}.$
The limit in \eqref{e.bemeas} is meant in terms of the Prokhorov metric on the set of measures on $W_{\D},$ when $W_{\D}$ is endowed with some canonical distance on curves, and we refer to \cite[Section~5.1]{lawler2005conformally} for a precise definition of these notions. The formula \eqref{e.bemeas} can be deduced from combining the bottom of p.124 and (5.12) in \cite{lawler2005conformally}, see also the paragraph after (7) in \cite{lawler2000universality}. Note that $\mu$ is supported on excursions which start and end on $\partial\D.$

Under $\Pm,$ the excursion process $\omega$ is a Poisson point process on $W_\D \times  \BR_+$ with intensity measure $\mu \otimes \id \intens$. For $\intens>0,$ writing $\omega = \sum_{i \geq 0} \delta_{(w_i, \intens_i)},$ where $\delta$ is a Dirac mass, we let 
\begin{equation}\label{e.omega_alpha}
\omega_\intens := \sum_{i \geq 0} \delta_{w_i} \mathbb{I} \{\intens_i \leq \intens\},
\end{equation}
and note that under $\Pm$ the process $\omega_\intens$ is a Poisson point process with intensity measure $\intens \mu.$ We refer to $\omega_\intens$ as the Brownian excursion process at level $\intens$.
For $\intens>0$, the Brownian excursion set at level $\intens$ is then defined as
\begin{equation*}
\CI^{\intens} := \bigcup_{\intens_i \le \intens } \bigcup_{s =0}^{T_{w_i}} w_i(s)
\end{equation*}
and we let $\V^{\intens} := \D\setminus \be^{\intens}$ denote the corresponding vacant set. \cite[Proposition~5.8]{lawler2005conformally} says that $\mu$, and consequently $\Pm$, are invariant under the conformal automorphisms of $\D$.

We now discuss how the random set $\be^{\intens}\cap K$ can be generated for a compact $K$ in ${\mathbb D}$. This is what we refer to as ``local picture". We first introduce some additional notation.  For $K\Subset \D$, 
let $\tau_K:=\inf\{t>0;\,Z_t\in K\}$ be the hitting time of $K$ and let $L_K:=\sup\{0<t\le \tau_{\partial \D}\,:\,Z_t\in K\}$ denote the last exit time, with the conventions $\inf\varnothing=+\infty$ and $\sup\varnothing=-\infty.$  A point $x$ is said to be regular for $K$ if $\Pm_x(\tau_K=0)=1$. We will assume that all compact sets $K$ appearing below satisfy the condition that all $x\in K$ are regular.

For $K \Subset \D$ let $e_K( {\rm d}y)$ denote the equilibrium measure (for Brownian motion in $\D$) of $K$, see for example \cite[Theorem~$24.14$]{kallenberg_2002}.  It is the finite measure supported on $\partial K$ satisfying
\begin{equation}\label{e.eqmeas}
\Pm_x \left(  Z(L_K) \in {\rm d}y,\, L_K>0  \right) = G(x,y) e_K( {\rm d}y),
\end{equation}
where $G(x,y)$ is the Green's function for Brownian motion in $\D$ stopped upon hitting $\partial \D$. We recall that
\begin{equation}
\label{eq:Greeneverywhere}
G(w,z)=\frac{1}{2\pi}\log\frac{|1-\bar{w}z|}{|w-z|}\mbox{ for }w,z\in \D,
\end{equation}
so that in particular
\begin{equation}
\label{eq:Greenat0}
G(0,r e^{i \theta})=\frac{\log(1/r)}{2\pi}\mbox{ for }0<r<1\mbox{ and }0\le \theta<2\pi.
\end{equation}
Furthermore, the capacity (relative to $\D$) of $K\Subset \D$ is denoted by $\capac(K)$ and is defined as the total mass of $e_K,$ and we denote by $\overline{e}_K:=e_K/\capac(K)$ the normalized equilibrium measure.

The expression for $e_{B(r)}$ for $0<r<1$ is known and can be derived at once; indeed, using the above and~\eqref{e.eqmeas}, we have
\begin{equation}\label{e.eqmeasexpr1}
\Pm_0 \left(  Z(L_{B(r)}) \in {\rm d}y, 0 < L_{B(r)}  \right)=\frac{\log(1/r)}{2\pi}e_{B(r)}(  {\rm d}y).
\end{equation}
On the other hand, by rotational invariance, we have that
\begin{equation}\label{e.eqmeasexpr2}
\Pm_0 \left(  Z(L_{B(r)}) \in {\rm d}y, 0 < L_{B(r)}   \right)=\sigma_r( {\rm d}y),
\end{equation}
From~\eqref{e.eqmeasexpr1} and~\eqref{e.eqmeasexpr2} we get that
\begin{equation}\label{e.eqmeasdisc}
e_{B(r)}({\rm d}y)=\frac{2\pi}{\log(1/r)}\sigma_r({\rm d}y).
\end{equation}
The capacity of $B(r)$ is therefore given by
\begin{equation}
\label{capball}
\capac(B(r))=\int_{y\in \partial B(r)} e_{B(r)}({\rm d}y)=\frac{2\pi}{\log(1/r)}.
\end{equation}
For $K\Subset \D$, the hitting kernel is defined as 
\[
h_K(x,{\rm d}y):=\Pm_x(Z(\tau_K)\in {\rm d}y,\,\tau_K<\infty).
\]
The equilibrium measure satisfies the following consistency property, see \cite[Proposition~24.15]{kallenberg_2002}: If $K_1\Subset K_2 \Subset \D$, then
\begin{equation}\label{e.consistency}
e_{K_1}({\rm d}y)  = \int_{x\in \partial K_2} h_{K_1}(x,{\rm d}y) e_{K_2}( {\rm d}x).
\end{equation}
For all $x\in{\R^2}$ and $r<|x|<R,$ it thus follows from \eqref{e.eqmeasdisc}, \eqref{e.consistency} and invariance by scaling and rotation of Brownian motion that
\begin{equation}
\label{eq:hittingBM}
\Pm_x \left( \tau_{B(r)}< \tau_{B(R)^\ch} \right) = \frac{\log(R/|x|)}{\log(R/r)}.
\end{equation}
When $B=B(r)$ for some $r\in{(0,1)},$ we can pointwise define a probability measure $\Pm_x^B, x \in \partial B,$ as the weak limit $\lim_{z \to x} \Pm_z ( \, \cdot \,  | \, \tau_B = \infty),$ where the limit is taken for $z\in{\D\setminus B}.$ The existence of this limit can be proven using \cite[Theorem~4.1]{MR932248} for the domain $D=\D\setminus B,$ and $\Pm_x^B$ corresponds to the normalized excursion measure from \cite[Definition~3.1]{MR932248}, restricted to the trajectories hitting $\partial\D$ (this restricted measure is finite, see the end of p.\ 34 in \cite{MR932248}). Note moreover that the limit as $z\rightarrow x$ of $G_D(z,y)/P_z(\tau_B=\infty)$ is positive by \eqref{eq:hittingBM} and since the normal derivative of $G_D(x,y)$ is positive for each $y\in{D}$ (and is in fact a multiple of the Poisson kernel). Then $\Pm_x^B$ corresponds informally to the law of $(Z_{s+L_B})_{s\geq0}$ under $\Pm_{y}(\cdot\,|\,L_B\geq0,Z({L_B})=x)$, and for $x\in{\overline{B}^c}$ we define $\Pm_x^B=\Pm_x(\cdot\,|\,\tau_B=\infty)$. Using invariance by time reversal of Brownian motion, see \cite[Theorem~24.18]{kallenberg_2002} we thus have that for all $s\in{(0,1)}$ with $\overline{B}\subset B(s),$
\begin{align}
\begin{split}
\label{lawbeforeHK}
    \text{ under }&\Pm_{e_{B(s)}}(\cdot\,|\,\tau_B<\infty,Z(\tau_B)),\ \text{the process } (Z(\tau_B-t))_{t\in{(0,\tau_B]}}\text{ has } \\
    &\text{ law }\Pm_{Z(\tau_B)}^B\big((Z(t))_{t\in{(0,L_{B(s)}]}}\in{\cdot}\big).
\end{split}
\end{align}

One can prove \eqref{lawbeforeHK} as follows: let $B^{(\eps)}$ be a ball with the same center as $B,$ but with radius increased by $\eps>0,$ and define  $L^{(\eps)}:=\sup\{t\leq \tau_B:\,X_t\in{\partial B^{(\eps)}}\}.$  One can use invariance by time reversal and the definition of $\Pm_{\cdot}^B$ to show that, on the event $L^{(\eps)}\geq0,$ under $\Pm_{e_{B(s)}}(\cdot\,|\,\tau_B<\infty,Z({L^{(\eps)}})),$   $(Z(L^{(\eps)}-t))_{t\in{(0,L^{(\eps)}]}}$ has law $\Pm_{Z(L^{(\eps)})}^B((Z(t))_{t\in{[0,L_{B(s)})}}\in{\cdot}).$ Letting $\eps\rightarrow0,$ we obtain \eqref{lawbeforeHK}.

For any measurable sets of trajectories $A, A'$ in ${W}_{\D},$ any $r\in{(0,1)}$ and $\eps\in{(0,1)}$ with $r<1-\eps,$ using the strong Markov property at time $\tau_B$, with $B=B(r)$, we have 
\begin{equation}\label{e.beforwardmeas}
\begin{split}
&  \frac{2 \pi}{ \epsilon} \Pm_{\sigma_{1-\epsilon}} \left( ( Z({\tau_B-t}))_{t \in{(0,\tau_B]}}  \in A',  ( Z({t+\tau_B}))_{t \geq 0}  \in A, \tau_B< \infty \right)   \\
&=  \frac{2 \pi }{ \epsilon } \E_{\sigma_{1-\epsilon}} \left[ \Pm_{Z(\tau_B)} (A) \I\{ ( Z({\tau_B-t}))_{t \in{(0,\tau_B]}}  \in A',\ \tau_B< \infty\} \right]  \\
&\overset{\eqref{e.eqmeasdisc}}{=}  \frac{- \log(1-\epsilon) }{ \epsilon }  \E_{e_{B(1-\epsilon)}} \left[ \Pm_{Z(\tau_B)} (A) \I\{ ( Z({\tau_B-t}))_{t \in{(0,\tau_B]}}  \in A',\ \tau_B< \infty\} \right]  \\
&\overset{\eqref{lawbeforeHK}}{=}  \frac{- \log(1-\epsilon) }{ \epsilon }\E_{e_{B(1-\epsilon)}} \left[ \Pm_{Z(\tau_B)} (A)\Pm_{Z(\tau_B)}^B( ( Z({t}))_{t \in{(0,L_{B(1-\eps)}]}}  \in A') \I\{ \tau_B< \infty\} \right]  \\
&= \frac{- \log(1-\epsilon) }{ \epsilon } \int_{\partial B(1-\epsilon)} \int_{\partial B} {\Pm}_y(A)\Pm_{y}^B( ( Z({t}))_{t \in{(0,L_{B(1-\eps)}]}}  \in A') h_B(x,{\rm d}y)  e_{B(1-\epsilon)}( {\rm d}x)  \\
&=\frac{- \log(1-\epsilon) }{ \epsilon }
\int_{\partial B}  {\Pm}_y(A) \Pm_{y}^B( ( Z({t}))_{t \in{(0,L_{B(1-\eps)}]}}  \in A') \int_{\partial B(1-\epsilon)}h_B(x,{\rm d}y)e_{B(0,1-\epsilon)}( {\rm d}x)   \\
&\stackrel{~\eqref{e.consistency}}{=} \frac{- \log(1-\epsilon) }{ \epsilon }
\int_{\partial B}  \Pm_y(A)\Pm_{y}^B( ( Z({t}))_{t \in{(0,L_{B(1-\eps)}]}}  \in A')  e_B({\rm d}y).
\end{split}
\end{equation}
In particular, letting $\eps\rightarrow0,$ we obtain
\begin{equation}
    \label{e.beforwardmeas2}
    \mu\big( ( Z({\tau_B-t}))_{t \in{(0,\tau_B]}}  \in A' , ( Z({t+\tau_B}))_{t \geq 0}  \in A,\tau_B< \infty \big)=\int_{\partial B} \Pm_y(A)\Pm_y^B(A')  e_B( {\rm d}y).
\end{equation}
Indeed, consider the map from 
\[\{(w_1,w_2)\in{W_{\D}\times W_{\D}}:w_1(0)=w_2(0)\in{\partial B}\text{ and }w_1(t)\notin{B}, \,\forall\, t>0\}
\]
to $W_{\overline{B}}$ which associates to $(w_1,w_2)$ the trajectory  $w$ characterized via $w(t)=w_1(T_{w_1}-t)$ if $t\leq T_{w_1}$ and $w(t)=w_2(t-T_{w_1})$ if $t\in{[T_{w_1},T_{w_1}+T_{w_2}]}.$ It is then not hard to check that this map defines a homeomorphism for  the relative topologies, and thus \eqref{e.beforwardmeas} entirely characterizes the law of $2\pi\eps^{-1}\mathbb{P}_{\sigma_{1-\eps}}.$ Since the measure on the right-hand side of \eqref{e.beforwardmeas2} is clearly the limit (for the Prokhorov metric introduced below \eqref{e.bemeas}) of the measure in the last line of \eqref{e.beforwardmeas} as $\eps\rightarrow0$, we obtain \eqref{e.beforwardmeas2} by \eqref{e.bemeas}. Note that one could  alternatively see \eqref{e.beforwardmeas2} as a consequence of  \cite[(5.6)]{lawler2005conformally}, but our current proof of \eqref{e.beforwardmeas2} highlights the link between the Brownian excursion cloud and Brownian interlacements, as it is proved similarly as \cite[Lemma~2.1]{sznitman2013scaling}. One can rephrase \eqref{e.beforwardmeas2} in terms of Poisson point processes as follows.

\begin{proposition}
\label{prop:localdescription}
For each $r\in{(0,1)},$ writing $B=B(r),$ let $N_B \sim \text{Poi}( \intens \capac (B)),$ let $(\sigma_i)_{i\geq1}$ be an i.i.d.\ family of random variables on $\partial B$ with distribution $\overline{e}_B.$ Conditionally on the  $(\sigma_i)_{i\geq1},$ let $(Z_i^+)_{i\geq1},$ resp.\ $(Z_i^-)_{i\geq1},$ be two families of independent Brownian motions resp.\ excursions with law $\Pm_{\sigma_i},$ resp.\ $\Pm_{\sigma_i}^B,$ for each $i\geq1.$ Let us also define $Z_i(t)=Z_i^-(T_{Z_i^-}-t)$ if $t\leq T_{Z_i^-},$ and $Z_i(t+T_{Z_i^-})=Z_i^+(t)$ if $t\in{(0,T_{Z_i^+}]}$ for each $i\geq1.$ Then $\sum_{i=1}^{N_B}\delta_{Z_i}$ has the same law as the point process of trajectories in the support of $\omega_\intens$ hitting $B.$
\end{proposition}
\begin{remark}
There is a similar local description for Brownian interlacements on balls in ${\mathbb R}^d$ for $d\ge 3$, see \cite[(2.3)]{sznitman2013scaling}. Actually, modulo time shift, this description can be extended to any compact $K$ when considering forwards part of trajectories only, see \cite[(2.24)]{sznitman2013scaling}, and one could also prove a similar statement in our context. Namely, for a compact $K\Subset \D,$ let $N_K\sim \text{Poi}( \intens \capac (K)),$ let $(\sigma_i)_{i\geq1}$ be an i.i.d.\ family of random variables on $\partial K$ with distribution $\overline{e}_K,$ and conditionally on the former let $(Z_i^+)_{i\geq1}$  be a family of independent Brownian motions with law $\Pm_{\sigma_i}$ for each $i\geq1.$ Then $\sum_{i=1}^{N_K}\delta_{Z_i^+}$ has the same law as the point process of trajectories in the support of $\omega_\intens$ hitting $K,$ started at their first hitting time of $K.$ 
\end{remark}

We end this section by defining the Brownian loop soup. This is done in complete analogy with the random walk loop soup, and the first construction is due to Lawler and Werner in \cite{lawler2004soup}. 
The (rooted) Brownian loop measure on $\BC$ is defined by
\begin{equation}
\label{eq:defcontloops}
\mu_{\rm loop}^r (\cdot) := \int_{\mathbb{C}} \int_0^\infty \frac{1}{2 \pi t^2}\Pm_{x,x}^t(\cdot )\, \id t \, \id x, 
\end{equation}
where $\Pm_{x,x}^t(\cdot )$ denotes the law of a Brownian bridge of duration $t$. By an unrooted loop we mean an equivalence class of loops modulo time shift, where two loops $\gamma, \gamma'$ are equivalent if one can be obtained by a time shift of the other: $\gamma(\cdot) = \gamma'(\cdot +h), h >0$. The unrooted loop measure, $\mu_{\rm loop}$, is then defined by the image of the canonical projection of rooted loops on unrooted loops. 

The Brownian loop soup in $\D$ with parameter $\lambda>0$ is then defined as a Poisson point process on the space of unrooted loops with intensity measure $\lambda \mu_{\rm loop}^{\D},$ where $\mu_{\rm loop}^{\D}$ is the restriction of $\mu_{\rm loop}$ to the set of unrooted loops contained in $\D$. The loops in the loop soup form clusters: two loops $\el, \el'$ are in the same cluster if there is a finite set of loops $\{\el_0, \ldots, \el_n\}$ such $\el_0=\el, \el_n=\el'$ and $\el_{j-1} \cap \el_j \neq \varnothing.$  Up to extending the underlying probability space, we moreover always assume that the Brownian loop soup is defined under $\Pm,$ and independent from the Brownian excursion process $\omega_u,$ see \eqref{e.omega_alpha}. For each $\lambda\geq0$ and $\intens>0,$ we denote by $\be^{\intens,\lambda}$ the closure of the union of all the clusters of loops, for the loop soup  at level $\lambda,$ which hit the Brownian excursion set $\be^\intens$ at level $\intens.$ Moreover, we denote by $\V^{\intens,\lambda}:=(\be^{\intens,\lambda})^{\ch}$ the corresponding vacant set. Note that $\be^{\intens,0}=\be^{\intens}$ and $\V^{\intens,0}=\V^{\intens}.$

\subsection{Conformal restriction and the Schramm-Loewner evolution}\label{s.sle}
The determination of the critical values in the continuum setting is based on a Schramm-Loewner evolution (SLE) computation and the well-known link to restriction measures. This section recalls the needed facts. While most (if not all) of this is standard material, we have chosen to give statements and provide references. See, e.g.\ \cite[Section~8]{lawler2003conformal} for further discussion. Let $\HH := \{z: \Im \, z > 0\}$ be the complex upper half-plane. Let $X_-$ be the set of continuous curves $\gamma$ connecting $0$ with $\infty$ in $\overline{\mathbb{H}}$ and with the property that $\gamma \cap \mathbb{R} \subset (-\infty,0]$. One may turn the set of curves into a metric space either by viewing them as continuous functions up to increasing reparametrization (with the associated supremum norm) or as compact sets with the Hausdorff topology after mapping to the disc. The exact point of view is not important in the present context and we will not discuss this in further detail. We say that a probability measure $\Pm=\Pm_{\mathbb{H},0, \infty}$ on $X_{-}$ satisfies one-sided conformal restriction with exponent $\alpha > 0$ if for any (relatively) compact $A$ such that $\overline{A}\cap\H=A,$ $\HH \setminus A$ is simply connected and $\overline A \cap \mathbb{R} \subset (0,\infty)$, we have that
\begin{equation}
\label{eq:onesidedrestriction}
\Pm\left( \gamma \cap A = \varnothing \right) = \phi_A'(0)^\alpha.
\end{equation}
Here $\phi_A: \HH \setminus A \to \HH$ is the conformal map fixing $0$ and satisfying $\phi_A(z) = z + o(z)$ as $z \to \infty$. A conformally invariant measure defined in some other simply connected domain with two marked boundary points is said to satisfy one-sided conformal restriction if its image in $\HH$ does so. If $A$ is as above then the law $\Pm_{\mathbb{H}, 0, \infty}$ conditioned on $\gamma \cap A = \varnothing$ is the same as $\Pm_{\mathbb{H} \setminus A, 0, \infty}$, the latter defined by push-forward via conformal transformation. 

An important example of a probability measure which satisfies one-sided conformal restriction can be obtained from the (right boundary of a) cloud sampled from a Poisson realization of the Brownian excursion measure, as we now recall. Let $D\neq\BC$ be a simply connected Jordan domain, let $\phi^D:\D\rightarrow D$ be some fixed conformal transformation, and, for each $u>0,$ denote by $\omega_u^D$ the point process obtained by replacing each excursion $w_i$ in the definition \eqref{e.omega_alpha} of $\omega_u$ by its image via $\phi^D.$  This corresponds to the usual definition of the Brownian excursion process on $D$ by  \cite[Proposition~5.8]{lawler2005conformally}, and its law does not depend on the particular choice of the conformal transformation $\phi^D.$ For two arcs $\Gamma_1,\Gamma_2\subset\partial D,$ let $\be^{\intens,D}_{\Gamma_1,\Gamma_2}$ be the union of the traces of the trajectories in $\omega_u^D$ which start on $\Gamma_1$ and end on $\Gamma_2.$ Let $\Gamma\subset \partial D$ be a Jordan arc and suppose $K\subset \overline{D}$ is a closed set whose intersection with $\partial D$ is contained in $\Gamma$. Write $\Gamma^\ch = \partial D \setminus \Gamma$ and let $F_{D,\Gamma}(K)$ be the filling of $K$ with respect to $\Gamma$, that is,  $F_{D,\Gamma}(K)$ is the union of $K$ with all $z \in \overline D$ such that $K$ separates $z$ from $\Gamma^\ch$ in $\overline D$. We write 
\begin{equation}\label{e.sepbry}
\partial_{D,\Gamma} K :=( \partial (D \setminus F_{D,\Gamma}(K)) ) \setminus \Gamma^\ch.
\end{equation}
Then, as shown in \cite[Theorem~8]{werner2005conformal} with a multiplicative constant $c\alpha,$ and \cite[Theorem~2.12]{wu2015conformal} for the explicit constant $\pi \alpha,$ we have the following.
\begin{lemma}[\cite{werner2005conformal,wu2015conformal}]
\label{lem:excarerestriction}
    For each $\alpha>0,$ $\partial_{\H,\R^-}\mathcal{I}^{\pi\alpha,\H}_{\R^-,\R^-}$ satisfies one-sided conformal restriction with exponent $\alpha.$
\end{lemma}

We now discuss the link between one-sided restriction and SLE, following \cite{lawler2003conformal,werner-wu-cle}. Let $B_{t}$ be a one-dimensional standard Brownian motion and for $\kappa > 0$, set $U_{t} = B_{\kappa t}$. The SLE$_{\kappa}$ Loewner chain is defined by
\[
\partial_{t} g_{t}(z) = \frac{2}{g_{t}(z)-U_{t}}, \qquad 0 \le t < T_{z}, \qquad g_{0}(z)=z,
\]
where  $T_{z} := \inf\{t \ge 0 :  \textrm{Im}\, g_{t}(z) = 0\}$. The associated hulls are defined by $K_t = \{z: T_z \le t\}, \, t \ge 0$. For each $t$, $g_t$ is a conformal map from the unbounded connected component $H_t$ of $\mathbb{H}\setminus K_t$ onto $\mathbb{H}$. The SLE$_{\kappa}$ curve, which connects $0$ with $\infty$ in $\mathbb{H}$ can then be defined by $\gamma(t) := \lim_{y \downarrow 0} g_{t}^{-1}(U_{t} + iy), \, t\ge 0$. The SLE curve in other domains is defined by conformal transformation.

Next, for $\rho>-2$ and $v\in{\R}$ consider the SDE
\[
dW_{t} = \sqrt{\kappa}dB_{t} - \frac{\rho \, dt}{V_{t}-W_{t}}, \quad dV_{t}=\frac{2 \, dt}{V_{t}-W_{t}}, \quad (W_{0}, V_{0}) = (0,v).
\]
The SLE$_{\kappa}(\rho)$ Loewner chain with force point $v$ is the Loewner chain driven by $(W_{t})$ as above. The case of relevance to this paper is $v=0^-.$ Care is needed if $\rho$ is too large in absolute value and negative, though this is not an issue in the cases we will consider. 

Define for each $\kappa\in{[8/3,4]}$ and $\alpha > 0$ the function
\begin{equation}
    \label{eq:rhokappaalpha}
    \rho_\kappa(\alpha) := \frac{-8 + \kappa + \sqrt{
   16 + \kappa(16\alpha-8) + \kappa^2}}{2}.
\end{equation}
It was proved in \cite[Theorem~8.4]{lawler2003conformal} that the left-filling of the hulls of an SLE$_{8/3}(\rho)$ process with $\rho=\rho_{8/3}(\alpha)$ satisfies one-sided restriction with exponent $\alpha$ thereby providing a ``Brownian'' construction of SLE curves. By adding an independent CLE (conformal loop ensemble) process this can be extended to other values of $\kappa$, see \cite{werner-wu-cle}: Let $\kappa \in [8/3,4]$. Sample a Brownian loop soup in $\mathbb{D}$ of intensity $\lambda(\kappa),$ as defined below \eqref{eq:defcontloops}, where $2\lambda(\kappa),$ see \eqref{eq:defckappa}, is the corresponding \emph{central charge} parameter, which also appears in the context of Conformal Field Theory. (See the discussion of the choice of intensity on p1 of \cite{Lupu-CLE}; it differs by a factor $2$ to the choice made in \cite{lawler2004soup}.) For a domain $D,$ we define the Brownian loop soup  in $D$ with intensity $\lambda(\kappa)$ as the image by $\phi^D$ of the Brownian loop soup in $\D$ at intensity $\lambda(\kappa),$ defined below \eqref{eq:defcontloops}   We define a cluster of loops in $D$ similarly as in $\D,$ see below \eqref{eq:defcontloops}. The collection of outer boundaries of the set of clusters form a conformal loop ensemble, CLE$_\kappa$ process in $D$, see \cite{sheffield-werner}. Now consider a curve $\gamma$ from $0$ to $\infty$ in $\H$ which satisfies one-sided restriction with exponent $\alpha,$ independent of the loop soup, and let $S(\kappa,\alpha)$ be the closure of the set of loop clusters, for the loop soup in $\H$ of intensity $\lambda(\kappa),$ which hit $\gamma,$ and $\eta(\kappa,\alpha)=\partial_{\H,\R^-}S(\kappa,\alpha).$

\begin{lemma}[\cite{werner-wu-cle}]
\label{lem:etaisSLE}
 The set $\eta(\kappa,\alpha)$ has the law of the trace of an SLE$_\kappa(\rho_{\kappa}(\alpha))$ curve.  
\end{lemma}

In view of Lemma~\ref{lem:excarerestriction}, if we take $\gamma$ to be the rightmost boundary of $\be_{\R^-,\R^-}^{\pi\alpha,\H}$, then adding loop clusters with intensity $\lambda(\kappa)$ to $\be_{\R^-,\R^-}^{\pi\alpha,\H},$ we obtain SLE$_\kappa(\rho_{\kappa}(\alpha)).$ This property can be extended to the restriction of the Brownian excursion process to subsets of $\D$ (or $\H$). For a domain $D\subset\D$ with $\partial D\cap\partial\D\neq\varnothing$ and $\Gamma_1,\Gamma_2\subset \partial D\cap\partial\D,$ let $\be^{\intens}_{\Gamma_1,\Gamma_2,D}$ be the union of the traces of the trajectories in $\omega_u$ entirely included in $D,$ starting on $\Gamma_1$ and ending on $\Gamma_2.$ In other words, $\be^{\intens}_{\Gamma_1,\Gamma_2,D}$ corresponds to the trajectories in $\be^{\intens,\D}_{\Gamma_1,\Gamma_2}$ which are included in $D.$ Moreover for each $\lambda>0,$ let $\be^{\intens,\lambda}_{\Gamma_1,\Gamma_2,D},$ resp.\ $\be^{\intens,\lambda,D}_{\Gamma_1,\Gamma_2},$ be the closure of the union of the loop clusters with intensity $\lambda,$ for the restriction of the Brownian loop soup in $\D$ to loops entirely included in $D,$ resp.\ for the Brownian loop soup in $D,$  which intersect $\be^{\intens}_{\Gamma_1,\Gamma_2,D},$ resp.\ $\be^{\intens,D}_{\Gamma_1,\Gamma_2}.$ Note that $\be^{\intens,\lambda}_{\partial\D,\partial\D,\D}=\be^{\intens,\lambda,\D}_{\partial\D,\partial\D}=\be^{\intens,\lambda},$ as defined below \eqref{eq:defcontloops}. We also take the convention $\be^{u,0}_{\Gamma_1,\Gamma_2,D}:=\be^{u}_{\Gamma_1,\Gamma_2,D}$ and $\be^{u,0,D}_{\Gamma_1,\Gamma_2}:=\be^{u,D}_{\Gamma_1,\Gamma_2}.$  Following \cite[Proposition~5.12]{lawler2005conformally} or  \cite[(7)]{lawler2000universality}, the Brownian excursion set satisfies the following form of restriction property.

\begin{lemma}
\label{lem:resexc}
Let $D\subset \D$ be a Jordan domain with $\partial D\cap\partial\D\neq\varnothing$, and let $\Gamma_1,$ $\Gamma_2$ be two closed arcs of $\partial D\cap\partial \D.$ Then for each $u>0$ and $\lambda\geq0,$ $\be^{\intens,\lambda}_{\Gamma_1,\Gamma_2,D}$ and $\be^{\intens,\lambda,D}_{\Gamma_1,\Gamma_2}$ have the same law.
\end{lemma}
\begin{proof}
Let $\Gamma_1^{(n)}$ and $\Gamma_2^{(n)}$ be two sequences of closed arcs such that $\Gamma_1^{(n)}$ and $\Gamma_2^{(n)}$ are disjoint for each $n\in\N,$ and the sets $\Gamma_1^{(n)}\times\Gamma_2^{(n)},$ $n\in\N,$ form a partition of $\{(x,y)\in{\Gamma_1\times\Gamma_2}:x\neq y\}.$ Then by \cite[ Proposition~5.2]{lawler2005conformally}, the sets $\be^{\intens}_{\Gamma_1^{(n)},\Gamma_2^{(n)},D}$ and $\be^{\intens,D}_{\Gamma_1^{(n)},\Gamma_2^{(n)}}$ have the same law, and by independence we can conclude since a.s.\ there are no trajectory starting and ending at the same point. Adding the loop soup clusters, which have the same law for both sets by \cite[Proposition~6]{lawler2004soup}, we can conclude. 
\end{proof}

Let $\Gamma\subset \partial \D$ be a Jordan arc, and abbreviate $\partial_{\Gamma}K:=\partial_{\Gamma,\D}K,$ see \eqref{e.sepbry}. For a Jordan domain $D\subset\D$ such that $\Gamma\subset\partial D\cap\partial\D,$ $\Gamma\neq\partial \D,$ we also denote by $\phi_{\Gamma}^D$ some choice of conformal transformation from $D$ to $\H,$ which maps $\Gamma$ to $\R^-.$ Combining Lemmas~\ref{lem:excarerestriction}, \ref{lem:etaisSLE} and \ref{lem:resexc}, with conformal invariance we obtain the following result, which is our main tool to study percolation of $\be^{\intens,\lambda}$ in Appendix~\ref{sec:contperco}.

\begin{lemma}\label{l.brownianexcursionandlooprestriction}
Let $D\subset \D$ be a simply connected Jordan domain with $\partial D\cap\partial\D\neq\varnothing$, and $\Gamma$ be a closed arc of $\partial D\cap\partial \D,$ $\Gamma\neq\partial \D.$ Then, for any $\kappa\in{[8/3,4]}$ and $\alpha>0,$ $\phi_{\Gamma}^D(\partial_{\Gamma}\be^{\pi \alpha,\lambda(\kappa)}_{\Gamma,\Gamma,D})$  has the same law as the trace of an $\text{SLE}_{\kappa}(\rho_{\kappa}(\alpha))$ curve.
\end{lemma}

Let us finish this section with the following consequence of  \cite[Lemma~8.3]{lawler2003conformal} and \eqref{eq:rhokappaalpha}.

\begin{lemma}[\cite{lawler2003conformal}]
\label{lem:SLE-ka-r}
For $\alpha > 0$, let $\gamma=\gamma([0,\infty))$ be the trace of the SLE$_{\kappa}(\rho_{\kappa}(\alpha))$ curve in $\mathbb{H}.$ 
If $0 < \alpha < (8-\kappa)/16$ then almost surely
$\gamma$ intersects $(-\infty, 0),$ and if $\alpha \ge (8-\kappa)/16$ then almost surely $\gamma$ does not intersect $(-\infty, 0)$.
 \end{lemma}

\subsection{Statements of continuum results}
For $r\in{[0,1)}$, recall that $B(r){\longleftrightarrow}\partial\D\text{ in }\CV^{\intens,\lambda}$ corresponds to the event that the closure of a component of $\CV^{\intens,\lambda}$ intersecting $B(r)$ also intersects $\partial\D$. When $\lambda=0,$ define the critical parameter
\begin{equation}
\label{def:u*c}
\intens_*^c(r) := \inf \left\{ \intens \ge 0 :  \Pm\left(B(r)\stackrel{\V^{\intens}}{\longleftrightarrow}\partial\D\right)=0 \right\}.
\end{equation}

Recall the definition of $\lambda(\kappa)$ from \eqref{eq:defckappa}.

\begin{theorem}\label{t.mainthm}
For all $\kappa\in{[8/3,4]}$ and $r\in{[0,1)}$
\begin{equation*}
    \Pm\big(B(r)\stackrel{\CV^{\intens,\lambda(\kappa)}}{\longleftrightarrow}\partial\D\big)=0\text{ if and only if }u\geq\frac{(8-\kappa)\pi}{16}.
\end{equation*}
In particular, the critical value for percolation in $\CV^{\intens}$ satisfies $u_*^c(r)=\pi/3.$
\end{theorem}

The statement in Theorem~\ref{t.mainthm} is given in \cite[Section~5]{werner-qian}, and the main ingredients of the proof are found there. In Appendix~\ref{sec:contperco}, for the convenience of the reader we complete the details of the proof. A consequence of Theorem~\ref{t.mainthm} and the result on visibility from \cite{elias2017visibility} is that for $\intens\in [\pi/4,\pi/3)$, we have $ \Pm\big(0{\leftrightarrow}\partial\D\text{ in }\CV^{\intens}\big)>0$, but a.s.\ no visibility to infinity from the origin, as well as no percolation in $\V^{u,1/2}.$

The percolative properties of $\V^{\intens,\lambda(\kappa)}$ are particularly interesting for two special values of $\kappa$: first $\kappa=8/3,$ with $\lambda(8/3)=0,$ which simply corresponds to $\V^{\intens}$ and gives us the equality $u_*^c=\pi/3.$ The other value of special interest is $\kappa=4,$ and $\lambda(4)=1/2,$ which is linked to the Gaussian free field (GFF). Indeed, for each $h>0$ denote by $\mathbb{A}_{-h}$ the first passage set of the GFF on $\D$ with zero-boundary condition, as defined in \cite{ArLuSe-20b}, which informally corresponds to the set of points in $\D$ which can be connected to $\partial\D$ by a path above level $-h$ for the continuous GFF on $\D.$ Then by \cite[Proposition~5.3]{ArLuSe-20a}, $\mathbb{A}_{-h}$ has the same law as $\be^{\frac{h^2}{2},1/2}\cup\partial\D$. (Note that $\be^{\frac{h^2}2}$ corresponds to the Brownian excursion set at level $h$ in the parametrization of \cite{ArLuSe-20a}.) We thus directly deduce the following from Theorem~\ref{t.mainthm}.

\begin{corollary}
\label{cor:percocontGFF}
For each $r\in{[0,1)},$ the probability that $B(r)$ intersects a connected component of the complement of $\mathbb{A}_{-h}$ intersecting $\partial \D$, is $0$ if and only if $h\geq\sqrt{\pi/2}.$
\end{corollary}

\begin{remark}
\begin{enumerate}[label=\arabic*)]
\phantomsection\label{rk:othercontperco}
\item \label{rk:differentnormalization}
The exact value $u_*^c(r)=u_*^d(r)=\pi/3,$ as well as the bound $h_*^d(r)\leq\sqrt{\pi/2}$ from Theorem~\ref{the:main} depend on our choice of normalization for the definition of the Gaussian free field and the excursion clouds, but the inequalities $0<h_*^d(r)<\sqrt{2u_*^d(r)}$ do not (as long as the normalization is consistent). We shortly explain how these values would change for a different choice of normalization. Indeed, if one divides by some $t>0$ the  discrete Green function \eqref{eq:killedGreen}, multiplies by $t$ the discrete boundary Poisson kernel \eqref{eq:defdisPoisson}, and also multiplies by $t$ the measure $\mu$ of Brownian excursion from \eqref{e.bemeas}, then we would obtain $u_*^c(r)=u_*^d(r)=\pi/(3t),$ as well as the bound $h_*^d(r)\leq\sqrt{\pi/(2t)}.$ Our specific choice of normalization is consistent with the usual literature on the Brownian excursion cloud, and thus for instance consistent with the normalization from \cite{ArLuSe-20a}, but another possible natural choice of normalization from the discrete point of view would be to take $t=1/4$ to remove the division by $4$ in \eqref{eq:killedGreen} or the multiplication by $4$ in \eqref{eq:defdisPoisson}. This corresponds to considering $\D_n$ as a weighted graph with total weight $1$ instead of $4$ at each vertex $x\in{\D_n},$  see Remark~\ref{rk:defexcursion},\ref{rk:choiceofweight}.
\item Using \eqref{eq:defckappa}, one could also equivalently phrase Theorem~\ref{t.mainthm} in terms of the intensity $\lambda$ of the Brownian loop soup. Namely for all $\intens>0$ and $\lambda\in{[0,1/2]},$ if either $\intens<\pi/4$ or
 \begin{equation*}
     \intens<\frac\pi3\text{ and }\lambda<\frac{(\pi-3\intens)(8\intens-\pi)}{\pi(\pi-2\intens)},
 \end{equation*}
 then the probability that $0$ is in an infinite component of the vacant set $\V^{\intens,\lambda}$ is strictly positive, and otherwise this probability is $0.$ 

Note that if $\lambda>1/2,$ then $\mathcal{V}^{u,\lambda}=\varnothing$ a.s.\ for all $u>0,$ see \cite[Lemma~9.4 and Proposition~11.1]{sheffield-werner}.

\item \label{finiteenergy}By Theorem~\ref{t.mainthm}, the positivity of the probability that $B(r)$ is connected to $\partial\D$ in the vacant set $\V^{\intens,\lambda}$ does not depend on the choice of $r\in{[0,1)}.$ This fact could actually be proved directly using a type of finite energy property for $\be^{\intens,\lambda}.$ Indeed, on the event that $B(r)$ is connected to $\partial\D$ in $\V^{\intens,\lambda}$ and that $B(r)\cap \be^{\intens,\lambda}=\varnothing,$ we have that $0$ is connected to $\partial\D$ in $\V^{\intens}.$  Since the probability that $B(r)\cap \be^{\intens,\lambda}=\varnothing$ is positive (this is simply the probability that $\be^{\intens}$ avoids all the loop clusters hitting $B(r),$ which is compact), we thus obtain by the FKG inequality that 
\begin{equation*}
    c\Pm(B(r)\stackrel{\CV^{\intens,\lambda}}{\longleftrightarrow}\partial\D)\leq\Pm(0\stackrel{\CV^{\intens,\lambda}}{\longleftrightarrow}\partial\D)\leq\Pm(B(r)\stackrel{\CV^{\intens,\lambda}}{\longleftrightarrow}\partial\D),
\end{equation*}
for some constant $c=c(u,\lambda,r),$ and we can conclude. 
\end{enumerate}
\end{remark}

\section{Coupling random walk and Brownian excursions}\label{s.coupling}

In this section, we prove in Theorem~\ref{the:couplingdiscontexc1} that one can couple the random walk excursion cloud, defined above \eqref{eq:VunDef}, with the Brownian excursion cloud, defined in \eqref{e.omega_alpha}. Convergence for the corresponding excursion measures, for excursions starting on some arc $\Gamma\subset\partial\D$ and ending on some disjoint arc $\Gamma'\subset\partial\D,$ was first proved in \cite{kozdron_scaling_2006}. Convergence for the full excursion set has been proved in \cite[Lemma~4.6.2]{ArLuSe-20a}. However, these two results are not explicit at all on the exact distance between the random walk and the Brownian motion excursions, as well as on the exact distance from the boundary of $\D$ at which this coupling is valid. They are therefore not adequate for proving quantitative results for percolation near the boundary such as \eqref{eq:percowithoutloops}, see also Remark~\ref{rk:connectiontoboundary}. When the domain is the unit disk $\D,$ we remedy this problem in this section, as well as present a more detailed proof of the coupling in \cite{ArLuSe-20a}.

There are three main steps in our proof of the coupling between the random walk excursions and the Brownian excursions. The first is to show that with high probability, the same number of trajectories from random walk excursions and Brownian excursions hit a fixed ball $B,$ see Lemma~\ref{l.capconv}. The second is to show that we can couple the corresponding hitting distributions, see Lemma~\ref{l.EQcoup}. The third is to show that the trajectories inside the ball $B$ can be coupled so that they are close to each other, which is done in Appendix~\ref{sec:KMT}. Moreover, in Theorems~\ref{the:couplingloops} and \ref{the:couplingwithloops}, we explain how to add clusters of loops to the coupling using results from \cite{MR2255196,ArLuSe-20a}. 

\subsection{Convergence of capacities and equilibrium measures}
Recall that for a compact set $K \Subset \D$ the (continuous) equilibrium measure of $K$ is denoted by $e_K$ while for $K' \subset \D_n$ we denote the discrete equilibrium measure by $e_{K'}^{(n)},$ see \eqref{defequiandcap} and \eqref{e.eqmeas}. Combining  \cite[Theorem~4.4.4, (6.11) and Proposition~6.3.5]{MR2677157} with \eqref{eq:killedGreen} and \eqref{eq:Greenat0} we see that for all $y\in{\D_n},$
\begin{equation}\label{e.greenest}
 | G(0,y)-G^{(n)}(0,y)| \leq O \left(\frac{1}{|y|n} \right).
\end{equation}
Our first result states that the sequence of discrete approximations of the  capacity of a ball at mesoscopic distance from the boundary indeed converges to the capacity of the continuous ball.

\begin{lemma}\label{l.capconv}
There exists $c<\infty$ such that for all $n\in\N$ and $r\in{(2/n,1)},$  
 \begin{equation}
\label{eq:approcap}
\left| \capac^{(n)}(B_n(r))-\capac(B(r)) \right| \leq c \frac{\capac^{(n)}(B_n(r))\capac(B(r))}{rn}.
\end{equation}
\end{lemma}
\begin{proof}
Abbreviate $B=B(r) \subset \D$ and $B_n = B_n(r).$ First, recall the last exit decomposition for the simple random walk \eqref{eq:lastexitdis}, which implies
\begin{align*}
1=\Pm_0^{(n)}\big(L_{B_n}^{(n)}>0\big)= \sum_{y \in \widehat{\partial} B_n} G^{(n)}(0,y) e_{B_n}^{(n)}(y).
\end{align*}
Using \eqref{eq:Greenat0} and \eqref{capball} we thus obtain by rearranging 
\begin{align*}
\frac1{2\pi}\log(1/r)\left(\mathrm{cap}(B)-\mathrm{cap}^{(n)}(B_n)\right)&=\sum_{y\in{\widehat{\partial}B_n}}e_{B_n}^{(n)}(y)\left(G^{(n)}(0,y)-\frac1{2\pi}\log(1/r)\right)
\\&=\sum_{y\in{\widehat{\partial}B_n}}e_{B_n}^{(n)}(y)\left(G^{(n)}(0,y)-G(0,y)\right)
\\&\quad+\frac1{2\pi}\sum_{y\in{\widehat{\partial}B_n}}e_{B_n}^{(n)}(y)\log(r/|y|).
\end{align*}
Using \eqref{e.greenest} and the fact that $||y|-r|\leq 1/n$ for all $y\in{\widehat{\partial}B_n}$ we thus obtain taking absolute values that
\begin{align*}
\left|\mathrm{cap}(B)-\mathrm{cap}^{(n)}(B_n)\right|\leq \frac{c\,\mathrm{cap}^{(n)}(B_n)}{rn\log(1/r)},
\end{align*}
and we can conclude by \eqref{capball}.
\end{proof}

\begin{remark}
If we let $B= B(1- \epsilon_n)$ and $B_n=B_n(1-\epsilon_n)$, where $\epsilon_n \to 0$ as $n \to \infty, $ then \eqref{capball} and the last lemma~provides an estimate of the form 
\[
\left| \capac^{(n)}(B_n) - \capac(B) \right| \leq c \frac{1}{n\epsilon_n^2},
\]
implying that the capacities are close as long as $\epsilon_n\sqrt{n}\gg 1$.
\end{remark}
As a corollary~we obtain an estimate of the total variation norm for the last exit distribution and the normalized equilibrium measure in the case of  balls. 
\begin{corollary}\label{cor.LEEQCoup}
There exists $c<\infty$ such that for all $n\in\N,$ $r\in{(2/n,1)}$ and $y\in{\widehat{\partial}B_n(r)}$ 
\begin{equation*}
    \left|\Pm_0^{(n)}\left( X_{L_{B_n}} =y \right) - \overline{e}_{B_n}^{(n)}(y)\right|\leq   \frac{c{e}^{(n)}_{B_n}(y)}{rn},
\end{equation*}
where $B_n=B_n(r).$ In particular,
\begin{equation}
\label{eq:equiexit}
    \left\|  \Pm_0^{(n)}\left( X_{L_{B_n}} \in \cdot  \right)-\overline{e}^{(n)}_{B_n}(\cdot) \right\|_{TV}\leq  \frac{c\,\mathrm{cap}^{(n)}(B_n)}{rn}.
\end{equation}
\end{corollary}
Note that the corresponding statement for the Brownian motion is exact, see \eqref{e.eqmeasexpr1} and \eqref{capball}. 
\begin{proof}
The proof follows by \eqref{e.greenest} in combination with Lemma~\ref{l.capconv} as follows. By \eqref{e.greenest}, \eqref{eq:Greenat0} and \eqref{capball}, together with the fact that  $|y| \in(r-2/n,r+2/n) $  we have 
\begin{equation}\label{e.greenapprox}
\left|G^{(n)}(0,y)-\frac1{\mathrm{cap}(B)}\right| \leq \left|G^{(n)}(0,y)-G(0,y)\right|+\frac1{2\pi}|\log(r/|y|)|=O\left( \frac{1}{rn} \right).
\end{equation}
Moreover, in light of Lemma~\ref{l.capconv} we deduce that 
\begin{equation}\label{e.capquot}
    \left|\frac{1}{\mathrm{cap}^{(n)}(B_n)}-\frac1{\mathrm{cap}(B)}\right|\overset{\eqref{eq:approcap}}{=}O\left( \frac{1}{rn} \right).
\end{equation}
Combining \eqref{eq:lastexitdis}, \eqref{e.greenapprox} and \eqref{e.capquot} we obtain
\begin{align*}
\left|\Pm^{(n)}_0\left( X_{L_{B_n}^{(n)}}=y \right)-\overline{e}^{(n)}_{B_n}(y)\right|&= \left|G^{(n)}(0,y)  -\frac1{\capac^{(n)}(B_n)}\right|{e}^{(n)}_{B_n}(y) \\    & {=} O\left( \frac{{e}^{(n)}_{B_n}(y)}{rn} \right),
\end{align*}  
and \eqref{eq:equiexit} follows by summing over $y\in{\widehat{\partial}B_n}.$
\end{proof}

We now combine Corollary~\ref{cor.LEEQCoup} with a coupling result between the last exit time of a ball by a Brownian motion and a random walk, see Lemma~\ref{l.gtype}, to obtain the desired coupling between the normalized equilibrium measures.

\begin{lemma}\label{l.EQcoup}
There exist $s_0>0$ and $c<\infty$ such that for all $r\in{(\frac12,1)}$ and $n\in{\N},$ writing $B=B(r)$ and $B_n=B_n(r),$ there exists a coupling $\Q_r$ between random variables $E_{B_n}^{(n)}$ and  $E_B$  with  distributions $\overline{e}^{(n)}_{B_n}$ and $\overline{e}_B,$ respectively, satisfying for all $s\geq s_0$ that
\begin{equation}
\label{e.boundapproequi}
    \Q_r \left(\left| E_{B_n}^{(n)}- E_B\right| \geq \frac{s \log(n)}{n}\right) \leq \frac{c}{s}+ \frac{c\log(n)}{n(1-r)}.
\end{equation}
\end{lemma}
\begin{proof}
Assume that $X^{(n)}$ and $Z$ 
at times ${{L}_{B_n}^{(n)}}$ and ${L_{B }}$, respectively,
are coupled as in Lemma~\ref{l.gtype} under $\Pm_{0,0}^{(n)}.$ We let $E_B{=} Z_{L_B},$ which has law $\overline{e}_B$ by rotational invariance. By Corollary~\ref{cor.LEEQCoup}, up to increasing $\Pm_{0,0}^{(n)}$ to a bigger probability space with probability measure denoted by $\Q_r,$ there is a coupling of a random variable $E_{B_n}^{(n)}$ with $X_{L_{B_n}^{(n)}}^{(n)}$  such that $E_{B_n}^{(n)}$ has law $\overline{e}_{B_n}^{(n)}$ and
\[
\Q_r \left( E_{B_n}^{(n)} \neq  X_{L_{B_n}^{(n)}}^{(n)} \right) \leq c \frac{\capac^{(n)}(B_n)}{n}.
\]
Using \eqref{e:boundlastexit}, we can easily conclude since $\capac^{(n)}(B_n)\leq c\,\capac(B)\leq c'/(1-r)$ by \eqref{eq:approcap} and \eqref{capball} if $r\geq1/2$, and since we can assume w.l.o.g.\ that $1-r\geq c/n$.
\end{proof}

\subsection{The coupling}
\label{sec:coupling}

Using Lemmas~\ref{l.capconv}, \ref{l.EQcoup} and \ref{lem:KMTuntilboundary}, we obtain the following coupling between the Brownian excursion cloud and the random walk excursion cloud. For each $r\in{[0,1)},$ $n\in\N$ and $\intens>0,$ we denote by 
\begin{equation}
\label{eq:defOmega}
    \Omega_{\intens}^{(n)}(r):=\left\{\big(e(2n^2t+\tau_{B_n(r)}^{(n)}(e))\big)_{t\in{\Big[0,\frac{t_e-\tau_{B_n(r)}^{(n)}(e)}{2n^2}\Big]}}:\,e\in{\text{supp}(\omega_{\intens}^{(n)})},\,\tau_{B_n(r)}^{(n)}(e)<\infty\right\}
\end{equation}
the set of all excursions in $\omega_u^{(n)}$ that hit $B_n(r)$  started from their location at the first time of hitting $B_n(r)$ with time rescaled by $2n^2.$ 
The value of $e(t)$ for non-integer $t$ is obtained by linear interpolation, and for $w\in{\Omega_{\intens}^{(n)}(r)}$ we take $w(t)=\Delta$ for all $t\geq (t_e-\tau_{B_n(r)}^{(n)}(e))/(2n^2),$ where $\Delta$ is some cemetery point. We similarly define $\Omega_{\intens}(r)$ as the set of trajectories in the support of $\omega_{\intens}$ hitting $B(r),$ started from the first time of hitting $B(r)$, and equal to $\Delta$ after hitting $\partial\D.$ We moreover take the convention $|\Delta-x|=\infty$ for all $x\in{\D}.$

\begin{theorem}
\label{the:couplingdiscontexc1}
There exist constants $c,C>0$ and $s_0<\infty$ such that for all $n\in\N,$  $1/2\leq r\leq 1-C/n,$ $u>0$ and $s\geq s_0$ there is a coupling between $\omega^{(n)}_{\intens}$ and $\omega_{\intens}$ such that on an event $\mathcal{E}_1^{(n)}$ with probability at least
\begin{equation}
\label{eq:probacoupling}
    1-\frac{c\intens}{(1-r)} \left( \frac{1}{s}+ \sqrt{\frac{\log(n)}{n(1-r)}}\right),
\end{equation}
there exists a bijection $F:\Omega_{\intens}(r)\rightarrow\Omega_{\intens}^{(n)}(r)$ such that
\begin{equation}
\label{eq:boundcoupling}
    \sup_{t\in{[0,\overline{L}_{B(r)}(w)]}}|w(t)-F(w)(t)|\leq \frac{s\log(n)}{n}\text{ for all $w\in{\Omega_{\intens}(r)}$},
\end{equation}
where $\overline{L}_{B(r)}(w):=\sup\{t\geq0:|w(t)|\wedge|F(w)(t)|\leq r\}$ is the last time at which either $w$ or $F(w)$ are in $B(r).$
\end{theorem}

\begin{proof}
Abbreviate $B=B(r)$ and $B_n=B_n(r).$ By Lemma~\ref{l.EQcoup}, one can find an i.i.d.\ sequence of random variables $(\sigma_n^{(i)},\sigma^{(i)}),$ $i\in\N,$ such that for each $i\in\N$ the law of $\sigma_n^{(i)}$ is $\overline{e}_{B_n}^{(n)},$ the law of $\sigma^{(i)}$ is $\overline{e}_B,$ and for all $s\geq s_0,$
\begin{equation}
\label{eq:couplingentrancepoint}
    \Pm\Big(|{\sigma}_n^{(i)}-\sigma^{(i)}|\geq \frac{s\log(n)}{2n}\Big)\leq \frac{c}{s}+ \frac{c\log(n)}{n(1-r)}.
\end{equation}
Now, using the KMT coupling from Lemma~\ref{lem:KMTuntilboundary}, we can produce a sequence of independent pairs $\big(\widehat{X}^{(n,i)},Z^{(i)} \big),$ $i\in{\N},$ of rescaled simple random walks on $\D_n$ and Brownian motions on $\D$ such that for each $i\in\N,$ $\widehat{X}^{(n,i)}$ starts in ${\sigma}_n^{(i)},$ $Z^{(i)}$ starts in ${\sigma}^{(i)},$ and
\begin{equation}
\label{eq:KMTforallwalks1}
\Pm \left(\sup_{t\in{[0,\overline{L}_{B(r)}^{(i)}]}} |\widehat{X}^{(n,i)}_t-Z^{(i)}_t| \geq   \frac{s\log(n)}{n} \,\Big|\,|\sigma_n^{(i)}-\sigma^{(i)}|\leq \frac{s\log(n)}{2n}\right) \leq \frac{cs\log(n)}{n(1-r)},
\end{equation}
where $\overline{L}_{B(r)}^{(i)}:=\sup\{t\geq0:|\widehat{X}^{(n,i)}_t|\wedge|Z^{(i)}_t|\leq r\}.$

Let us now couple the number of discrete and continuous excursions. Since $1-r\geq C/n,$ for a large enough constant $C,$ then combining \eqref{capball} and \eqref{eq:approcap} one infers that $\mathrm{cap}^{(n)}(B_n)\leq2\mathrm{cap}(B).$  Therefore, using \eqref{eq:approcap} again, there exists a standard coupling between Poisson random variables $Y_n\sim \text{Poi}\left(\intens\mathrm{cap}^{(n)}(B_n)\right)$ and $Y\sim \text{Poi}\left(\intens\mathrm{cap}(B)\right)$ such that
\begin{equation}
\label{eq:1stcouplingPoisson}
\Pm \left( Y_n \neq Y \right)\leq \frac{c\intens\mathrm{cap}(B)^2}{n}.
\end{equation}
By \eqref{capball} we moreover have
\begin{equation}
    \label{eq:boundoncap1}
    \mathrm{cap}(B(r))=\frac{2\pi}{\log(1/r)}\leq\frac{2\pi}{1-r}.
\end{equation}
Using \eqref{eq:1stcouplingPoisson}, combining \eqref{eq:couplingentrancepoint}, \eqref{eq:KMTforallwalks1} and \eqref{eq:boundoncap1} with a union bound, noting that $Y$ has mean $u\mathrm{cap}(B(r))$ and that $\{{Z^{(i)}},\,i\in{\{1,\dots,Y\}}\}$ and $\{{\widehat{X}^{(n,i)}},\,i\in{\{1,\dots,Y_n\}}\}$ have respectively the same law as $\Omega_{\intens}(r)$ and $\Omega_{\intens}^{(n)}(r),$ we obtain a bijection $F$ satisfying \eqref{eq:boundcoupling} with probability 
\begin{equation*}
    1-\frac{c\intens}{(1-r)} \left( \frac{1}{s}+ \frac{s\log(n)}{n(1-r)}\right).
\end{equation*}
Noting that the probability to find a coupling satisfying \eqref{eq:boundcoupling} is increasing in $s,$ one can replace $s$ by $\sqrt{n(1-r)/\log(n)}$ whenever $s\geq\sqrt{n(1-r)/\log(n)},$ and we obtain the probability \eqref{eq:probacoupling}.
\end{proof}

\begin{remark}
For fixed $r,u>0,$ Theorem~\ref{the:couplingdiscontexc1} allows us to approximate continuous excursions by discrete excursions with high probability as $n\rightarrow\infty$ for $s$ large enough. However, if $r=r(n),$ our approximation is only valid with high probability as $n\rightarrow\infty$ only if $r\leq 1-(c^2\log(n)/n)^{1/3}$ for a large constant $c$ (and $s=(cn/\log(n))^{1/3}$ for instance) i.e., the coupling fails once one gets too close to the boundary of the unit disk.
If one is only interested in coupling the excursions in a small region $A\subset B(r),$ Theorem~\ref{the:couplingdiscontexc1} could be improved by changing the factor $u/(1-r)$ in \eqref{eq:probacoupling} by  $u\mathrm{cap}(A).$ 
\end{remark}

As a direct consequence of Theorem~\ref{the:couplingdiscontexc1}, one can moreover couple the discrete and continuous excursion sets. We denote by $d_H(A,A')$ the Hausdorff distance between two sets $A,A'\subset\D,$ that is $d_H(A,A')$ is the smallest $\delta>0$ such that $A\subset A'+B(\delta)$ and $A'\subset A+B(\delta).$

\begin{corollary}
\label{cor:couplingexcursions}
For all $n\in\N,$ $u>0$ and $s\geq s_0,$ there exists a coupling between $\tilde{\be}_n^{\intens}$ and $\be^{\intens}$ such that, 
\begin{equation*}
    d_H(\tilde{\be}_n^{\intens},\be^{\intens})\leq \frac{s\log(n)}{n}\text{ with probability at least } 1-\frac{cun}{s^{3/2}\log(n)}.
\end{equation*}
\end{corollary}
\begin{proof}
This is a direct consequence of Theorem~\ref{the:couplingdiscontexc1} for $r=1-s\log(n)/n,$ noting that $d_H(\be_n^u,\tilde{\be}_n^u)\leq 1/n,$ $\partial\D\subset B(\tilde{\be}^u_n,1/n)$ and $\partial\D\subset \overline{{\be}^u}$ a.s. 
\end{proof}

Let us now explain how to couple connected components of loop soups, following \cite{MR2255196} and \cite{MR3485399}. Recall the definition of the loop soups on the cable system $\widetilde{\D}_n$ introduced at the end of Section~\ref{sec:cableSystemExc}, whose restriction to $\D_n$ is the random walk loop soup introduced below \eqref{eq:defdisloop}, and of the Brownian loop soup from below \eqref{eq:defcontloops}. We call $e$ an excursion of the Brownian loop soup if there exists a loop $\ell$ in the Brownian loop soup and some stopping times $T_1<T_2$ for $\ell$ such that $e=(\ell(t))_{t\in{[T_1,T_2]}},$ and we denote by $\text{trace}(e)\subset\D$ the set $\{\ell(t),\,t\in{[T_1,T_2]}\}.$ Moreover two excursions $e=(\ell(t))_{t\in{[T_1,T_2]}}$ and $e'=(\ell'(t))_{t\in{[T_1',T_2']}}$ are called disjoint if either $\ell$ and $\ell'$ are two different loops, or $\ell=\ell'$ and $[T_1,T_2]\cap[T_1',T_2']=\varnothing.$  The following Theorem~is tailored to our purpose in Section~\ref{sec:discretecritpara}. 

\begin{theorem}
\label{the:couplingloops}
For each $\lambda>0,$ there exist constants $c,c'>0$ and for each $k\in{\N}$, there is a coupling between a Brownian loop soup on $\D$ and a cable system loop soup on $\widetilde{\D}_k$ at level $\lambda$, as well as an event $\mathcal{E}_2^{(k)}$ with probability at least $1-c/\sqrt{k},$ such that the following holds true. For each  disjoint family of excursions $(e_i)_{i\leq E},$ with $E\in{\N},$ in the Brownian loop soup such that
\begin{equation*}
    \mathcal{L}=\bigcup_{i\leq E}\text{trace}(e_i)
\end{equation*}
is connected, there exists $N\in{\N}$ (depending on the choice of the family $(e_i)_{i\leq E}$) so that for all $n\geq N$, on the event $\mathcal{E}_2^{(n)}$ there is a connected subset $\mathcal{L}^{(n)}$ of the trace on $\widetilde{\D}_n$ of the cable system loop soup satisfying
\begin{equation}
\label{eq:couplingloops}
    d_H(\mathcal{L},\mathcal{L}^{(n)})\leq \frac{c'\log(n)}{n}.
\end{equation}
\end{theorem}
\begin{proof}
By \cite[Corollary~5.4]{MR2255196} with $\theta\in{(5/3,2)}$, on an event $\mathcal{E}_2^{(n)}$ with probability at least $1-c/\sqrt{n},$ we can couple each Brownian loop soup $\ell$ with time duration $t_{\ell}\geq n^{-1/3}$ with a time-changed discrete-time random walk loop $\ell^{(n)}$ with time duration $t_{\ell^{(n)}}\geq n^{-1/3}$ such that 
\begin{equation}
\label{eq:loopsareclose}
    \big|\ell(st_{\ell})-\ell^{(n)}(st_{\ell^{(n)}})\big|\leq \frac{c'\log(n)}{n}\text{ for all }s\in{[0,1]},
\end{equation}
where $\ell(s)$ is obtained for $s\notin{\N}$ by linear interpolation. Note that since $E<\infty,$ there exists $\delta>0$ such that the time duration of each loop in $\mathcal{L}$ is at least $\delta,$ and we assume from now on that $n$ is large enough so that $n^{-1/3}\leq \delta.$ If $e_i=(\ell_i(s))_{s\in{[T_{i,1},T_{i,2}]}},$ we write $e_i^{(n)}=(\ell_i^{(n)}(st_{\ell_i^{(n)}}/t_{\ell_i}))_{s\in{[T_{i,1},T_{i,2}]}}.$ Let us denote by $\mathcal{L}^{(n)}$ the union of the traces of the cable system excursions corresponding to $(e_i^{(n)})_{i\leq E}.$ Using \cite[(2.4)]{LJ-11}, one knows that discrete time loops from \cite{MR2255196} and continuous time loops from \eqref{eq:defdisloop} have the same trace on $\D_n,$ and since the distance between a random walk soup and its corresponding cable system loop is at most $1/n,$  \eqref{eq:couplingloops} clearly holds.

We now show that $\mathcal{L}^{(n)}$ is connected for $n$ large enough. Proceeding similarly as in the proof of \cite[Lemma~2.7]{MR3485399}, the following is a consequence of the conditions $\mathcal{C}_j,$ ${j\in{\N}},$ in \cite[Lemma~2.6]{MR3485399}: for each loop $\ell$ in the Brownian loop soup, and each stopping time $T$ for $\ell,$ there exists a.s.\ a sequence $(\eps_j)_{j\in{\N}}$ decreasing to $0$ such that, for any continuous functions $f:[0,t_{\ell}]\rightarrow\D$ with $\|f\|_{\infty}\leq \eps_j/12$ and any connected paths $\gamma$ such that $B(\ell(T),\eps_j/2)\leftrightarrow B(\ell(T),\eps_j)^{\ch}$ in $\gamma,$ we have $\gamma\cap A\neq\varnothing,$ where $A=\{\ell(s)+f(s):\,s\in{[T,H]}\}$ and $H=\inf\{t\geq T:\ell(t)\in{B(\ell(T),\eps_j)^{\ch}}\}.$ In other words, if $\gamma$ is a path starting close to $\ell(T)$ going far enough from $\ell(T),$ and $A$ is a set close enough to $\ell$ in a neighborhood of $\ell(T),$ then $\gamma$  intersects $A.$ 

Since $\mathcal{L}$ is connected, there exists for each $i,j\in{\{1,\dots,E\}}$ a sequence $k_1,\dots,k_p$ such that $k_1=i,$ $k_p=j$ and $\text{trace}(e_{k_i})\cap\text{trace}(e_{k_{i+1}})\neq\varnothing$ for each $i<p.$ For each $k,k'\in{\{1,\dots,E\}}$ such that $\text{trace}(e_{k})\cap\text{trace}(e_{k'})\neq\varnothing,$ define $H_{k,k'}$ as the hitting time of $\text{trace}(e_{k'})$ by $e_{k}.$ There exists a sequence $(\eps_j^{(k,k')})_{j\in{\N}}$ decreasing to $0,$ such that for each $j\in{\N}$ with $\text{trace}(e_{k'})\cap B(e_{k}(H_{k,k'}),2\eps_j^{(k,k')})^{\ch}\neq\varnothing,$ we have $\text{trace}(e_{k}^{(n)})\cap\text{trace}(e_{k'}^{(n)})\neq\varnothing$ if $c'\log(n)/n\leq \eps_j^{(k,k')}/2,$ since then $B(e_{k}(H_{k,k'}),\eps_j^{(k,k')}/2)\leftrightarrow B(e_{k}(H_{k,k'}),\eps_j^{(k,k')})^{\ch}$ in $\text{trace}(e_{k'}^{(n)})$ by \eqref{eq:loopsareclose}. Fixing for each $k,k'$ some $j$ large enough so that $\text{trace}(e_{k'})\cap B(e_{k}(H_{k,k'}),2\eps_j^{(k,k')})^{\ch}\neq\varnothing,$ and $n$ large enough so that $c'\log(n)/n\leq \eps_j^{(k,k')}/2$ uniformly in $k,k'$ (there are at most $E^2$ such $k,k'$), we thus have $\text{trace}(e_{k_i}^{(n)})\cap\text{trace}(e_{k_{i+1}}^{(n)})\neq\varnothing$ for each $i<p,$ and thus the set $\mathcal{L}^{(n)}$ is connected for $n$ large enough.
\end{proof}

Recall the definition of the sets of excursion plus loops $\tilde{\be}^{u,\lambda}_n$ from Section~\ref{sec:cableSystemExc} and $\be^{u,\lambda}$ from below \eqref{eq:defcontloops}. Combining Theorems~\ref{the:couplingdiscontexc1} and \ref{the:couplingloops}, one can show that each loop soup cluster $\mathcal{L}$ hitting $\be^{u}$ can be approximated by a set $\mathcal{L}^{(n)}$ of cable system loops which is connected and intersects  $\tilde{\be}^{u,\lambda}_n$ for $n$ large enough. In other words, one can show that $\be^{u,\lambda}\subset B(\tilde{\be}^{u,\lambda}_n,\eps_n)$ for $n$ large enough and a sequence $\eps_n\rightarrow0.$ However, the reverse inclusion is more difficult to obtain, since the clusters of small loops on the cable system cannot be well approximated by Brownian motion loops, and thus might be asymptotically strictly larger than the Brownian motion loop clusters. This problem is solved in  \cite{ArLuSe-20a} using the non-percolation of the loop soup clusters on general domains, see \cite[Lemma~4.13]{ArLuSe-20a}. More precisely, recalling the definition of $\tilde{\be}^{u,\lambda}_n$ from Section~\ref{sec:cableSystemExc} and of $\be^{u,\lambda}$ from below \eqref{eq:defcontloops}, the following follows from \cite[Proposition~4.11]{ArLuSe-20a} and Skorokhod's representation theorem.

\begin{theorem}[\cite{ArLuSe-20a}]
\label{the:couplingwithloops}
For each $u,\lambda>0,$ there exist for all $n\in\N$ a coupling of $\be^{\intens,\lambda}$ and $\widetilde{\be}_n^{\intens,\lambda},$ such that $d_H(\be^{\intens,\lambda},\widetilde{\be}_n^{\intens,\lambda})\rightarrow0$ as $n\rightarrow\infty.$ 
\end{theorem}

Note that, contrary to Corollary~\ref{cor:couplingexcursions} and Theorem~\ref{the:couplingloops}, the coupling from Theorem~\ref{the:couplingwithloops} does not give explicitly the rate at which $\tilde{\be}_n^{u,\lambda}$ converges to  $\be^{u,\lambda}.$ One could try to make the arguments from \cite{ArLuSe-20a} more explicit in $n,$ but this would not lead to any significant improvement of our main results, see Remark~\ref{rk:polycloseboundary}.

\section{Discrete critical values}
\label{sec:discretecritpara}
In this section, we prove that the discrete percolation parameters are asymptotically the same as the critical values obtained for the continuous percolation in Appendix~\ref{sec:contperco}. Our main tool will be the couplings between the continuous and discrete models from the previous section, that is Theorem~\ref{the:couplingdiscontexc1} for excursions, Theorem~\ref{the:couplingloops} for loop soups, and Theorem~\ref{the:couplingwithloops} for sets of excursions plus loops.  We first recall the Beurling estimate, which will also be useful in Section~\ref{s:gfflvlperc}.

\begin{lemma}
\label{lem:beurling}
There exists $c<\infty$ such that for all $n\in{\N},$ all connected sets $A\subset \D_n,$ $R>0$ and  $x\in{\D_n}$ with $\varnothing\neq A\cap\partial B_n(x,R)\subset{\D}_n,$
\begin{equation}
\label{eq:disbeurling}
    \Pm_x^{(n)}\big(\tau_A>\tau_{\partial B_n(x,R)}\big)\leq c\left(\frac{d(x, A)}{R}\right)^{1/2}.
\end{equation}
Moreover, the same result holds for the continuum case, that is, when replacing $\D_n$ by $\D,$ $B_n(x,t)$ by $B(x,R),$ and $\Pm_x^{(n)}$ by $\Pm_x.$ 
\end{lemma}
A proof of Lemma~\ref{lem:beurling} can be found in \cite[Lemma~2.3]{MR1249129} for random walks, and is in fact a consequence of \cite[Lemma~2.5.2]{MR2985195}; and in \cite[Proposition~3.79]{lawler2005conformally} for Brownian motion. We now present an application of the Beurling estimate for certain sets, tailored to our future purposes.  For $r\in [0,1),$ $r'\in{(r,1)},$ $\eps\in{(0,(1-r')/6)},$ and $j\in{\{0,1,2\}}$ let 
\begin{equation}
    \label{eq:defHi+}
    \begin{split}
        \H^{+}_{j,\eps}:=&\Big\{x\in{{B}(1-(j+1)\eps)\setminus \overline{B}(r+j(r'-r)/2)}:\,\text{arg}(x)\in{\big((j-7)\pi/28,-j\pi/28\big)}\Big\}
        \\\cup&\Big\{x\in{{B}(1-(j+1)\eps)\setminus \overline{B}(r+j(r'-r)/2)}:\,\text{arg}(x)\in{\big(\pi+j\pi/28,\pi+(7-j)\pi/28\big)}\Big\}
        \\\cup&\Big\{x\in{{B}(1-(j+4)\eps)\setminus \overline{B}(r+j(r'-r)/2)}:\,\text{arg}(x)\in{\big(-j\pi/28,\pi+j\pi/28\big)}\Big\}.
    \end{split}
\end{equation}
Note that $\H^{\pm}_{2,\eps}\subset\H^{\pm}_{1,\eps}\subset\H^{\pm}_{0,\eps}$ and that we made the dependency of $\H^{\pm}_{j,\eps}$ on $r,r'$ implicit to simplify notation.  Moreover, for $j\in{\{0,2\}}$ we let
\begin{equation}
\label{eq:defS+}
\begin{gathered}
    S_{j,\eps}^{+,{r}}:=\Big\{x\in{\partial B(1-(4-j/4)\eps)}:\,\text{arg}(x)\in{\big((j-7)\pi/28,-j\pi/28\big)}\Big\},
    \\    S_{j,\eps}^{+,{l}}:=\Big\{x\in{\partial B(1-(4-j/4)\eps)}:\,\text{arg}(x)\in{\big(\pi+j\pi/28,\pi+(7-j)\pi/28\big)}\Big\},
    \\\overline{S}_{j,\eps}^{+,r}:=\Big\{x\in{B(1-(j+4)\eps)\setminus \overline{B}(r+j(r'-r)/2)}:\,\text{arg}(x)=-j\pi/28\Big\},
    \\    \overline{S}_{j,\eps}^{+,l}:=\Big\{x\in{B(1-(j+4)\eps)\setminus \overline{B}(r+j(r'-r)/2)}:\,\text{arg}(x)=\pi+j\pi/28\Big\}.
    \end{gathered}
\end{equation}
We also define the sets $\H_{j,\eps}^-$, $S_{j,\eps}^{-,l}$, $S_{j,\eps}^{-,r}$, $\overline{S}_{j,\eps}^{-,l}$ and $\overline{S}_{j,\eps}^{-,r}$  as the respective reflections through $\R$ of the sets $\H_{j,\eps}^+$, $S_{j,\eps}^{+,r}$, $S_{j,\eps}^{+,l}$, $\overline{S}_{j,\eps}^{+,r}$ and $\overline{S}_{j,\eps}^{+,l}$; here, superscripts $l$ and $r$ stand for left and right, and in particular should not be confused with the radius $r$. We refer to Figure \ref{F:discretecritpar} below for an illustration of all these sets when $r=0$. For $A\subset \D$, $j\in{\{0,2\}}$ and $\pm\in{\{-,+\}}$ let us introduce the events 
\begin{equation}
\label{eq:defFi+-}
\begin{gathered}
    F_{j,\eps}^{\pm}(A)=\left\{
    \text{$A$ contains a path included in $\H_{j,\eps}^{\pm}$ and hitting both $S_{j,\eps}^{\pm,r}$ and $S_{j,\eps}^{\pm,l}$}
    \right\},
    \\    \overline{F}_{j,\eps}^{\pm}(A)=\left\{
    \text{$A$ contains a path included in $\H_{j,\eps}^{\pm}$ and hitting both $\overline{S}_{j,\eps}^{\pm,r}$ and $\overline{S}_{j,\eps}^{\pm,l}$}
    \right\}.
\end{gathered}
\end{equation}
In the rest of the section, when we write that $F_{j,\eps}^{\pm}$ satisfies some property, it means that both $F_{j,\eps}^+$ and $F_{j,\eps}^-$ satisfy this property. Let us quickly explain the reason for introducing these events, and the strategy of the proof of Theorem~\ref{the:maindisexcloops}. At first glance, it might seem enough, in view of Theorem~\ref{t.mainthm} and the couplings from Section~\ref{sec:coupling}, to prove  that equivalently on the continuum and on the cable system, there is a path of excursions plus clusters of loops in $\{x\in{\D}:\,\pm\Im(x)\geq0\}$ which hits both $\{x:\text{arg}(x)=\pi\}$ and $\{x:\text{arg}(x)=0\}$. However, one additionally needs that the union of the previous excursions plus clusters of loops in the upper and lower part of the disk  form together a surface blocking $0$ from $\partial \D$, which is not always the case (think for instance of a spiral around $0$). To avoid this problem, it will be easier to consider the events $F_{j,\eps}^\pm$ and $\overline{F}_{j,\eps}^{\pm}$ instead, as we now explain.

 When $u\geq (8-\kappa)\pi/16$, one can show by a similar reasoning as in the proof of Lemma~\ref{lem:2} that the event $F_{2,0}^{\pm}(\be^{u,\lambda})$ occurs a.s, and thus $F_{2,\eps}^{\pm}(\be^{u,\lambda})$ occurs with high probability for $\eps$ small enough. Under the appropriate couplings of discrete and continuous excursions and loops in $B(1-\eps)$ from Section~\ref{sec:coupling}, one deduces that $F_{0,\eps}^{\pm}(\tilde{\be}_n^{u,\lambda})$ also occurs with high probability, see \eqref{eq:approxconnection1}. Since for any $A\subset\D$ the event $F_{0,\eps}^{\pm}(A)$ implies that there is a path in $A$ disconnecting $B(r)$ from $\{x\in{\partial B(1-4\eps)}:\,\pm\Im(x)\geq0\}$ in $B(1-4\eps)$, see the left-hand side of Figure~\ref{F:discretecritpar}, we have that
\begin{equation}
\label{eq:interestofF}
\text{ if }F_{0,\eps}^{+}(A)\text{ and }F_{0,\eps}^{-}(A)\text{ both occur, then }B(r)\not\leftrightarrow \partial B(1-4\eps)\text{ in }A^{\ch}.
\end{equation}
By combining the previous observations we can conclude the proof in the case $u\geq (8-\kappa)\pi/16$ of Theorem~\ref{the:maindisexcloops}. On the other hand, when $\intens<(8-\kappa)\pi/16$, by a similar reasoning as in the proof of Lemma~\ref{lem:1}, one can show that $\overline{F}_{0,0}^\pm(\be^{u,\lambda})^{\ch}$ occurs with positive probability, and thus $\overline{F}_{0,\eps}^\pm(\be^{u,\lambda})^{\ch}$ as well for $\eps$ small enough. Moreover, under a similar coupling as before, see \eqref{eq:approxconnection2}, we deduce that $\overline{F}_{2,\eps}^\pm(\tilde{\be}_n^{u,\lambda})^{\ch}$ also occurs with positive probability. Finally by the observation that for any $A\subset\D$
\begin{equation}
\label{eq:interestofFbar}
    \text{ if }\overline{F}_{2,\eps}^{\pm}(A)^{\ch}\text{ occurs, then }B(r')\leftrightarrow \partial B(1-6\eps)\text{ in }A^{\ch},
\end{equation}
see the right-hand side of Figure~\ref{F:discretecritpar}, we are able to finish the proof of Theorem~\ref{the:maindisexcloops}. Note that in the proof of Theorem~\ref{the:maindisexcloops}, we will actually use an easier argument based on the coupling from Theorem~\ref{the:couplingwithloops} in the case $\intens<(8-\kappa)\pi/16$, but the previous reasoning relying on \eqref{eq:interestofFbar} will still be useful when trying to get precise control on the dependency of $\eps$ on $n$ in the case $\lambda=0$ from \eqref{eq:percowithoutloops}.

Let us now give the details of the previous strategy. One of the main obstacles is to prove that if there is a path of connected excursions and loop clusters in the continuum, then there is also such a similar connected path in the cable system. Indeed, the couplings from Section~\ref{sec:coupling} only imply that when two continuous excursions, or an excursion and a loop cluster, intersect each other, then the cable system excursions, or excursion and loop cluster, are close to one another, but do not necessarily intersect each other. As we shall see in Lemma~\ref{Lem:simpleRWresult}, using Lemma~\ref{lem:beurling}, one can actually show that they will intersect each other with high probability. 

Moreover, for each $n\in\N,$ denote by ${X}^{\pm}$ under $\Pm_x^{(n)},$ $x\in{\H_{1,\eps}^{\pm}},$ the trace on ${\D}_n$ of the random walk ${X}$ on  ${\D}_n,$ killed on hitting $(\H_{0,\eps}^{\pm})^{\ch}.$ We define similarly $Z^{\pm}$ as the trace on $\D$ of the Brownian motion $Z$ on $\D$ killed on hitting $(\H_{0,\eps}^{\pm})^{\ch}.$
\begin{lemma}
\label{Lem:simpleRWresult}
    There exists a constant $c<\infty$ such that for all $r\in [0,1),$ $r'\in{(r,1)},$  $\eps\leq c'$ (for some constant $c'>0$ depending only on $r,r'$), $\delta>0,$ $n\in\N,$ $x\in{\H_{0,\eps}^{\pm}},$ and any connected set $\mathcal{C}\subset\H_{0,\eps}^{\pm}\cap \D_n$ with diameter at least $\eps/4$ intersecting $\H_{1,\eps}^{\pm},$ 
	\begin{equation}
	\label{eq:simpleRWresult}
	    \Pm_x^{(n)}\Big({X}^{(\pm)}\cap \mathcal{C}=\varnothing,\ d\big(X^{(\pm)},\, \mathcal{C}\cap\H_{1,\eps}^{\pm}\big)\leq\delta\Big)\leq c\left(\frac{\delta}{\eps}\right)^{1/2}.
	\end{equation}
	Moreover, the same result still holds for the continuum setting $n=\infty,$ that is when replacing $\D_n$ by $\D,$ $\Pm_x^{(n)}$ by $\Pm_x$ and $X^{\pm}$ by $Z^{\pm}.$
\end{lemma}
\begin{proof}
We do the proof for the random walk $X$ on  ${\D}_n$; the proof for the continuum case proceeds analogously. Let $H_{\delta}$ be the first time $X$ hits $B_n(\mathcal{C}\cap\H_{1,\eps}^{\pm},\delta),$ following the standard notation 
$H_{\delta} = \infty$ if $X \cap B_n(\mathcal{C}\cap\H_{1,\eps}^{\pm},\delta) = \varnothing.$ Without loss of generality, in order to prove \eqref{eq:simpleRWresult} we can and will assume from now on that $H_{\delta}<\infty$ and $\delta<\eps/40.$
Note that $\H_{1,\eps}^{\pm}\neq\varnothing$ and  $B_n(\H_{1,\eps}^{\pm},\eps/10)\subset \H_{0,\eps}^{\pm}$ for $\eps$ small enough, depending on the choice of $r,r',$ and that $\mathcal{C}\cap \partial B_n(X(H_{\delta}),\eps/20)\neq\varnothing.$ Due to the Markov property, conditionally on $H_{\delta}$ and $X(H_{\delta}),$  the process $(X(t))_{t\geq H_{\delta}}$ until the first time it exits $B_n(X(H_{\delta}),\eps/20)(\subset\H_{0,\eps}^{\pm})$ has the same law as a random walk on $\D_n$ started in $X(H_{\delta})$ until the first time it exits $B_n(X(H_{\delta}),\eps/20).$ As a consequence,  by the Beurling estimate, see Lemma~\ref{lem:beurling}, we have that there exists a constant $c<\infty$ such that for all $n\in\N,$
\begin{equation*}
    \Pm_x^{(n)}\Big((X(t))_{t\geq H_{\delta}} \text{ leaves }B_n(X(H_{\delta}),\eps/20)\text{ before hitting }\mathcal{C}\,\big|\,H_{\delta},X(H_{\delta})\Big)\leq c\left(\frac{\delta}{\eps}\right)^{1/2}.
\end{equation*}
Integrating, \eqref{eq:simpleRWresult} follows.
\end{proof}

We begin with the following lemma, which essentially controls the probability to have a connection in $\V^{\intens,\lambda}$ but not in $\widetilde{\V}^{\intens,\lambda}_n,$ and vice versa.

\begin{lemma}
\label{lem:approxconnection}
For all $\intens>0,$ $\lambda\geq0$ and $\eps\in{(0,1)}$ there exists a coupling $\Q^{u,\lambda,\eps}_n$ between $\tilde{\V}_n^{u,\lambda}$ and $\V^{u,\lambda}$ such that the following holds: for all  $r\in [0,1-6\eps)$ and $r'\in{(r,1-6\eps)}$,
\begin{equation}
\label{eq:approxconnection1}
    \lim\limits_{n\rightarrow\infty}\Q^{u,\lambda,\eps}_n\big(F_{0,\eps}^{\pm}(\tilde{\be}_n^{u,\lambda})^{\ch},F_{2,\eps}^\pm({\be}^{u,\lambda})\big)=0,
\end{equation}
and for all $r\in [0,1)$ and $r'\in{(r,1)}$, letting $\eps_n=7n^{-1/7}$,
\begin{equation}
\label{eq:approxconnection2}
   \lim\limits_{n\rightarrow\infty}\Q^{u,0,\eps_n}_n\big(\overline{F}_{0,\eps_n}^{\pm}({\be}^{u})^{\ch},\overline{F}_{2,\eps_n}^\pm(\tilde{\be}_n^{u})\big)=0.
\end{equation}
\end{lemma}
\begin{proof}
We first describe how the coupling $\Q_n^{u,\lambda,\eps}$ is constructed, and then explain why this coupling satisfies \eqref{eq:approxconnection1} and \eqref{eq:approxconnection2}. First couple the discrete excursions $\be^{\intens}_n$ and the continuous excursion $\be^{\intens}$ as in Theorem~\ref{the:couplingdiscontexc1} for $s=\log(n)/\eps$ on an event $\mathcal{E}_1^{(n)}$ with probability at least $1-c(\log(n))^{-1}-\sqrt{c\log(n)/(n\eps^3)},$ for some large enough constant $c=c(u),$ so that, there exists
a bijection $F:\Omega_{\intens}(1-\eps)\rightarrow\Omega_{\intens}^{(n)}(1-\eps)$ (see \eqref{eq:defOmega} and below) such that
\begin{equation}
\label{eq:boundcoupling2}
    \sup_{t\in{[0,\overline{L}_{B(1-\eps)}(w)]}}|w(t)-F(w)(t)|\leq \frac{c\log(n)^2}{\eps n}=:f(n,\eps)\text{ for all $w\in{\Omega_{\intens}(1-\eps)}$}.
\end{equation}
We also couple the cable system excursion process $\tilde{\be}^u_n$ with $\be^u_n$ so that the trace of $\tilde{\be}^u_n$ on $\D_n$ is $\be^u_n,$ and for each $w\in{\Omega_u(1-\eps)},$ we denote by $\tilde{F}(w)$ the cable system excursion corresponding to $F(w),$ see Section~\ref{sec:cableSystemExc}. 
Moreover, if $\lambda\neq0,$ we couple the cable system loop soup and the Brownian loop soup at level $\lambda$ as in Theorem~\ref{the:couplingloops}, and in particular there is an event $\mathcal{E}_2^{(n)}$  with probability at least $1-o_n(1)$ under which \eqref{eq:couplingloops} is satisfied.

We start with the proof of \eqref{eq:approxconnection1}.  Let $(w_i)_{i\in{\{1,\dots,N_{\eps}\}}}$ be some enumeration of the Brownian excursions hitting $B(1-\eps),$ and let $w'_i$ be the part of $w_i$ in $\Omega_u(1-\eps).$ For each $i\leq N_{\eps},$ we decompose the cable system trajectory $\tilde{F}(w_i')$ into subexcursions in $\H_{0,\eps}^{+}$, starting and ending in $\partial \H_{0,\eps}^{+}$. We denote by $(E_{i,j}^{(n),+})_{j=1,\dots,K_i^{\pm}}$ the subexcursions which hit $\H_{1,\eps}^{+}$. We decompose the trajectory $\tilde{F}(w_i')$ again, this time into subexcursions in $\H_{0,\eps}^{-}$ starting and ending in $\partial \H_{0,\eps}^{-}$, and denote by $(E_{i,j}^{(n),-})_{j=1,\dots,K_i^{\pm}}$ the subexcursions which hit $\H_{1,\eps}^{-}$. We remark that some (parts) of the excursions $(E_{i,j}^{(n),-})_{j}$ and $(E_{i,j}^{(n),+})_{j}$ may coincide. Note that upon choosing $\eps>0$ small enough, $\partial\H_{0,\eps}^{\pm}$ and $\H_{1,\eps}^{\pm}$ are at positive distance. Therefore $K_i^{\pm}$ is a.s.\ finite, and  it is possible that $\H_{1,\eps}^{\pm}$ is never visited by $\tilde{F}(w_i'),$ and in this case $K_i^{\pm}=0.$ Note also that $\tilde{F}(w_i')$ could hit $B(r),$ but its trajectory inside ${B}(r)$ does not appear in the decomposition $(E_{i,j}^{(n),\pm})_{j=1,\dots,K_i^{\pm}}.$

For each  $i\in{\{1,\dots,N_{\eps}}\}$ and $j\in{\{1,\dots,K_i^{\pm}\}}$ let $E^{\pm}_{i,j}$ be the subtrajectory of $w_i',$ whose starting and ending time are the same as the ones of $E^{(n),\pm}_{i,j}$ for $\tilde{F}(w_i').$ Note that $E^{\pm}_{i,j}$ starts and end at random times which depend on the random walk excursions, and thus might not be Markovian. Then using \eqref{eq:boundcoupling2} and noting that cable system excursions and discrete excursions are at Hausdorff distance at most $1/n,$  on the event $\mathcal{E}_1^{(n)},$ $d_H\big(E^{\pm}_{i,j},E^{(n),\pm}_{i,j}\big)\leq f(n,\eps)$ (up to changing the constant $c$ in \eqref{eq:boundcoupling2}), where we identify trajectories with their trace on $\D,$ and the Hausdorff distance is defined above Corollary~\ref{cor:couplingexcursions}. Moreover, in view of \eqref{eq:defHi+}, if $\eps\geq cf(\eps,n)$ and $r'-r\geq cf(\eps,n),$ each excursion of $w_i$ in $\H_{2,\eps}^{\pm}$ is part of $E^{\pm}_{i,j}$ for some $j\leq K_i^{\pm}.$

If $\lambda>0,$  we moreover denote by $(e_j^{\pm})_{j\leq E_{\delta,\eps}^{\pm}}$ the excursions in $\H_{1,\eps}^{\pm}$ which intersect $\H_{2,\eps}^{\pm}$ and come from loops with diameter at least $\delta,$ in the Brownian loop soup at level $\lambda,$ where $\delta>0$ is a parameter we will choose in \eqref{eq:choicedelta}. 
The sets $(\text{trace}(e_j^{\pm}))_{j\leq E_{\delta,\eps}^{\pm}}$ form connected components in $\H_{1,\eps}^{\pm}$, that we denote by $(\mathcal{L}_i^{\pm})_{i\leq L_{\delta,\eps}^{\pm}}.$  Then, since our choice of $\delta$ will not depend on $n,$ see \eqref{eq:choicedelta}, on the coupling event $\mathcal{E}_2^{(n)},$ for $n$ large enough and for each $i\leq L_{\delta,\eps}^{\pm}$ there exists a.s.\ a connected set $\mathcal{L}_i^{(n),\pm}$ included in the trace on $\widetilde{\D}_n$ of the cable system loop soup, and such that $d_{H}(\mathcal{L}_i^{(n),\pm},\mathcal{L}_i^{\pm})\leq c\log(n)/n\leq f(n,\eps).$

Let $\mathcal{C}^{(n),\pm}_k,$ $k\leq C^{\pm}_{\eps,\delta},$ be some enumeration of the excursions $E_{i,j}^{(n),\pm},$ $i\in{\{1,\dots,N_{\eps}\}}$ and $j\in{\{1,\dots,K_i^{\pm}\}},$ plus (if $\lambda\neq0$)  the connected clusters of loop excursions $\mathcal{L}_i^{(n),\pm},$ $i\leq L_{\delta,\eps}^{\pm}.$ For each $k\in{\{1,\dots,C_{\eps,\delta}^{\pm}\}},$ there is a continuous excursion, or cluster of loops, $\mathcal{C}_k^{\pm}$ at Hausdorff distance at most $f(n,\eps)$ from $\mathcal{C}_k^{(n),\pm},$ and all the parts in  $\H_{2,\eps}^{\pm}$ of continuous loop clusters, of loops with diameter at least $\delta$, and of excursions are in some $\mathcal{C}_k^{\pm}$ if $\eps\wedge(r'-r)\geq cf(n,\eps).$ We define for each $i\in{\{1,\dots,N_{\eps}\}},$ $j\in{\{1,\dots,K_i^{\pm}\}}$ and $k\in{\{1,\dots,C^{\pm}_{\eps,\delta}\}}$ the events
\begin{equation}
\label{eq:defApmnij}
    A_{i,j,k}^{(n),\pm}:=\left\{d\big(E^{(n),\pm}_{i,j},\mathcal{C}^{(n),\pm}_k\cap\H_{1,\eps}^{\pm}\big)>2f(n,\eps)\right\}\cup\left\{E^{(n),\pm}_{i,j}\cap \mathcal{C}^{(n),\pm}_k\neq\varnothing\right\}
\end{equation}
and 
\begin{equation}
\label{eq:defApmn}
    A^{(n),\pm}_{\eps,\delta}:=\bigcap_{i\in{\{1,\dots,N_{\eps}\}}}\bigcap_{j\in{\{1,\dots,K_i^{\pm}\}}}\bigcap_{ k\in{\{1,\dots,C_{\eps,\delta}^{\pm}\}}}A_{i,j,k}^{(n),\pm}.
\end{equation}
Note that on this event, any two excursions, or any excursion and loop cluster, which are close enough will intersect each other. We will now argue that, on this event, the event $F_{2,\eps}^{\pm}$ for Brownian excursions plus loops imply $F_{0,\eps}^{\pm}$ for cable system excursions plus loops when they are close to each other, see \eqref{eq:ifAnepsdeltathendiscretedisconnection}. 
Let $\be^{\intens,\lambda,\delta}$ be the union of the clusters, consisting of continuous loops with diameter at least $\delta$ for the Brownian loop soup with intensity $\lambda,$ which hit $\be^{\intens}$. On the event $F_{2,\eps}^\pm({\be}^{u,\lambda,\delta}),$ see \eqref{eq:defFi+-}, there exist $M^{\pm}_{\eps,\delta}\in\N_0$ and $k_0,\dots,k_{M^{\pm}_{\eps,\delta}}$ such that 
\begin{enumerate}[i)]
    \item $\mathcal{C}_{k_{i-1}}^{\pm}\cap\mathcal{C}_{k_i}^{\pm}\cap \H_{2,\eps}^{\pm}\neq\varnothing$ for all $i\in\{1,\dots,M^{\pm}_{\eps,\delta}\},$
    \item if $\mathcal{C}_{k_{i-1}}^{\pm}$ is a connected cluster of loop excursions for some $i\in\{1,\dots,M^{\pm}_{\eps,\delta}\},$ then $\mathcal{C}_{k_i}^{\pm}$ is a Brownian excursion,
    \item $\mathcal{C}_{k_0}^{+}\cap S_{2,\eps}^{\pm,l}\neq\varnothing$ and $\mathcal{C}_{k_{M^{+}_{\eps,\delta}}}^{+}\cap S_{2,\eps}^{\pm,r}\neq\varnothing$.
\end{enumerate}

\begin{figure}[ht]
  \centering 
  \includegraphics[scale=0.76]{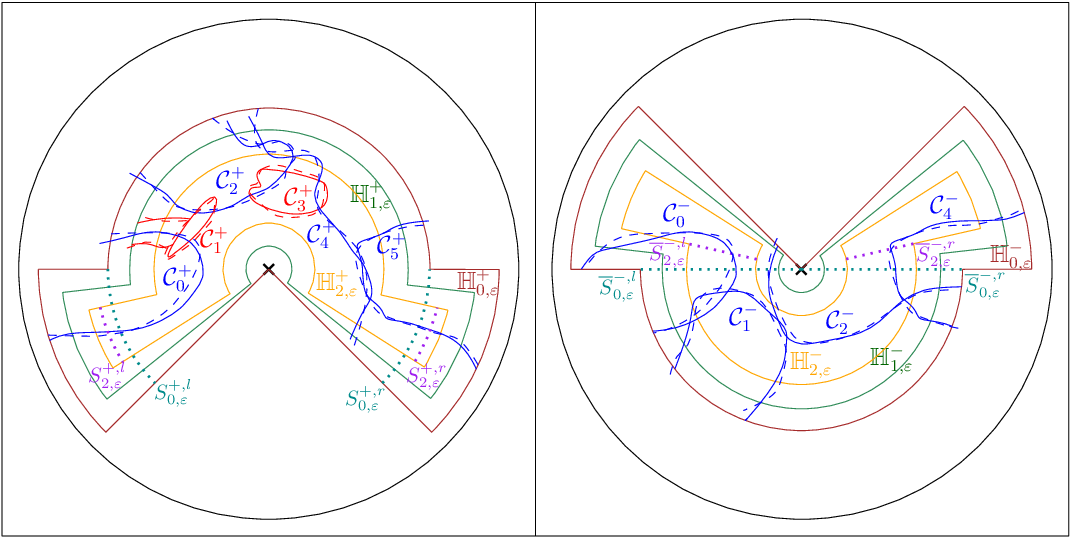}
 \caption{Blue lines correspond to excursions and red sets to loop clusters. Dashed lines represent the sets $\mathcal{C}_k^{(n),\pm}$, that is excursions and loop clusters  on the cable system $\widetilde{\D}_n,$ whereas continuous lines represent the sets $\mathcal{C}_k^{\pm}$, that is excursions and loop clusters on $\D.$ The dotted purple lines correspond to the sets $S_{j,\eps}^{+,l}$ and $S_{j,\eps}^{+,r}$ on the left, and to the sets $\overline{S}_{j,\eps}^{+,l}$ and $\overline{S}_{j,\eps}^{+,r}$ on the right, $j\in{\{0,2\}}$. On the left the events $F_{2,\eps}^+(\be^{u,\lambda,\delta})$ and $F_{0,\eps}^+(\tilde{\be}^{u,\lambda}_n)$ both occur. On the right, the event $\overline{F}_{2,\eps}^-(\be^{u})$ occurs, but not the event $\overline{F}_{0,\eps}^-(\tilde{\be}^{u}_n)$ since the dashed excursions $\mathcal{C}_1^{(n),-}$ and $\mathcal{C}_2^{(n),-}$ are close and do not intersect.} 
  \label{F:discretecritpar}
\end{figure}

We refer to the left-hand side of Figure~\ref{F:discretecritpar} for details. If $\eps\wedge (r'-r)\geq cf(n,\eps),$ see \eqref{eq:boundcoupling2}, for each $i\in\{1,\dots,M^{\pm}_{\eps,\delta}\},$ one can easily deduce from i) that $d\big(\mathcal{C}^{(n),\pm}_{k_{i-1}}\cap\H_{1,\eps}^{\pm},\mathcal{C}^{(n),\pm}_{k_i}\cap\H_{1,\eps}^{\pm}\big)\leq 2f(n,\eps)$ on the event $\mathcal{E}_1^{(n)}\cap\mathcal{E}_2^{(n)},$ and thus by ii) $\mathcal{C}^{(n),\pm}_{k_{i-1}}\cap \mathcal{C}^{(n),\pm}_{k_i}\neq\varnothing$ on the event $A_{\eps,\delta}^{(n),\pm}.$ On the intersection of these events, $\bigcup_{k=0}^{M^{\pm}_{\eps,\delta}}\mathcal{C}^{(n),\pm}_{k_i}$ is thus a set connected in $\H_{0,\eps}^{\pm}$ for $n$ large enough, and, by \eqref{eq:defS+} as well as iii), $\bigcup_{k=0}^{M^{+}_{\eps,\delta}}\mathcal{C}^{(n),+}_{k_i}$ intersects both $\{x\in{\D\setminus B(1-4\eps)}:\,\text{arg}(x)\in{[-\pi/4,0]}\}$ and $\{x\in{\D\setminus B(1-4\eps)}:\,\text{arg}(x)\in{[\pi,5\pi/4]}\}$. By a simple geometric argument, see Figure~\ref{F:discretecritpar}, one deduces that  $\bigcup_{k=0}^{M^{+}_{\eps,\delta}}\mathcal{C}^{(n),+}_{k_i}$ intersects both $S_{0,\eps}^{+,r}$ and $S_{0,\eps}^{+,l}$. Let us denote by $\mathcal{A}^{(n),+}_{\eps,\delta}$ the event that  $\bigcup_{k=0}^{M^{+}_{\eps,\delta}}\mathcal{C}^{(n),+}_{k_i}$ intersects $\tilde{\be}_n^u$, and is thus included in $\tilde{\be}_n^{u,\lambda}$ by definition.  Proceeding similarly for $\bigcup_{k=0}^{M^{-}_{\eps,\delta}}\mathcal{C}^{(n),-}_{k_i}$, for each $\delta>0,$ $\eps>0,$ and $n$ large enough so that $\eps\wedge(r'-r)\geq cf(n,\eps)$ we therefore have 
\begin{equation}
\label{eq:ifAnepsdeltathendiscretedisconnection}
    \mathcal{E}_1^{(n)}\cap\mathcal{E}_2^{(n)}\cap A_{\eps,\delta}^{(n),\pm}\cap \mathcal{A}_{\eps,\delta}^{(n),\pm}\cap F_{2,\eps}^\pm({\be}^{u,\lambda,\delta})\subset F_{0,\eps}^\pm(\widetilde{\be}_n^{u,\lambda}).
\end{equation}

We first consider the event $\mathcal{A}^{(n),\pm}_{\eps,\delta}$, which trivially always occurs on the event $\mathcal{E}_1^{(n)}\cap\mathcal{E}_2^{(n)}\cap A_{\eps,\delta}^{(n),\pm}$ in case $M_{\eps,\delta}^\pm\geq2$ since one of the sets $\mathcal{C}^{(n),\pm}_{k_i}$ is then included in $\tilde{\be}_n^u$, and their union is connected. On the previous event, the only possibility for the event $\mathcal{A}^{(n),\pm}_{\eps,\delta}$ to not occur is when $\mathcal{C}_{k_0}^{\pm}$ is a continuous loop cluster that intersects both $S_{2,\eps}^{\pm,r}$ and $S_{2,\eps}^{\pm,l}$. Since $\mathcal{C}_{k_0}^{\pm}\subset\be^{u,\lambda,\delta}$, the loop cluster $\mathcal{L}$ of loops in $\D$ with diameter at least $\delta$ which contains $\mathcal{C}_{k_0}^{\pm}$ intersects some excursion $w$ in $\omega_u$. The excursion $w$ belongs to $\Omega_u(1-n^{-1/4})$ for $n$ large enough, and is thus at distance less than $n^{-1/2}$ from a discrete excursion $w^{(n)}$ in $\Omega_u^{(n)}(1-n^{-1/4})$ with high probability as $n\rightarrow\infty$ by Theorem~\ref{the:couplingdiscontexc1}. Moreover, the set $\mathcal{L}$ is at distance less than $\log(n)/n$ from a connected set of discrete loops $\mathcal{L}^{(n)}$ with high probability as $n\rightarrow\infty$ by Theorem~\ref{the:couplingloops}. If the event $\mathcal{A}^{(n),\pm}_{\eps,\delta}$ does not occur, we then have that $w^{(n)}$ and $\mathcal{L}^{(n)}$ do not intersect each other, but since they are both at distance at most $n^{-1/2}$ from a given point which does not depend of $n$, we deduce from \eqref{eq:disbeurling} that for all $\delta,\eps>0$
\begin{equation*}
    \lim\limits_{n\rightarrow\infty}\mathbb{P}\big(\mathcal{E}_1^{(n)}\cap\mathcal{E}_2^{(n)}\cap A_{\eps,\delta}^{(n),\pm}\cap \big(\mathcal{A}^{(n),\pm}_{\eps,\delta}\big)^{\ch}\big)=0.
\end{equation*}

Let us now prove that the event $A_{\eps,\delta}^{(n),\pm}$ occurs with high probability. Each subexcursion $E_{i,j}^{(n),\pm}$ intersects $\H_{1,\eps}^{\pm}$ and thus has diameter at least $\eps/10$ and, assuming that $\delta\wedge\eps\geq cf(n,\eps),$   each cluster $\mathcal{L}_i^{(n),\pm}$ intersects $\H_{1,\eps}^{\pm}$ and has diameter at least $\delta/2$ since $\mathcal{L}_i^{\pm}$ intersects $\H_{2,\eps}^{\pm}$ and has diameter at least $\delta.$ Note that if the discrete excursions intersect each other then the cable system excursions also intersect each other. Therefore, using a union bound and the strong Markov property, Lemma~\ref{Lem:simpleRWresult} implies that for any $p,m,n\in\N,$ on the event that the number of Brownian excursions is smaller than $p$ and the number of loop soup excursions is smaller than $m,$ if $\delta\wedge\eps\geq cf(n,\eps),$ one has 
\begin{equation}
\label{eq:consbeurling}
\begin{split}
\Pm\Big((A_{\eps,\delta}^{(n),\pm})^\ch,\sum_{i=1}^{N_{\eps}}K_i^{\pm}\leq p,L_{\delta,\eps}^{\pm}\leq m,\,\Big|\,x_{i,j}^{(n),\pm},&i\in{\{1,\dots,N_b^{\pm}\}},j\in{\{1,\dots,K_i^{\pm}\}}\Big)
    \\&\leq cp(p+m)\left(\frac{f(n,\eps)}{\eps\wedge\delta}\right)^{1/2}
\end{split}
\end{equation}
where $x_{i,j}^{(n),\pm}$ denotes the hitting point of $\H_{1,\eps}^{\pm}$ for the subexcursion $E_{i,j}^{(n),\pm}.$  Note also that if  $\eps\wedge(r'-r)\geq cf(n,\eps)$, on the event $\mathcal{E}_1^{(n)}$ one can upper bound $K_i^{\pm}$ by the number of subexcursions in $B(1-4\eps-\eps/3)$ for $w_i$ which hit $B(1-4\eps-2\eps/3)$, plus the number of subexcursions in $B(1-\eps-\eps/3)$ for $w_i$ which hit $B(1-\eps-2\eps/3)$, plus the number of subexcursions $\D\setminus B(r+(r'-r)/6)$ for $w_i$ which hit $\D\setminus B(r+(r'-r)/3)$, plus the number of subexcursions in $\{x\in{\D\setminus B((r+r')/4)}:\pm\text{arg}(x)\in{[-20\pi/84,\pi+20\pi/84]} \}$ for $w_i$ which hit $\{x\in{\D\setminus B((r+r')/4)}:\pm\text{arg}(x)\in{[-19\pi/84,\pi+19\pi/84]}$ \}, plus the number of subexcursions in $\{x\in{\D\setminus B((r+r')/4)}:\pm\text{arg}(x)\in{[-\pi+\pi/84,-\pi/84]} \}$ for $w_i$ which hit $\{x\in{\D\setminus B((r+r')/4)}:\pm\text{arg}(x)\in{[-\pi+2\pi/84,-2\pi/84]} \}$. In view of \eqref{eq:hittingBM}, we deduce that $K_i^{\pm}$ can be upper bounded by a sum of five geometric random variables with constant parameters, depending only on $r,r'$. Combining this with \eqref{capball} and recalling that $N_{\eps}$ is a Poisson random variable with parameter $u\mathrm{cap}(B(1-\eps)),$ one can thus find a constant $C<\infty,$  depending only on $u,r,r',$ such that if  $\eps\wedge(r'-r)\geq cf(n,\eps)$ then the total number of Brownian excursions we consider satisfies
\begin{equation*}
    \E\Big[\sum_{i=1}^{N_{\eps}}K_i^{\pm}\I\big\{\mathcal{E}_1^{(n)}\big\}\Big]\leq C\eps^{-1}.
\end{equation*}
Markov's inequality then yields for all $t>0$
\begin{equation}
\label{eq:poissonbound}
    \Pm\left(\sum_{i=1}^{N_{\eps}}K_i^{\pm}\geq (t\eta\eps)^{-1},\,\mathcal{E}_1^{(n)}\right)\leq t\eta\eps \E\Big[\sum_{i=1}^{N_{\eps}}K_i^{\pm}\I\big\{\mathcal{E}_1^{(n)}\big\}\Big]\leq Ct\eta.
\end{equation}
If $\lambda=0$ take $\delta=1$ and otherwise, for each $\eta>0,$ take $\delta=\delta(\eta,\eps,r,r')>0$ small enough so that
\begin{equation}
\label{eq:choicedelta}
    \Pm\big(F_{2,\eps}^\pm({\be}^{u,\lambda,\delta})^\ch,F_{2,\eps}^\pm({\be}^{u,\lambda})\big)\leq \eta/7.
\end{equation}
The existence of such a $\delta$ follows from the fact that if $F_{2,\eps}^\pm(\be^{u,\lambda})$ occurs, then there is a continuous path $\pi$ in $\mathcal{I}^{\intens,\lambda}\cap \H_{2,\eps}^{\pm}$ which hits both $S_{2,\eps}^{\pm,l}$ and $S_{2,\eps}^{\pm,r}$ and which consists of finitely many loops and excursions. Since the details are slightly cumbersome, we defer the proof of this fact to the end of this proof.

Further choose $t$ small enough so that \eqref{eq:poissonbound} is bounded by $\eta/7$,
$m=m(\eta,\eps,r,r')$ large enough (if $\lambda\neq0$) so that, with the previous choice of $\delta,$ the number of loop excursions satisfies
\begin{equation}
\label{eq:choicem}
    \Pm(L_{\delta,\eps}^{\pm}> m)\leq\eta/7,
\end{equation}
and $n=n(\eta,\eps,r,r')$ large enough so that recalling \eqref{eq:boundcoupling2}, $\eps\geq cf(n,\eps),$ $r'-r\geq cf(n,\eps)$ and  $\delta\geq cf(n,\eps)$ (for $\delta$ as in \eqref{eq:choicedelta}), as required for \eqref{eq:ifAnepsdeltathendiscretedisconnection}, \eqref{eq:consbeurling} and \eqref{eq:poissonbound} to hold; and so that \eqref{eq:consbeurling} (for $p=(t\eta\eps)^{-1}$, $\delta$ as in \eqref{eq:choicedelta} and $m$ as in \eqref{eq:choicem}), $\Pm\big(\mathcal{E}_1^{(n)}\cap\mathcal{E}_2^{(n)}\cap A_{\eps,\delta}^{(n),\pm}\cap\big(\mathcal{A}^{(n),\pm}_{\eps,\delta}\big)^{\ch}\big)$, $\Pm\big((\mathcal{E}_1^{(n)})^\ch\big)$ and $\Pm\big((\mathcal{E}_2^{(n)})^\ch\big)$ are all bounded by $\eta/7$. It then follows from \eqref{eq:ifAnepsdeltathendiscretedisconnection} that the probability in \eqref{eq:approxconnection1} is smaller than $\eta,$ and \eqref{eq:approxconnection1} follows readily.

The proof of \eqref{eq:approxconnection2} is similar to the proof of \eqref{eq:approxconnection1} when $\lambda=0,$ that is when there is no loops, but exchanging the roles of the cable system excursions and Brownian excursions on $\mathbb{D}$: first now define $E^{\pm}_{i,j}$ as the subexcursions of $w_i'$ in $\H_{0,\eps}^{\pm}$ hitting $\H_{1,\eps}^{\pm}$, and $E_{i,j}^{(n),\pm}$ as the part of $\tilde{F}(w_i')$ close to $E_{i,j}^{\pm}.$ Then defining $A_{\eps}^{\pm}$ similarly as in \eqref{eq:defApmnij} and \eqref{eq:defApmn} but for the excursions $E^{\pm}_{i,j},$ and forgetting about the loops, one has similarly as in \eqref{eq:ifAnepsdeltathendiscretedisconnection} that $\mathcal{E}_1^{(n)}\cap\mathcal{E}_2^{(n)}\cap A_{\eps}^{\pm}\cap \overline{F}_{2,\eps}^\pm(\widetilde{\be}^{u}_n)\subset \overline{F}_{0,\eps}^\pm({\be}^{u}).$ We refer to the right-hand side of Figure~\ref{F:discretecritpar} for an illustration. Moreover, the bound \eqref{eq:consbeurling} still holds for $A_{\eps}^{\pm}$ by Lemma~\ref{Lem:simpleRWresult}, considering the hitting points $x_{i,j}^{\pm}$ of $\H_{1,\eps}^{\pm}$ for $E_{i,j}^{\pm}$ instead of $x_{i,j}^{(n),\pm}.$ Note additionally that  when $\eps=\eps_n$ and $n$ is large enough, then $\eps\geq cf(n,\eps),$ see \eqref{eq:boundcoupling2}, and both the right-hand side of \eqref{eq:consbeurling}, for $m=0$, $\delta=1$ and $p=c\eps^{-1},$ as allowed by \eqref{eq:poissonbound}, and $\Pm\big((\mathcal{E}_1^{(n)})^{\ch}\big)$, converge to zero.  This finishes the proof of \eqref{eq:approxconnection2}.

It remains to prove \eqref{eq:choicedelta}. If $F_{2,\eps}^\pm({\be}^{u,\lambda})$ occurs, then by \eqref{eq:defFi+-} there is a continuous path $\pi$ in $\mathcal{I}^{\intens,\lambda}\cap \H_{2,\eps}^{\pm}$ which hits both $S_{2,\eps}^{\pm,l}$ and $S_{2,\eps}^{\pm,r}$. Moreover since  $\H_{2,\eps}^{\pm}$ is open,  $\pi$ is at positive distance $s$ from $\partial \H_{2,\eps}^{\pm}$. By local finiteness of loop clusters, see Lemmas~9.4 and~9.7 as well as Propositions~10.3 and~11.1 in \cite{sheffield-werner}, the clusters of loops included in $\H_{2,\eps}^{\pm}$ which reach distance at least $s/2$ from  $\partial \H_{2,\eps}^{\pm}$ are all at positive distance $s'$ from $\partial \H_{2,\eps}^{\pm}$. We decompose the excursions in $\omega_u$ into subexcursions in $\H_{2,\eps}^{\pm}$, as well as the loops at level $\lambda$ which intersect $\partial \H_{2,\eps}^{\pm}$ and whose loop cluster in $\D$ intersect an excursion in $\omega_u$, into subexcursions in $\H_{2,\eps}^{\pm}$. We denote by $w_1,\dots,w_P$, the previous subexcursions of excursions or loops which reach distance $s'$ from $\partial \H_{2,\eps}^{\pm}$.  Note that there are only finitely many such subexcursions by Proposition~\ref{prop:localdescription} and local finiteness of loops, as well as properties of Brownian motion. We further decompose the loops entirely included in $\H_{2,\eps}^{\pm}$ into loop clusters, and we write $\mathbf{C}_i$, $1\leq i\leq P$, for the union of $w_i$ and all those loop clusters which intersect $w_i$, as well as $\overline{\mathbf{C}}_i$ for its closure. Then by construction any point in $\H_{2,\eps}^{\pm}$ at distance at least $s/2$ from $\partial \H_{2,\eps}^{\pm}$, which belongs to a loop whose loop cluster in $\D$ intersect an excursion of $\omega_u$, belongs to one of the sets $\mathbf{C}_i$, $1\leq i\leq P$. In particular, $\pi$ is included in the union of the sets $\overline{\mathbf{C}}_i$ for $1\leq i\leq P$.

Moreover, for $1\leq i\neq j\leq P$, if $\overline{\mathbf{C}}_i$ intersects $\overline{\mathbf{C}}_j$ in some point $x\in{\H_{2,\eps}^{\pm}}$, then a.s.\ $\mathbf{C}_i$ intersects $\mathbf{C}_j$. Indeed, either $w_i$ or $w_j$ intersects $x$, and then the previous statement follows from properties of Brownian motion on $\D$; or both $w_i$ and $w_j$ are at positive distance from $x$, and then $x$ is in the closure of a loop cluster (for the loops included in $\H_{2,\eps}^{\pm}$) intersecting $w_i$, resp.\ $w_j$, by local finiteness of loop clusters, and these two clusters a.s.\ actually coincide since the outer boundaries of distinct loop clusters  are a.s.\ always at positive distance from one another by \cite[Proposition~10.3 and~11.1]{sheffield-werner}, see also p.\ 1899 therein. Define recursively on $k\geq1$ the number $i_k$ as  the smallest index $i\in{\{1,\dots,P\}\setminus \{i_1,\dots,i_{k-1}\}}$ such that $\pi$ is in $\overline{\mathbf{C}}_i$ when first exiting the union of $\overline{\mathbf{C}}_{i_{k'}}$, $1\leq k'\leq k-1$.  Assuming that $P'$ clusters are explored during the previous recursion, then for each $1\leq k\leq P'$, we have by definition that $\overline{\mathbf{C}}_{i_k}\cap\overline{\mathbf{C}}_{{j_k}}\neq\varnothing$ for some $j_k<i_k$, and thus a.s.\  ${\mathbf{C}}_{i_k}\cap {\mathbf{C}}_{j_{k}}\neq\varnothing$. By definition of loop clusters, see below \eqref{eq:defcontloops}, one can moreover connect any point of $\mathbf{C}_{j_k}$ to $\mathbf{C}_{i_k}$ using a finite number of loops and excursions in $\H_{2,\eps}^{\pm}$. Iterating, we obtain that one can connect any two points in the union of $\mathbf{C}_{i_j}$, $1\leq j\leq P'$, using only finitely many loops and subexcursions therein. Since the union of $\overline{\mathbf{C}}_{i_j}$, $1\leq j\leq P'$, intersect both $S_{2,\eps}^{\pm,l}$ and $S_{2,\eps}^{\pm,r}$, this is a.s.\ also the case for the union of ${\mathbf{C}}_{i_j}$, $1\leq j\leq P'$. Therefore, one can a.s.\ find a path in $\H_{2,\eps}^{\pm}$ connecting $S_{2,\eps}^{\pm,l}$ to $S_{2,\eps}^{\pm,r}$ using a finite number of loops and subexcursions in the union of ${\mathbf{C}}_{k}$, $1\leq k\leq P$. In particular, there exists $\delta>0$ such that all the loops along this path have diameter at least $\delta$, and such that any loop among the subexcursions $w_1,\dots,w_P$ is connected to $\be^u$ using only loops of diameter at least $\delta$. In particular, the event $F_{2,\eps}^\pm({\be}^{u,\lambda,\delta})$ occurs, which finishes the proof of \eqref{eq:choicedelta}.

\end{proof}

\begin{remark}
    In the proof of Lemma~\ref{lem:approxconnection}, the main strategy when $\lambda=0$ is the following: if $F_{2,\eps}^{\pm}(\be^u)\cap\mathcal{E}_1^{(n)}$ occurs then there is a sequence of cable system excursions in $\H_{0,\eps}^{\pm}$ almost connecting $S_{0,\eps}^{\pm,r}$ to $S_{0,\eps}^{\pm,l}$,  that is with only finitely many small gaps between the excursions, and we then show that these excursions actually fill these gaps with high probability by Lemma~\ref{Lem:simpleRWresult}, and vice versa. Another possible approach would be to proceed similarly as in \cite[Lemma~2.7]{MR3485399}, see also \cite[Theorem~5.1]{MR3547746} for a similar approach : if $F_{2,\eps}^{\pm}(\be^u)$ occurs then there is a finite set of continuous excursions in $\H_{2,\eps}^{\pm}$ almost connecting $S_{2,\eps}^{\pm,r}$ to $S_{2,\eps}^{\pm,l}$, and then with high probability these excursions are strongly entangled, that is any set close enough (with respect to the Hausdorff distance) to these excursions still connects $S_{0,\eps}^{\pm,r}$ to $S_{0,\eps}^{\pm,l}$ in $\H_{0,\eps}^{\pm}$, and thus $F_{0,\eps}^{\pm}(\tilde{\be}_n^u)$ occurs on the coupling event $\mathcal{E}_1^{(n)}$. However this argument seems more complicated to implement for the other direction, that is when starting with $F_{2,\eps}^{\pm}(\tilde{\be}_n^u)$, and in particular it is hard to have good control on the connectivity of these excursions near the boundary, hence our different approach to prove Lemma~\ref{lem:approxconnection}. 
\end{remark}

We can now establish Theorem~\ref{the:maindisexcloops}.
\begin{proof}[Proof of Theorem~\ref{the:maindisexcloops}]
Let us first consider the case $u\geq (8-\kappa)\pi/16$, which uses an argument similar to the proof of Lemma~\ref{lem:2}. By Lemma~\ref{l.brownianexcursionandlooprestriction} for $\Gamma=\big\{x\in{\partial\D}:\,\pm\arg(x)\in{\big[-4\pi/28,4\pi/28+\pi\big]}\big\}$ and $D=\H^{\pm}_{2,0}$, combined with Lemma~\ref{lem:SLE-ka-r}, the event $F_{2,0}^\pm({\be}^{u,\lambda}\big)$ occurs with probability one. Since the liminf of the events $F_{2,\eps}^\pm({\be}^{u,\lambda}\big)$ as $\eps\searrow0$ is contained in $F_{2,0}^\pm({\be}^{u,\lambda}\big)$, combining the previous observation with \eqref{eq:approxconnection1} we have for all $\eps\in{(0,1/6)}$, $r\in [0,1-6\eps)$ and $r'\in{(r,1-6\eps)}$
\begin{equation}
\label{eq:subcriticalproof}
    \limsup_{\eps\rightarrow0}\limsup_{n\rightarrow\infty}\Pm\big(F_{0,\eps}^{\pm}(\tilde{\be}_n^{u,\lambda})^{\ch}\big)\leq \limsup_{\eps\rightarrow0}\Pm\big(F_{2,\eps}^\pm({\be}^{u,\lambda}\big)^{\ch}\big)=0,
\end{equation}
Using \eqref{eq:interestofF}, one easily deduces \eqref{eq:percowithloops} for $u\geq (8-\kappa)\pi/16$.

Let us now turn to the case $u< (8-\kappa)\pi/16$, which follows from the following simple observation: for all  $\eps\in{(0,1)},$
\begin{equation}
\label{supercriteqc2}
    \liminf_{n\rightarrow\infty}\Pm\big(0\stackrel{\tilde{\mathcal{V}}_n^{\intens,\lambda(\kappa)}}{\longleftrightarrow}{\partial}B_n(1-\eps)\big)\geq \Pm\big(0\stackrel{\mathcal{V}^{\intens,\lambda(\kappa)}}{\longleftrightarrow}\partial\D\big).
\end{equation}
Indeed, if $0\leftrightarrow \partial\D$ in $\V^{\intens,\lambda},$ then there exists a path $\pi$ between $0$ and $\partial B(1-\eps/2)$ and $\delta\in{(0,\eps)}$ so that $B(\pi,\delta)\subset \V^{\intens,\lambda}.$ Then, under the coupling from  Theorem~\ref{the:couplingwithloops}, $\pi\subset \widetilde{\V}_n^{\intens,\lambda}$ for $n$ large enough, which implies \eqref{supercriteqc2}. Combined with Theorem~\ref{t.mainthm}, this proves \eqref{eq:percowithloops} for $u< (8-\kappa)\pi/16$.
\end{proof}

\begin{proof}[Proof of Theorem~\ref{the:maindisexc}]
The equality \eqref{eq:percowithoutloopssub} follows from \eqref{eq:percowithloops} for $\kappa=8/3,$ and we now prove \eqref{eq:percowithoutloops} using an argument similar to the proof of Lemma~\ref{lem:1} for $\lambda=0$, from which we will use the notation. The only difference is that instead of using Lemmas~\ref{l.brownianexcursionandlooprestriction} and~\ref{lem:SLE-ka-r}  for $D=\D$ to build a path $\gamma$ in $\D\setminus (\be^{\intens,\lambda(\kappa)}_{\mathbb{T}_+,\mathbb{T}_+,\D})$ from $0$ to $\mathring{\mathbb{T}}_+$ on the event $E_1$ from \eqref{eq:defE1}, we use them for $D=\{x\in{\D}:\,\Im(x)>0\}$, which imply that we can a.s.\ build a path $\gamma$ now in $D\setminus (\be^{\intens,\lambda(\kappa)}_{\mathbb{T}_+,\mathbb{T}_+,D})$, still from $0$ to $\mathring{\mathbb{T}}_+$. We then replace the event $E_1$ by the event 
\begin{equation*}
    E'_1=\big\{(\be^{\intens}_{\mathbb{T}_+,\mathbb{T}_+,\D}\setminus \be^{\intens}_{\mathbb{T}_+,\mathbb{T}_+,D})\cap\gamma=\varnothing\big\}.
\end{equation*}
Note that $E'_1$ is contained in the intersection of the independent events $E'_2=\big\{{\be}^{\intens}_{\Gamma,\Gamma,\D}\cap \gamma=\varnothing\big\}$ and $E'_3$ that there is no excursions starting or ending in $\partial \D\setminus\Gamma$ which hit $\D\setminus D$, where we recall that $\Gamma$ is a closed arc in $\partial\D$ not intersecting $\gamma$ and containing $\mathbb{T}_-$ in its interior. One can show similarly as for $E_2$ and $E_3$ defined in \eqref{eq:defE2} and below that the events $E'_2$ and $E'_3$ have positive probability, and hence $E'_1$ as well. Moreover, on the event $E'_1$, $\gamma$ is a path in $D\setminus (\be^{\intens}_{\mathbb{T}_+,\mathbb{T}_+,\D})$ from $0$ to $\mathring{\mathbb{T}}_+$. Since the event $E'_1$ is measurable with respect to $\be^{\intens}_{\mathbb{T}_+,\mathbb{T}_+,\D}$, one can follow the proof Lemma~\ref{lem:1} replacing the event $E_1$ by $E'_1$ to deduce that $\gamma$ is a path in $D\setminus (\be^{\intens})$ from $0$ to $\mathring{\mathbb{T}}_+$ with positive probability, which implies that $\overline{F}_{0,0}^+({\be}^{u}\big)^{\ch}$ occurs with positive probability. Since the events $\overline{F}_{0,\eps}^+({\be}^{u}\big)$ are decreasing to $\overline{F}_{0,0}^+({\be}^{u}\big)$ as $\eps\searrow0$, combining the previous observation with \eqref{eq:approxconnection1} we have for all $r\in [0,1)$ and $r'\in{(r,1)}$
\begin{equation}
\label{eq:supercriticalproof}
    \liminf_{n\rightarrow\infty}\Pm\big(\overline{F}_{2,\eps_n}^{+}(\tilde{\be}_n^{u})^{\ch}\big)\geq \liminf_{n\rightarrow\infty}\Pm\big(\overline{F}_{0,\eps_n}^+({\be}^{u}\big)^{\ch}\big)>0,
\end{equation}
Using \eqref{eq:interestofFbar} and noting that $7\eps_n=n^{-1/7}$, we deduce  \eqref{eq:percowithoutloops} for $2r'(>0)$ instead of $r$. It remains to prove \eqref{eq:percowithoutloops} for $r=0$. We have
\begin{equation*}
    \Pm\big(0\stackrel{\tilde{\mathcal{V}}_n^{\intens}}{\longleftrightarrow}{\partial}B_n(1-n^{-1/7})\big)\geq \Pm\big(B_n(r)\stackrel{\tilde{\mathcal{V}}_n^{\intens}}{\longleftrightarrow}{\partial} B_n(1-n^{-1/7})\big)-\Pm\big(B_n(r)\cap \tilde{\mathcal{I}}_n^{\intens}\neq\varnothing\big).
\end{equation*}
It follows from \eqref{capball} and Proposition~\ref{prop:localdescription} that a.s.\ $B(r')\cap\be^u=\varnothing$ for $r'$ small enough.  Therefore by Corollary~\ref{cor:couplingexcursions} we obtain
\begin{equation*}
    \lim_{r\rightarrow0}\limsup_{n\rightarrow\infty}\Pm\big(B_n(r)\cap\tilde{\mathcal{I}}_n^{\intens}\neq\varnothing\big)=\lim_{r\rightarrow0}\Pm\big(B(r)\cap {\mathcal{I}}^{\intens}\neq\varnothing\big)=0,
\end{equation*}
and \eqref{eq:percowithoutloops} for $r=0$ then follows from \eqref{eq:percowithoutloops} and \eqref{eq:supercriticalproof} for $r=0$.  Since it is harder to connect $B_n(r)$ to ${\partial}B_n(1-n^{-1/7})$ than to ${\partial}B_n(1-\eps)$ for $n$ large enough, we conclude that $u_*^d(r)=\pi/3.$
\end{proof}

\begin{remark}
\label{rk:polycloseboundary}
In \eqref{eq:percowithloops}, we only obtain connection at distance $\eps>0$ from the boundary in the supercritical phase when $\kappa\in{(8/3,4]},$ and not at polynomial distance from the boundary as when $\kappa=8/3,$ see \eqref{eq:percowithoutloops}. In order to obtain connection at polynomial distance from the boundary when $\kappa\in{(8/3,4]},$ one would first need to extend Theorem~\ref{the:couplingwithloops} to obtain a coupling at polynomial distance from the boundary of the disk between discrete and continuous loops plus excursions, similarly as in Corollary~\ref{cor:couplingexcursions}, but that result would not be so interesting here. Indeed, we are mainly interested in two particular values of $\kappa$: first $\kappa=8/3,$ that is removing the loops, which is already covered in Theorem~\ref{the:couplingdiscontexc1}, and then $\kappa=4,$ due to its link with the GFF \eqref{eq:iso}. In our main dGFF result, Theorem~\ref{the:maindisgff}, we only use the isomorphism \eqref{eq:iso} to obtain the inequality $h_*^d(r)\leq \sqrt{\pi/2}.$ But this already implies $h_*^d(\boldsymbol{\eps},r)\leq \sqrt{\pi/2}$ for any sequences $\boldsymbol{\eps}$ decreasing to $0,$ see Remark~\ref{rk:connectiontoboundary}, \ref{rk:othercritexp}, and thus is already as strong as we want. In other words, improving Theorem~\ref{the:couplingwithloops} to obtain a coupling at polynomial distance from the boundary mainly leads to better percolation results in the supercritical phase of $\tilde{\V}^{u,\lambda}_n$ for $\lambda>0$, but we are mainly interested in the subcritical phase of $\tilde{\V}^{u,\lambda}_n$ for $\lambda=1/2$ due its link with the GFF.
\end{remark}

\section{Discrete Gaussian free field}\label{s:gfflvlperc}

In this section we prove Theorem~\ref{the:maindisgff}. Since $\lambda(4)=1/2$, the bound $h_*^d(r)\leq\sqrt{\pi/2}$ is a simple consequence of Theorem~\ref{the:maindisexcloops} for $\kappa=4$ and the isomorphism theorem, Proposition~\ref{prop:BEisom}, and we thus focus on the proof of \eqref{eq:GFFconnectionBds}, that is we prove that one can find $h> 0$ so that with probability bounded away from zero for  $n$ large,  $B_n(r),$ $r\in{(0,1)},$ can be connected to the boundary of  $\D_n$ in $E^{\ge h}_n.$ Using a simple consequence of the  isomorphism~(cf.\ Proposition~\ref{prop:BEisom}), we also deduce Corollary~\ref{cor:VunPercBds}. We conclude this section by some remarks about percolation for the Gaussian free field on the cable system, see Remark~\ref{rk:cablegff}.

The proof of \eqref{eq:GFFconnectionBds} is based on methods tracing back to \cite{BrLeMa-87} at least. These have been significantly extended in \cite{DiLi-18} and \cite{DiWiWu-20} by  the use of martingale theory. In the latter, a weak  version of \eqref{eq:GFFconnectionBds}, i.e., an analogue of
\[
\liminf_{n\to \infty} \Pm \big(B_n(r) \stackrel{E_n^{\ge 0}}{\longleftrightarrow} \partial \widehat{\partial}\D_n \big)>0
\]
is proven, where however $B_n(r)$ and $\partial \D_n$ are replaced by the left and right boundary of a discrete rectangle, respectively. The approach of \cite{DiWiWu-20} will also serve as a guideline for our proof.  However, quite a number of technical adaptations are in order due to the difference of connecting a macroscopic subset of $\D_n$ to the boundary $\partial \D_n$ in $E_n^{\ge h}$ for some $h>0$ instead of connecting the left boundary of a rectangle to the right boundary in $E_n^{\ge 0}.$ Since for us it is more adequate and  convenient to work in the setting of the unit disk as explained at the end of Section~\ref{sec:background} (see e.g.\ the proof Lemma~\ref{Lem:bigcaponboundary}), we provide a self-contained proof of \eqref{eq:GFFconnectionBds} in the rest of this section.

 For $r\in{(0,1)}$ as well as $n \in \N$ fixed, we will define for each $h\in{\R}$ the exploration process $(\widetilde{\mathcal E}_t^{\ge h}),$ $t \in [0,\infty),$ which will take values in the set of measurable subsets of $\widetilde \D_n.$ More precisely, let  
 \begin{equation} \label{eq:E=K}
  \widetilde{\mathcal{E}}_0^{\ge h} := \widetilde B_n(r),
 \end{equation} 
 and for $t >0$ define $\widetilde{\mathcal E}_t^{\ge h} \subset \widetilde \D_n$ to be the union of $\widetilde{\mathcal{E}}_0^{\ge h}$ with the set consisting of all points $y \in \widetilde \D_n$ for which there exists $t_\gamma \in [0,t]$ and a continuous path  $\gamma:[0,t_{\gamma}]\rightarrow\widetilde \D_n$ parametrized by arc length (attributing entire cables $I_e$ length $1/2$ when $e$ is an edge of $\D_n,$   with linear interpolation in between, and infinite length when $e$ is an edge between $\D_n$ and $\partial\D_n$), with the following properties:
 \begin{enumerate}
 \item  $\gamma(0) \in \widetilde{\mathcal{E}}_0^{\ge h};$
 
 \item    $\gamma(t_{\gamma})  = y;$
 
 \item  for all $x \in \{\gamma(s):\,s\in{(0,t_{\gamma})}\}\cap \D_n $  we have $\varphi_x \ge h.$ 
\end{enumerate}
In other words, $(\widetilde{\mathcal{E}}_t^{\ge h})_{t\geq0}$ continuously explores $\widetilde{\D}_n$ starting from $\widetilde{\mathcal{E}}_0^{\ge h},$ stopping the exploration along a path whenever a vertex $x\in{\D_n}$ with $\varphi_x<h$ is reached. Note that in order to simplify notation, we made the dependence of $\widetilde{\mathcal{E}}_t^{\ge h}$ on $n$ implicit, and we will do so for all the notation introduced in this section. 

For any $K\subset\widetilde{\D}_n$ set
\begin{equation} \label{eq:Mdef}
 M_{K} := \sum_{x\in{\widehat \partial K}}e_{K}^{(n)}(x)\varphi_x,
\end{equation} 
  where $\widehat{\partial}K$ is defined as in \eqref{eq:interior} for subsets $K$ of the cable system, and the cable system equilibrium measure has been introduced in Section~\ref{sec:cableSystemExc}.
Writing $\mathcal M_t^{\ge h} := M_{{\mathcal E}_t^{\ge h}}$
we have that  $(\mathcal M_t^{\ge h})_{t\geq0}$ 
is a martingale, the so-called  exploration martingale; see  \cite[Section~IV.6]{kups11574} for further details.

In addition, since $\widetilde{\mathcal E}_t^{\ge h}$ is increasing in $t \in [0,\infty)$, the limit
\begin{equation} \label{eq:calEInfty}
\widetilde{\mathcal E}_\infty^{\ge h} := \bigcup_{t \in [0,\infty)} \widetilde{\mathcal E}_t^{ \ge h}
\end{equation} 
is well-defined. When $\widetilde{\mathcal{E}}_{\infty}^{\ge h}$ is compact, that is when $\widetilde{\mathcal{E}}_{\infty}^{\ge h}$ does not contain an entire cable between $\D_n$ and $\partial\D_n$ (since these cables are half-open by definition of $\widetilde{\D}_n$ in Section~\ref{sec:notation}), we note that $\widetilde{\mathcal{E}}_t^{\ge h}=\widetilde{\mathcal{E}}_{\infty}^{\ge h}$ for $t$ large enough, and that $\widetilde{\mathcal{E}}_{\infty}^{\ge h}$ is a union of entire cables. Moreover, $\widetilde{\mathcal{E}}_{\infty}^{\ge h}$ is non-compact if and only if an edge between $\widehat{\partial}\D_n$ and $\partial\D_n$ is explored, which happens if and only if $B_n(r)$ is  connected to $\widehat{\partial}\D_n$ in $E_n^{\geqslant h}$ by definition. Therefore,
\[
\mathcal M_\infty^{\ge h} := \lim_{t \to \infty} \mathcal M_t^{\ge h}
\]
is well-defined when $B_n(r)$ is not connected to $\widehat \partial \D_n$ in $E_n^{\ge h},$ and is then equal to $M_{\widetilde{\mathcal{E}}_{\infty}^{\ge h}}$. We also write
\begin{equation*}
{\mathcal{E}}_t^{\ge h} := 
\widetilde{\mathcal{E}}_t^{\ge h} \cap \D_n\text{ for all }t\in{[0,\infty)\cup\{\infty\}}.
\end{equation*}
Since $\tilde{\mathcal{E}}_{\infty}^{\ge h}$ is a union of entire cables, including the endpoints, the equilibrium measure of  $\tilde{\mathcal{E}}_{\infty}^{\ge h}$ and ${\mathcal{E}}_{\infty}^{\ge h}$ are equal when $\tilde{\mathcal{E}}_{\infty}^{\ge h}$ is compact, and so
\begin{equation}
\label{eq:equalityCaps}
    \M_{\infty}^{\ge h}=M_{\mathcal{E}_{\infty}^{\ge h}}\text{ on }\big\{B_n(r)\stackrel{E^{\ge h}_n}{\longleftrightarrow}\widehat \partial \D_n\big\}^\ch.
\end{equation}
We now introduce some further notation that will prove useful in the rest of this section. For $r\in{(0,1)}$ we consider $K$ such that
\begin{equation} \label{eq:Lass}
K\text{ is a connected subset of } \D_n \text{ with }   B_n(r)\subset K
\end{equation} 
and $F$ such that
\begin{equation}  \label{eq:LK}
F\subset K \subset \D_n \text{ such that } F \supset \widehat \partial K\text{ and }\forall\,x\in{\widehat{\partial}K},\ \exists\ y\sim x\text{ with }y\in{K\setminus F}.
\end{equation} 
Note that the sets $F$ and $K$ depend implicitly on $n.$ In addition, let
\begin{equation} \label{eq:configDef}
\mathcal B^{h}_{K,F} := \big \{\varphi_x < h  \text{ for all } x \in F, \, \varphi_x \ge h\,  \text{ for all } x \in K \setminus F\big \}.
\end{equation}
Intuitively, configurations of the form $\mathcal B^{h}_{K,F}$ will play the role of (discretized) final configurations of the exploration process in case the clusters of $E^{\ge h}_n$ intersecting $B_n(r)$ do not connect $B_n(r)$ to $\widehat\partial \D_n$ in $E^{\ge h}_n,$ see \eqref{proof0connect2} below.

We now formulate some results which are modifications of findings from \cite{DiWiWu-20}, and which will prove useful in the following. For this purpose, recalling the definition of the interior boundary of a set from \eqref{eq:interior} as well as the equilibrium measure and capacity from \eqref{defequiandcap}, we let
\begin{equation} \label{eq:supHarmMeas}
{\rm Es}^{(n)}(K) 
:= \frac{\sup_{y \in \widehat{\partial}(K\setminus \widehat{\partial}K)} e_{ K\setminus \widehat{\partial}{K}}^{(n)}(y)}{\mathrm{cap}^{(n)}(K)}.
\end{equation}

\begin{proposition} \label{prop:condNegative}

There exists a function $\varepsilon :  (0,\infty)  \to (0,\infty)$ with $\varepsilon(t) \to 0$ as $t \searrow 0$ and a constant $\Cl[const]{repul}\in (0,1)$ such that for all $n \in \N,$ $r\in{(0,1)},$ $h\leq  \Cr{repul}$ and all $K$ and $F$ as in \eqref{eq:Lass} and \eqref{eq:LK}, we have
\begin{align} \label{eq:condNegative}
\begin{split}
 \Pm\big(M_{K}\le - \Cr{repul} \mathrm{cap}^{(n)}(K)\, | \, \mathcal B^{h}_{K,F} \big)
\ge 1-\varepsilon \big({\rm Es}^{(n)}(K)\big).
\end{split}
\end{align}
\end{proposition}
\begin{proof}
We use \cite[Proposition~4.1]{DiWiWu-20} applied with $\lambda=1,$ $\tilde{K}=\widetilde{\D}_n,$ $U=\widehat{\partial}\D_n,$ $I=K,$ $I^-=F,$ $I^+=K\setminus F,$ boundary condition $f=-h.$ Note that the set of points in $K$ which can be reached by the random walk starting from $\widehat{\partial}\D_n$ is $\widehat{\partial}K,$ see \eqref{eq:interior}, and that the set of points in $K\setminus \widehat{\partial}K$ in which the random walk starting from $\widehat{\partial}\D_n$ can hit $K\setminus \widehat{\partial}K$ for the first time is given by $\widehat{\partial}(K\setminus \widehat{\partial}K),$ cf.\ \eqref{eq:configDef}. Then, writing
\begin{equation*}
    Y:=\!\!\sum_{x\in{\widehat{\partial}}K}\sum_{y\in{\widehat{\partial}\D_n}}\!\!\Pm_y^{(n)}(X_{H_K}=x,H_K<\infty)(\varphi_x-h)\text{ and }{\rm Hm}^{(n)}(K):=\!\!\sum_{y\in{\widehat{\partial}\D_n}}\!\Pm_y^{(n)}(H_K<\infty),
\end{equation*}
there exists $\Cr{repul}\in{(0,1)}$ and a function $r$ converging to $0$ at $0$ such that uniformly in $h\in{[0,1)},$ $n \in \N$ and $K$ and $F$ as in \eqref{eq:Lass} and \eqref{eq:LK},
\begin{equation*}
     \Pm\big(Y\le - 8\Cr{repul} {\rm Hm}^{(n)}(K)\, | \, \mathcal B^{h}_{K,F} \big)
\ge 1-r \big(\xi^{(n)}(K)\big),
\end{equation*}
where
\begin{equation*}
    \xi^{(n)}(K):=\frac{1}{{\rm Hm}^{(n)}(K)}\sup_{x\in{\widehat{\partial}(K\setminus \widehat{\partial}K)}}\sum_{y\in{\widehat{\partial}\D_n}}\Pm_y^{(n)}\big(X_{H_{K\setminus \widehat{\partial}K}}=x,H_{K\setminus \widehat{\partial}K}<\infty\big).
\end{equation*}
We  now observe that for each $x\in{\widehat{\partial}\D_n},$ we have that  $e_{\D_n}^{(n)}(x)=|\{y\in{\partial \D_n }:\,y\sim x\}|\in{[1,4]}$ by \eqref{defequiandcap}, since $\{\tau_{\partial \D_n}<\widetilde{\tau}_{\widehat{\partial}\D_n}\}$ can only occur if $X$ jumps directly from $x$ to $\partial\D_n.$ Therefore, on the event $\mathcal{B}_{K,F}^h$ we have that
\begin{equation*}
    Y\geq \sum_{x\in{\widehat{\partial}}K}\sum_{y\in{\widehat{\partial}\D_n}}e_{\D_n}^{(n)}(y)\Pm_y^{(n)}(X_{H_K}=x,H_K<\infty)(\varphi_x-h)=M_K-h\mathrm{cap}^{(n)}(K),
\end{equation*}
where we used the discrete sweeping identity, see for instance \cite[(1.59)]{MR2932978}, as well as \eqref{eq:Mdef} in the last equality. Using again the discrete sweeping identity, we can upper bound similarly $\mathrm{cap}^{(n)}(K)$ by $4{\rm Hm}^{(n)}(K)$ and $\xi^{(n)}(K)$ by $4{\rm Es}^{(n)}(K).$ Taking $h\leq \Cr{repul}$ and $\eps(t)=r(4t)$ for all $t>0,$ we can conclude.
\end{proof}

In order to apply Proposition~\ref{prop:condNegative}, we will use the following fact.
\begin{proposition} \label{prop:supHM}
For each $r\in{(0,1)}$ there exists a function $\eps':\N \to(0,\infty)$ with $\eps'(n)\rightarrow0$ for $n \to \infty$ such that for all $n \in \N$ and $K$ as in \eqref{eq:Lass} we have that
\begin{equation} \label{eq:supHM}
{\rm Es}^{(n)}(K) \leq \eps'(n).
\end{equation}
\end{proposition}

The proof of Proposition~\ref{prop:supHM} is provided at the end of this section. We now explain how combining Proposition~\ref{prop:supHM}  with Proposition~\ref{prop:condNegative} entails Theorem
\ref{the:maindisgff}. For this purpose, we consider the quadratic variation process $\langle \mathcal M^{\ge h} \rangle_t$ of the continuous square integrable martingale $\mathcal M_t^{\ge h},$ cf.\ also \cite[(2.3)--(2.5), (4.6)]{DiWiWu-20}. Following \cite[Lemma~IV.6.1]{kups11574} one can easily prove that
\begin{equation} \label{eq:quadVarUB}
\langle \mathcal M^{\ge h} \rangle_t = \mathrm{cap}^{(n)}(\widetilde{\mathcal E}_t^{\ge h})-\mathrm{cap}^{(n)}(\widetilde{\mathcal E}_0^{\ge h}) \quad \text{for all $t \ge 0.$}
\end{equation}

Next, we have the following standard result on continuous martingales from \cite[Chapter~V, Theorem~1.7]{ReYo-99}.
\begin{proposition} \label{prop:BMconvMart}
Let $\mathcal L_t$ be a continuous square integrable martingale, $T_t:= \inf \{s \in [0, \infty)\, : \, \langle \mathcal L\rangle_s > t\}$ its parametrization by quadratic variation and $(W_t)_{t\geq0}$ be an independent Brownian motion starting in the origin. Then the process
\begin{align*}
B_t:= 
\left\{ \begin{array}{ll}
\mathcal L_{T_t} - \mathcal L_0, \quad & t < \langle \mathcal L\rangle_\infty,\\
W_{t-\langle \mathcal L\rangle_{\infty}}+\mathcal L_\infty - \mathcal L_0, \quad & t \ge \langle \mathcal L\rangle_\infty,
\end{array} \right.
\end{align*}
where $\langle \mathcal L\rangle_\infty:= \lim_{t\to \infty} \langle \mathcal L\rangle_t$,
is a Brownian motion starting in the origin and stopped at time $\langle \mathcal L \rangle_\infty.$
\end{proposition}
 With the above results at hand, we are now ready to prove Theorem~\ref{the:maindisgff}.
\begin{proof}[Proof of Theorem~\ref{the:maindisgff}]
First note that the bound $h_*^d(r)\leq\sqrt{\pi/2}$ is a simple consequence of Theorem~\ref{the:maindisexcloops} for $\kappa=4$ and the isomorphism theorem, Proposition~\ref{prop:BEisom}. 

We will now prove the first inequality in \eqref{eq:GFFconnectionBds}, which will also imply the first inequality in \eqref{eq:h*ineq}. On the event that $B_n(r)$ is not connected to $\widehat{\partial}\D_n$ in $E_n^{\geq h},$ taking $K={\mathcal{E}}_{\infty}^{\geq h}$ and $F=\{x\in{K}:\,\varphi_x< h\},$ it is clear by definition of the exploration that \eqref{eq:Lass} and \eqref{eq:LK} are satisfied, and that the event $\mathcal{B}_{K,F}^h$ occurs. Taking advantage of \eqref{eq:equalityCaps} and above, we obtain that for $h\leq \Cr{repul}$ we have
\begin{equation}\begin{split}
\label{proof0connect2} 
&\Pm\Big(\mathcal M_t^{\ge h} > - \Cr{repul} \mathrm{cap}^{(n)}(\widetilde{\mathcal E}_t^{\ge h})\text{ for all }t\geq0, \, \{B_n(r) \stackrel{E^{\ge h}_n}{\longleftrightarrow} \widehat \partial \D_n \}^\ch\Big)\\&
    \leq\Pm\Big(\mathcal M_\infty^{\ge h} > - \Cr{repul} \mathrm{cap}^{(n)}({\mathcal E}_\infty^{\ge h}), \, \{B_n(r) \stackrel{E^{\ge h}_n}{\longleftrightarrow} \widehat \partial \D_n \}^\ch\Big)
    \\&\leq \sum_{K,F}\Pm\Big(M_K > - \Cr{repul} \mathrm{cap}^{(n)}(K), \, \mathcal{B}_{K,F}^h \Big) \\
    &\leq \sum_{K,F} \eps\big({\rm Es}^{(n)}(K)\big) \Pm( \mathcal{B}_{K,F}^h),
    \end{split}
\end{equation}
where the sum is over all $K,\, F$ as in  \eqref{eq:Lass} and \eqref{eq:LK}, and we used Proposition~\ref{prop:condNegative} to obtain the last inequality.
 Regarding the right-hand side of the previous display, it follows from Proposition~\ref{prop:supHM} -- in particular the uniformity of the bound \eqref{eq:supHM} for $K$ as in \eqref{eq:Lass} -- and the fact that $\varepsilon'(r) \to 0$ as $r \searrow 0$, that
\begin{equation} \label{eq:EsConv0}
     \limsup_{n\rightarrow \infty}
    \sum_{K,F} \eps\big({\rm Es}^{(n)}( K)\big) \Pm( \mathcal{B}_{K,F}^h)=0,
\end{equation}
since the events $\mathcal{B}_{K,F}^h,$ for $K,\, F$ as in  \eqref{eq:Lass} and \eqref{eq:LK}, are pairwise disjoint. Displays \eqref{proof0connect2} and \eqref{eq:EsConv0} immediately entail that there exists $h>0$ such that
\begin{align}
\label{eq:probaEgoesto0}
\lim_{n\to \infty}\Pm\Big(\mathcal M_t^{\ge h} > - \Cr{repul} \mathrm{cap}_A^{(n)}(\widetilde{\mathcal E}_t^{\ge h})\text{ for all }t\geq0, \, \{B_n(r) \stackrel{E^{\ge h}_n}{\longleftrightarrow} \widehat \partial \D_n \}^\ch \Big) =0.
\end{align}
Next, for $a,b>0$ let us denote by 
\[
p(a,b) := \Pm\big (W_t > -at - b \, \forall t \ge 0\big)
\] 
the probability that the standard Brownian motion $(W_t)_{t\geq0}$ stays above the line $-at-b$  for  all $t\geq0.$ From Excercise $2.16$ in~\cite{PeresBM} we know that
\begin{equation} \label{eq:persProbPos}
p(a,b)=1-e^{-2 a b}>0.
\end{equation} 
 In order to derive a lower bound on $\Pm(B_n(r) \stackrel{E^{\ge h}_n}{\longleftrightarrow} \widehat \partial \D_n  ),$ we now construct a lower bound for the unrestricted probability $\Pm\big(\mathcal M_\infty^{\ge h} > - \Cr{repul} \mathrm{cap}^{(n)}({\mathcal E}_\infty^{\ge h}) \big).$ For this purpose, denote by $(B_t^{\ge h})$ the Brownian motion defined as the time change of the continuous martingale $\M^{\ge h}_t$ through  Proposition~\ref{prop:BMconvMart}, with $\M^{\ge h}_t$ playing the role of $\mathcal L_t$. Using \eqref{eq:quadVarUB} and the continuity of the quantities involved we infer that
\begin{align} \label{proof0connect1} 
\begin{split}
 &\Pm\Big(\mathcal M_t^{\ge h} > - \Cr{repul} \mathrm{cap}^{(n)}(\widetilde{\mathcal E}_t^{\ge h})\text{ for all }t\geq0 \Big)\\
 &\quad = \Pm\Big(\mathcal M_{T_t}^{\ge h} >-\Cr{repul}\mathrm{cap}^{(n)}(\widetilde{\mathcal{E}}_{T_t}^{\ge h}) \text{ for all }t\geq0\text{ such that }\,T_t<\infty\Big)\\
& \quad \ge \Pm\Big(B_t^{\ge h}>-\Cr{repul}\big(t+\mathrm{cap}^{(n)}({\mathcal E}_0^{\ge h})\big)-\M_0^{\ge h}\text{ for all }t\geq0\Big).
\end{split}
\end{align}
Proceeding similarly as in  \cite[(IV.6.3)]{kups11574}, one can easily show that $(B_t^{\ge h})_{t\geq0}$ is independent of $\M_0^{\ge h},$ and since $\Pm(\M_0^{\ge h}\geq0)=\frac12$, we can lower bound the right-hand side of \eqref{proof0connect1}  by 
\[
\frac12 \Pm\Big(B_t^{\ge h}>-\Cr{repul}\big(t+\mathrm{cap}^{(n)}({\mathcal E}_0^{\ge h})\big)\text{ for all }t\geq0\Big). 
\]
Furthermore, since $r\in{(0,1)},$ one can easily deduce from Lemma~\ref{l.capconv} that there exists a constant $\Cl[const]{ccapE0}=\Cr{ccapE0}(r)$ such that
\begin{equation*} 
\mathrm{cap}^{(n)}(\widetilde{\mathcal{E}}_0^{\ge h})
\ge \mathrm{cap}^{(n)}(B_n(r)) \geq \Cr{ccapE0} \quad \text{ for all }n>0. 
\end{equation*}
Therefore, we have for all $n>0$ that
\begin{align}
\label{proof0connect4} 
\begin{split}
    &\frac12p(\Cr{repul}, \Cr{repul}\Cr{ccapE0})
     \leq\Pm\Big(B_t^{\ge h}>-\Cr{repul} \big(t+\mathrm{cap}_A^{(n)}(\widetilde{\mathcal E}_0^{\ge h} )\big)-\M_0^{\ge h}\text{ for all }t\geq0\Big).
\end{split}
\end{align}
Displays \eqref{proof0connect1} and \eqref{proof0connect4} then supply us with
\begin{equation}
\label{eq:proof0connect5}
\frac12p(\Cr{repul},\Cr{repul}\Cr{ccapE0}) \le
\Pm\Big(\mathcal M_t^{\ge h} > - \Cr{repul} \mathrm{cap}_A^{(n)}(\widetilde{\mathcal E}_t^{\ge h})\text{ for all }t\geq0 \Big).
\end{equation}
Combining \eqref{eq:probaEgoesto0}, \eqref{eq:persProbPos} and \eqref{eq:proof0connect5} we finally infer that 
\begin{equation*}
    \liminf_{n\rightarrow0}\Pm\big(B_n(r) \stackrel{E^{\ge h}_n}{\longleftrightarrow} \widehat \partial \D_n \big )\geq \frac12p(\Cr{repul},\Cr{repul}\Cr{ccapE0}) > 0,
\end{equation*}
which finishes the proof of the first inequality in \eqref{eq:GFFconnectionBds}. 

The second inequality in \eqref{eq:GFFconnectionBds} is trivial, so we proceed with proving the third inequality. If there is a single excursion of the random walk excursion process at level $\intens$ which encloses the set $B_n(r)$, then by duality there can be no nearest-neighbor path -- in fact, not even a $*$-connected one --  from $B_n(r)$ to $\widehat{\partial} \D_n$ in $\mathcal{\V}_n^{\intens},$ and so there is also no such path in $E^{\ge \sqrt{2\intens}}_n$ by Proposition~\ref{prop:BEisom}.
 Now for any $\intens > 0$, one can easily deduce from the coupling in Theorem~\ref{the:couplingdiscontexc1} that this random walk excursion event has a probability bounded away from $0$, uniformly in $n \in \N.$ As a consequence, the inequality follows.
\end{proof}

\begin{remark}
\label{rk:exactcritparaGFF}
Let us denote by $\mathbb{A}_{h}^{(n)}$ the union of all the clusters of $\widetilde{E}^{\ge h}_n$ hitting $\partial\D_n.$ One can use the full isomorphism between the GFF and the discrete excursion process, \cite[Proposition~2.4]{ArLuSe-20a}, to show that $\tilde{\V}_n^{u,1/2}$ (see \eqref{eq:BuuVuuDef} for definition) has the same law as the complement of $\mathbb{A}_{-\sqrt{2u}}^{(n)}.$ Therefore, we deduce from Theorem~\ref{the:maindisexcloops} with $\kappa=4$ that the probability that $B_n(r)$ is connected to $\partial B_n(1-\eps)$ in the complement of $\mathbb{A}_{-h}^{(n)}$ goes to $0$ as $n\rightarrow\infty$ and $\eps\rightarrow0$ if and only if $h\geq\sqrt{\pi/2}.$ See also Corollary~\ref{cor:percocontGFF} for the respective result in the continuous setting.

Since $E^{\ge h}_n$ is included in the complement of $\mathbb{A}_{-h}^{(n)},$ the inequality $h_*^d(r)\leq\sqrt{\pi/2}$ can be seen as a simple consequence of this result, and we conjecture that this inequality is strict. The percolation problem of $E^{\ge h}_n$ is more classical in higher dimension than the one of the complement of $\mathbb{A}_{-h}^{(n)},$ and we have thus chosen to focus on it in this article.  
\end{remark}

Let us now turn to the proof of Proposition~\ref{prop:supHM}. We are first going to need the following estimate on the capacity of sets close to $A_{n}.$

\begin{lemma}
\label{Lem:bigcaponboundary}
For each $r\in{(0,1)},$ there exists a constant $c>0$ such that for all $n\in\N$ and $K$ as in \eqref{eq:Lass} with $d(K,\widehat{\partial}\D_n)\leq n^{-\frac1{2}}$ we have
\begin{equation*}
    \mathrm{cap}^{(n)}(K)\geq {c\log(n)}.
\end{equation*}
\end{lemma}
\begin{proof}
For some $s_0 > 0$ let us define $\pi_n:[0,s_0]\rightarrow\widetilde{\D}_n$  a continuous and injective path starting in $B_n(r),$ ending at a vertex at distance at most $n^{-1/2}$ from $\widehat{\partial}\D_n,$ and such that $\pi_n(s)$ is on a cable between two vertices in $K$ for all $s\in{[0,s_0]}.$ Define also the map $p_n:[|\pi_n(0)|,|\pi_n(s_0)|]\rightarrow\widetilde{\D}_n$ via $s\mapsto\pi_n(u_s),$ where $u_s=\inf\{r\geq0:\,|\pi_n(r)|\geq s\}.$ Note that $p_n$ is measurable since, for each segment $I$ included in one of the cables crossed by $\pi_n,$ the set $p_n^{-1}(I)$ is a finite union of  intervals. Since $r\in{(0,1)},$ we can moreover assume that $|\pi_n(0)|\geq c$ for some constant $c>0$ only depending on $r.$ Let us define the measure
\begin{equation*}
    \mu:\mathcal{B}(\D)\ni A\mapsto\nu\big(p_n^{-1}(A)\big),\text{ where }\nu(\mathrm{d}t):=\frac{\mathrm{d}t}{1-t},
\end{equation*}
which is supported on $\pi_n([0,s_0]).$  Moreover, by \eqref{eq:Greenat0} we have
\begin{equation*}
    \int_{\D} G(0,x)\, \mathrm{d}\mu(x)\leq \frac{1}{2\pi}\int_{c}^{1}\frac{\log(1/t)}{1-t}\, \mathrm{d}t \, (<\infty).
\end{equation*}
Therefore, by \cite[Chapter~2, Proposition~4.2]{sznitman1998brownian} there exists a constant $c'>0$ such that
\begin{equation*}
    \mathrm{cap}(\pi_n([0,s_0]))\geq c'\mu(\pi_n([0,s_0]))\geq \frac{c'}{2\pi}\int_{c}^{1-n^{-1/2}}\frac{1}{1-t}\, \mathrm{d}t\geq c''\log(n).
\end{equation*}
It moreover follows from Proposition~\ref{t.gencapconv} that
\begin{equation*}
    \mathrm{cap}^{(n)}(K)\geq \mathrm{cap}^{(n)}(\pi_n([0,s_0])\cap \D_n)\geq c\cdot\mathrm{cap}(\pi_n([0,s_0])),
\end{equation*}
and we can conclude.
\end{proof}

\begin{proof}[Proof of Proposition~\ref{prop:supHM}]
Using the Beurling estimate Lemma~\ref{lem:beurling} and recalling $e_{K\setminus \widehat{\partial}K}^{(n)}$ as well as ${\rm cap}^{(n)}(K)$ from \eqref{defequiandcap}, for all $x\in{\widehat{\partial} (K\setminus \widehat{\partial}K)}$ we have the upper bound
\begin{equation} \label{eq:hmUB}
e_{K\setminus \widehat{\partial}K}^{(n)}(x) =\sum_{y\sim x}\Pm_y^{(n)}(\tau_{K\setminus \widehat{\partial}K}> \tau_{\partial \D_n})\leq C\left(\frac{1}{nd(x,\widehat{\partial}\D_n)}\right)^{\frac12}.
\end{equation}
Since $B_n(r)\subset K$ and $r\in{(0,1)},$ by Lemma~\ref{l.capconv} there exists a constant $c>0,$ depending on $r$ only, such that $\mathrm{cap}^{(n)}(K)\geq c.$ Moreover $d(\widehat{\partial}(K\setminus \widehat{\partial}K),\widehat{\partial}\D_n)\geq d(K,\widehat{\partial}\D_n),$ and so in view of
\eqref{eq:hmUB},
\begin{equation*}
    {\rm Es}^{(n)}(K)\leq Cn^{-\frac14} \quad \text{ if }d(K,\widehat{\partial}\D_n)\geq n^{-\frac12}.
\end{equation*}
On the other hand, if $d(K,\widehat{\partial}\D_n)\leq n^{-\frac12}$ then we can easily conclude since in view of Lemma~\ref{Lem:bigcaponboundary} the denominator in the definition \eqref{eq:supHarmMeas} of ${\rm Es}^{(n)}(K)$  is larger than $c\log(n)$ while the numerator is bounded from above by $4.$
\end{proof}

We conclude with the following.  
\begin{proof}[Proof of Corollary~\ref{cor:VunPercBds}]
When $r\in{(0,1)},$ the first inequality follows from \eqref{eq:GFFconnectionBds} and the isomorphism theorem, Proposition~\ref{prop:BEisom}. When $r=0,$ similarly as in Remark~\ref{rk:othercontperco},\ref{finiteenergy}, one can prove that there is a type of finite energy property for the percolation of $\V_n^{\intens},$ that is for each $r\in{(0,1)},$ $\Pm\big(0\stackrel{\mathcal{V}_{n}^{\intens}}{\longleftrightarrow}
		\widehat{\partial}\D_n\big)$ is larger than  $c\Pm\big(B_n(r) \stackrel{\mathcal{V}_{n}^{\intens}}{\longleftrightarrow}
		\widehat{\partial}\D_n\big)$ for a constant $c=c(r,u)$ not depending on $n$ by Proposition~\ref{prop:interoncompacts} and Lemma~\ref{l.capconv}.

The second inequality is trivial. The last inequality of \eqref{eq:ineqsVacantSetPerc} follows from the fact that with probability uniformly bounded away from $0,$ for all $n$ large enough there exists a single excursion enclosing $B_n(r)$ and hence separating it from $\widehat{\partial} \D_n,$ which follows from the coupling with continuous excursions Theorem~\ref{the:couplingdiscontexc1}.  
\end{proof}

\begin{remark}
\label{rk:h*dependonr}
Note that the finite energy property used in the proof of Corollary~\ref{cor:VunPercBds} implies that $\intens_*^d(r)$ does not depend on $r\in{[0,1)}$ -- for $h_*^d(r)$ this does not seem to be clear.
\end{remark}

We finish this section by some remarks on the percolation for the level sets of the GFF on the cable system.

\begin{remark}
\phantomsection\label{rk:cablegff}
\begin{enumerate}[label=\arabic*)]
    \item \label{rk:tilh*>0}In \cite[Corollary~5.1]{ArLuSe-20a}, the inequality $\tilde{h}_*^d(r)\geq0$ for $r\in{(0,1)}$ is proved using the explicit formulas on the effective resistance between $0$ and $\widetilde{E}_n^{\geq h}$ for $h<0$ from  \cite[Corollary~14]{MR3827222}. In fact, they prove the following stronger statement: for all $h<0$ and $r\in{(0,1)}$
    \begin{equation}
        \label{eq:connectiontboundarycable}
            0 < \liminf_{n \to \infty} \Pm \Big (B_n(r) \stackrel{\widetilde{E}^{\ge h}_n}{\longleftrightarrow} \widehat{\partial}\D_n \Big) \le \limsup_{n \to \infty} \Pm \Big (B_n(r) \stackrel{\widetilde{E}^{\ge h}_n}{\longleftrightarrow}\widehat{\partial} \D_n \Big) < 1,
    \end{equation}
    which corresponds to the statement \eqref{eq:GFFconnectionBds} for the cable system at negative levels. Here $\longleftrightarrow$ denotes connection on the cable system $\widetilde{\D}_n,$ and not discrete connection in $\D_n.$ Let us present another short proof of the first inequality in \eqref{eq:connectiontboundarycable}: it follows from Proposition~\ref{prop:interoncompacts}, \eqref{capball} and Lemma~\ref{l.capconv} that for each $u>0$ and $r\in{(0,1)},$ the limit as $n\rightarrow\infty$ of the probability that the Brownian excursion set $\be_n^{u}$ intersect $B_n(r)$ is positive, which implies the first inequality in \eqref{eq:connectiontboundarycable} by the isomorphism \eqref{eq:iso}.
    
    \item \label{rk:r=0cablegff} When $r=0,$ one can easily prove that \eqref{eq:connectiontboundarycable} does not hold, contrary to what is stated in  \cite[Corollary~5.1]{ArLuSe-20a}. Indeed, by either \cite[Theorem~3.7]{DrePreRod3} or  \cite[Corollary~1]{MR3827222} we have for all $h>0$
    \begin{equation*}
        \Pm\Big(0\stackrel{\widetilde{E}^{\ge -h}_n}{\longleftrightarrow}\partial\D_n\Big)=\Pm\left(\varphi_0^{(n)}\in{(-h,h)}\right)\tend{n}{\infty}0
    \end{equation*}
    since $G^{(n)}(0,0),$ the variance of $\varphi_0^{(n)},$ diverges as $n\rightarrow\infty.$ Moreover, if $0\leftrightarrow\widehat{\partial}\D_n$ in $\widetilde{E}_n^{\geq -h},$ then $0\leftrightarrow{\partial}\D_n$ in $\widetilde{E}_n^{\geq -h}$ with positive probability, since the probability to cross the last edge between $\widehat{\partial}\D_n$ and $\partial\D_n$ is positive conditionally on $\varphi^{(n)}_{|\D_n}.$
    
    In fact, using the asymptotic $G^{(n)}(0,0)\geq c\log(n)$ as $n\rightarrow\infty,$ one can show that the probability that $0\leftrightarrow\widehat{\partial}\D_n$ in $\widetilde{E}_n^{\geq -(\log(n))^a}$ goes to $0$ as $n\rightarrow\infty$ for all $a<1/2.$ This indicates that,  for the level sets of the GFF, it is way harder to connect $0$ to the boundary of $\D_n,$ than $B_n(r)$ to the boundary of $\D_n$ for $r\in{(0,1)}.$ Therefore, it would not be surprising that $h_*^d(0)\leq0$ for the dGFF, hence the restriction to $r>0$ in Theorem~\ref{the:maindisgff}, contrary to the case of the excursion vacant set in Theorem~\ref{the:maindisexc}. In \cite{2dGFFperc} it is actually argued that $h_*^d(0)=0.$
\end{enumerate}
\end{remark}

\appendix
\section{Appendix: Continuum critical values}
\label{sec:contperco}
\renewcommand*{\thetheorem}{A.\arabic{theorem}}
\renewcommand{\theequation}{A.\arabic{equation}}

In this appendix, we prove Theorem~\ref{t.mainthm}. Recall the definitions of $\V^{\intens,\lambda}$ from below \eqref{eq:defcontloops}, of $\lambda(\kappa)$ from \eqref{eq:defckappa}, and of $B(r)\stackrel{\CV^{\intens,\lambda(\kappa)}}{\longleftrightarrow}\partial\D$ from above \eqref{def:u*c}. Given the link to the SLE process the idea for the computation of the critical values is very simple: consider independent excursion clouds starting and ending on the upper and lower parts of the unit circle $\mathbb{T}_{\pm}=\{z\in{\partial\D}:\pm\Im(z)\geq0\}.$ By Lemma~\ref{l.brownianexcursionandlooprestriction} and conformal invariance, the boundaries of the excursion clouds plus Brownian loop clusters $\partial_{\mathbb{T}_\pm}\mathcal{I}^{u,\lambda(\kappa)}_{\mathbb{T}_\pm,\mathbb{T}_\pm,\D}$ are SLE$_{\kappa}(\rho)$ paths in $\mathbb{D}$ connecting $-1$ with $1$.
The $\rho(\intens)$ relation is such that the SLE paths intersect the boundary away from $-1$ and $1$ if and only if $\intens < (8-\kappa)\pi/16$ (Lemma~\ref{lem:SLE-ka-r}), and the excursions plus loops do not separate $0$ from $\partial \mathbb{D}$ with positive probability. We refer to Figure \ref{fig:superandsubcrit} for a simulation when there is no loop soup, that is when $\kappa=8/3.$ 
\begin{figure}[ht!]
\setlength\arrayrulewidth{0.8pt}
\noindent\begin{tabular}{|c|c|} 
\hline\rule{0pt}{42.4ex}\hspace{-1.4mm}
 \includegraphics[width=6.53cm]{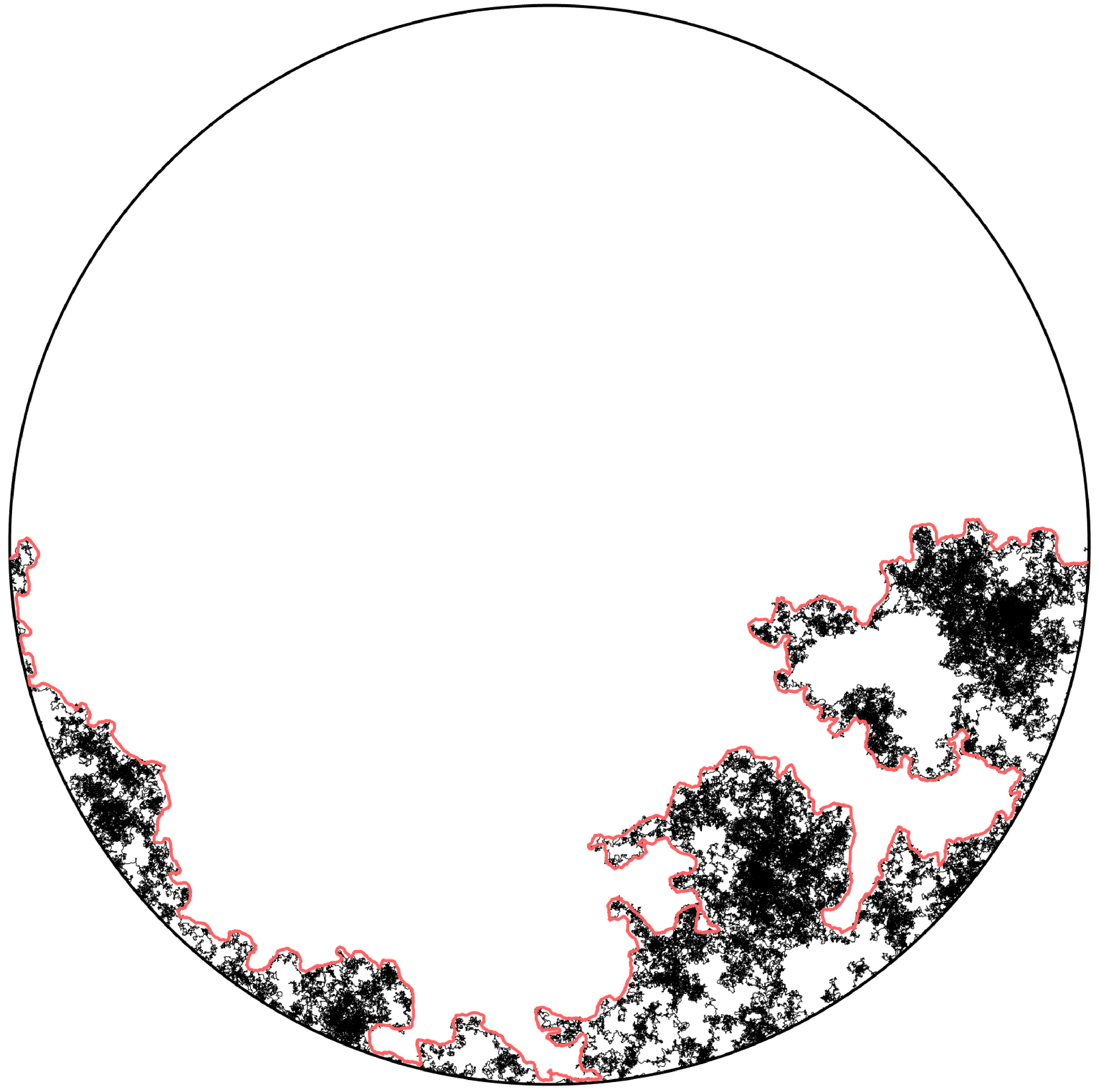} & 
\includegraphics[width=6.53cm]{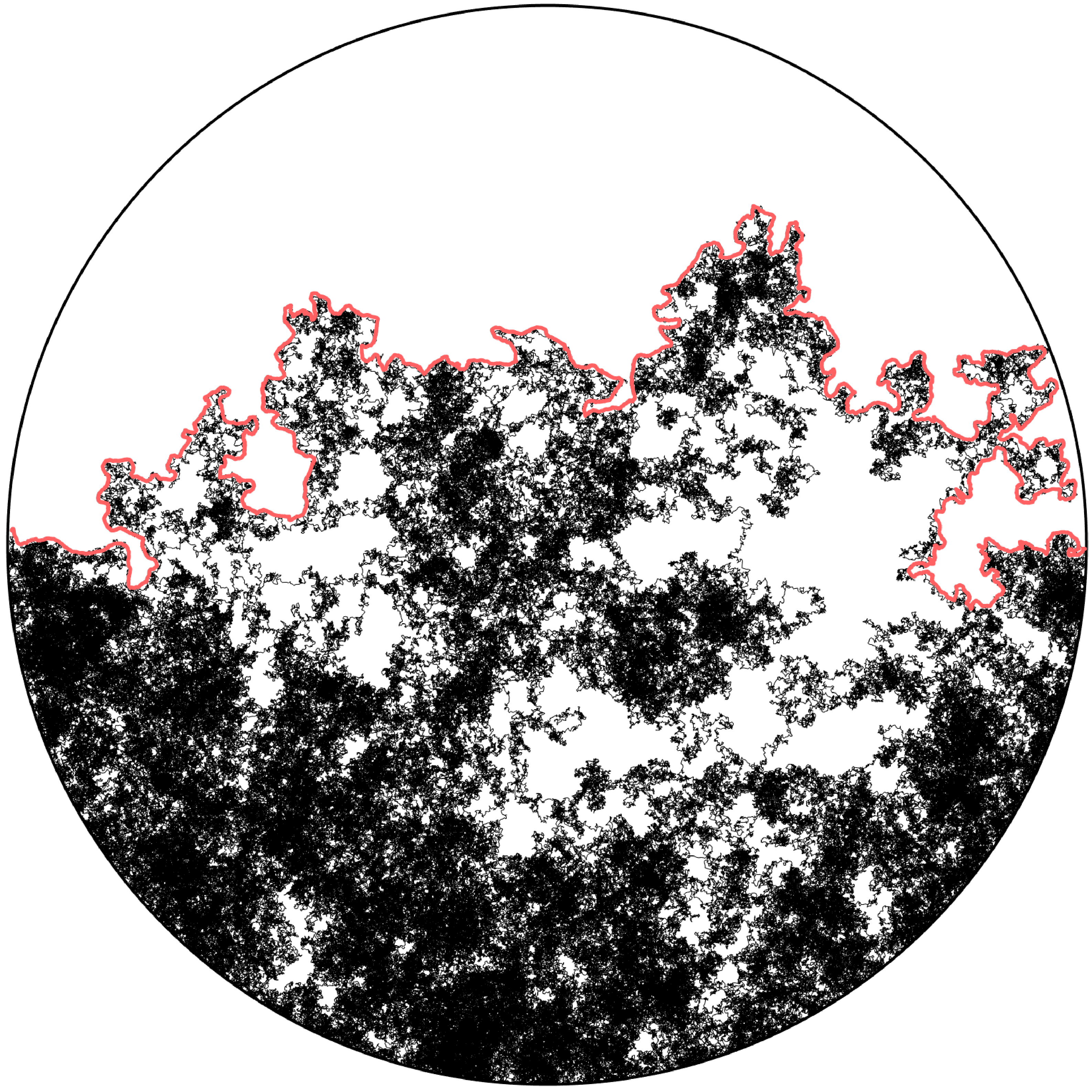} 
\\ \hline
\end{tabular}
\caption{A simulation of the Brownian excursion set $\mathcal{I}^u_{\mathbb{T}_-,\mathbb{T}_-,\D}$ in black, and the corresponding $\text{SLE}_{8/3}(\rho_{8/3}(u/\pi))$ curve $\partial_{\mathbb{T}_-}\mathcal{I}^u_{\mathbb{T}_-,\mathbb{T}_-,\D}$ in red.  On the left the supercritical case where $u<\pi/3,$ and on the right the subcritical case where $u\geq\pi/3.$}
\label{fig:superandsubcrit}
\end{figure}

This reasoning suggests that the critical parameter associated to $\mathcal{V}^{u,\lambda(\kappa)}$ is equal to $(8-\kappa)\pi/16,$ as noted in \cite[Section~5]{werner-qian}. But some points need to be dealt with in order to make this rigorous. For instance, if $u<(8-\kappa)\pi/16,$ then even if $0$ is not disconnected from $\partial\D$ by  the SLE paths $\partial_{\mathbb{T}_-}\mathcal{I}^{u,\lambda(\kappa)}_{\mathbb{T}_-,\mathbb{T}_-,\D}$ and $\partial_{\mathbb{T}_+}\mathcal{I}^{u,\lambda(\kappa)}_{\mathbb{T}_+,\mathbb{T}_+,\D},$ $0$ could still be disconnected from $\partial\D$ by $\mathcal{I}^{u,\lambda(\kappa)},$ via trajectories starting in $\mathbb{T}_+$ and ending in $\mathbb{T}_-.$ Similarly, if $u\geq (8-\kappa)\pi/16,$ even if $\partial_{\mathbb{T}_-}\mathcal{I}^{u,\lambda(\kappa)}_{\mathbb{T}_-,\mathbb{T}_-,\D}$ does not hit $\partial\D,$ $0$ could be connected to $\mathbb{T}_-$ in $\mathcal{V}^{u,\lambda(\kappa)}$ if the curve $\partial_{\mathbb{T}_-}\mathcal{I}^{u,\lambda(\kappa)}_{\mathbb{T}_-,\mathbb{T}_-,\D}$ passes ``above" $0.$ 

\begin{lemma}\label{lem:1}
Suppose $\kappa\in{[8/3,4]}$ and $0<\intens<(8-\kappa)\pi/16,$ then $\Pm\big(0\stackrel{\CV^{\intens,\lambda(\kappa)}}{\longleftrightarrow}\partial\D\big)>0.$
\end{lemma}
\begin{proof}
Let $A\subset\D$ be an arbitrary choice of deterministic line segment connecting $0$ with the interior of $\mathbb{T}_-$ (seen as a subset of $\partial\D$) inside $\D$ and let $A_{\lambda(\kappa)}$ be the closure of the set of clusters of loops at intensity $\lambda(\kappa)$ intersecting $A.$ By Lemmas~\ref{l.brownianexcursionandlooprestriction} and \ref{lem:SLE-ka-r}, $\eta : = \partial_{\mathbb{T}_+}\be^{\intens,\lambda(\kappa)}_{\mathbb{T}_+,\mathbb{T}_+,\D}$ is an SLE$_\kappa(\rho)$ 
curve in $\D$ from $-1$ to $1$ which intersects $\mathbb{T}_+$ away from the end points a.s.\ and does not intersect $\mathbb{T}_-$. Therefore, on the event 
\begin{equation}
\label{eq:defE1}
    E_1:=\big\{\be^{\intens,\lambda(\kappa)}_{\mathbb{T}_+,\mathbb{T}_+,\D}\cap A=\varnothing\big\}=\big\{\be^{\intens}_{\mathbb{T}_+,\mathbb{T}_+,\D}\cap A_{\lambda(\kappa)}=\varnothing\big\},
\end{equation}
the path $\eta$ does not separate $0$ from $\mathbb{T}_-$ in $\D$ and so on $E_1$ there exists a (random) closed path $\gamma$ in $(\be^{\intens,\lambda(\kappa)}_{\mathbb{T}_+,\mathbb{T}_+,\D})^\ch$ from $0$ to $\mathring{\mathbb{T}}_+.$ 
 
 We now claim that $\be^{\intens,\lambda(\kappa)}\setminus \be^{\intens,\lambda(\kappa)}_{\mathbb{T}_+,\mathbb{T}_+,\D}$ avoids $\gamma$ with positive probability. Let $\Gamma$ be a (random) closed arc in $\partial\D$ not intersecting $\gamma$ and containing $\mathbb{T}_-$ in its interior. We will resample $\be^{\intens}_{\mathbb{T}_+,\mathbb{T}_+,\D}$ in order to be able to consider conditionally independent ``good'' events on which $\be^{\intens,\lambda(\kappa)}\setminus \be^{\intens,\lambda(\kappa)}_{\mathbb{T}_+,\mathbb{T}_+,\D}$ avoids $\gamma$. The probabilities of the good events are strictly positive using the restriction formula \eqref{eq:onesidedrestriction}. Now for the details. Let $\widehat{\be}^{\intens}_{\mathbb{T}_+,\mathbb{T}_+,\D}$ be an independent random  set with the same law as ${\be}^{\intens}_{\mathbb{T}_+,\mathbb{T}_+,\D},$ and let $\widehat{\be}^{\intens}$ be the union of $\widehat{\be}^{\intens}_{\mathbb{T}_+,\mathbb{T}_+,\D}$ and the set of  points intersected by an excursion in the support of $\omega_u$ which does not start and end in $\mathbb{T}_+$. Then $\widehat{\be}^{\intens}$  has the same law as $\be^{\intens}.$ For any $\Gamma'\subset\partial\D,$ we also define $\widehat{\be}_{\Gamma',\Gamma',\D}^{\intens},$ $\widehat{\be}^{\intens,\lambda(\kappa)}$ and $\widehat{\be}_{\Gamma',\Gamma',\D}^{\intens,\lambda(\kappa)}$ analogously to  ${\be}_{\Gamma',\Gamma',\D}^{\intens},$ $\be^{\intens,\lambda(\kappa)}$ and $\be^{\intens,\lambda(\kappa)}_{\Gamma',\Gamma',\D},$ but for the excursions associated with $\widehat{\be}^{\intens}$ instead of ${\be}^{\intens}.$ On the event that $\gamma$ exists, let $\gamma_{\lambda(\kappa)}$ be the closure of the set of clusters of loops at intensity $\lambda(\kappa)$ hitting $\gamma,$ and define
\begin{equation}
\label{eq:defE2}
    E_2:=\big\{\widehat{\be}^{\intens,\lambda(\kappa)}_{\Gamma,\Gamma,\D}\cap \gamma=\varnothing\big\}=\big\{\widehat{\be}^{\intens}_{\Gamma,\Gamma,\D}\cap \gamma_{\lambda(\kappa)}=\varnothing\big\}.
\end{equation}
 Let $E_3$ be the event that $\widehat{\be}^{\intens}$ contains no excursions starting on $\mathbb{T}_-$ and ending on $\Gamma^\ch,$ or starting on $\Gamma^\ch$ and ending on $\mathbb{T}_-.$ Each loop cluster intersecting $\be^u$ but not $\be^u_{\mathbb{T}_+,\mathbb{T}_+,\D}$ intersects an excursion of $\widehat{\be}^u$ which does not start and end on $\mathbb{T}_+,$ and thus also intersects $\widehat{\be}^{u}_{\Gamma,\Gamma,\D}$ on the event $E_3.$ Therefore, the  path $\gamma$ does not intersect $\be^{\intens,\lambda(\kappa)}\setminus \be^{\intens,\lambda(\kappa)}_{\mathbb{T}_+,\mathbb{T}_+,\D}$ on the event $E_2\cap E_3,$ and so $\gamma$ is included in $\V^{\intens,\lambda(\kappa)}$ on the event $E_1\cap E_2\cap E_3.$ We refer to Figure~\ref{F:Lemma4.1} for details.
 
 \begin{figure}[ht!]
  \centering 
  \includegraphics[scale=0.76]{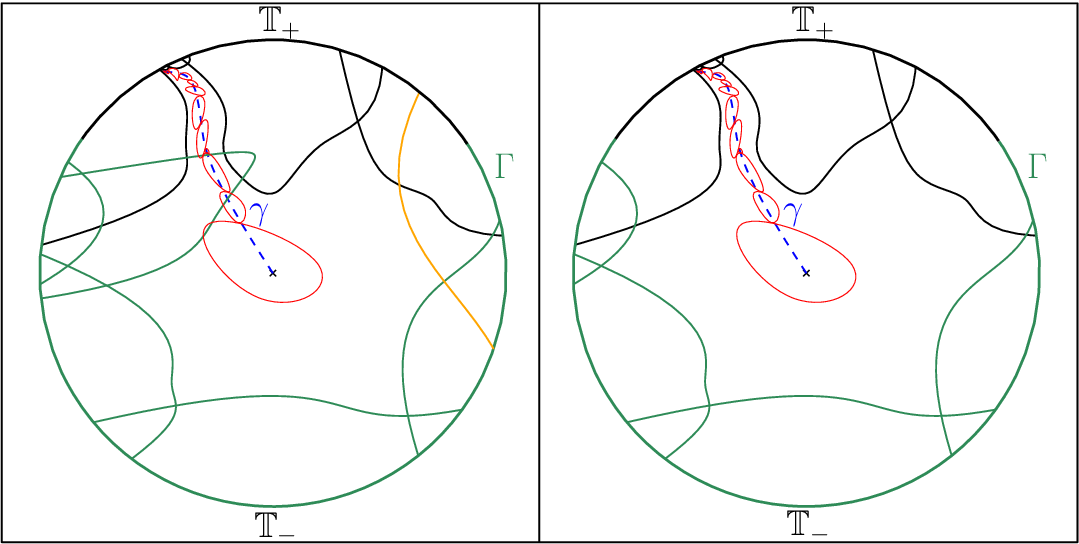}
  \caption{Some selected Brownian excursions, in black for $\be^{\intens}_{\mathbb{T}_+,\mathbb{T}_+,\D},$ in green for $\widehat{\be}^{\intens}_{\Gamma,\Gamma,\D},$ and in orange for the excursions of $\widehat{\be}^{\intens}$ starting on $\mathbb{T}_-$ and ending on $\Gamma^\ch,$ or starting on $\Gamma^\ch$ and ending on $\mathbb{T}_-.$ In blue the path $\gamma$ and in red the clusters of the loop soup at level $\lambda(\kappa)$ hitting $\gamma.$ On the left the event $E_2$ does not occur since there is a green trajectory hitting a red cluster, and the event $E_3$ also does not occur since there is an orange trajectory. On the right the events $E_2$ and $E_3$ occur, and thus $\gamma$ is a path from $0$ to $\partial\D$ in  $\V^{\intens,\lambda(\kappa)}.$} 
  \label{F:Lemma4.1}
\end{figure}

Note that it follows from the local finiteness of loop clusters, see \cite[Lemma~9.7 and Proposition~11.1]{sheffield-werner}, that both $\overline{A_{\lambda(\kappa)}}\cap\partial\D\subset\mathring{\mathbb{T}}_-$ and $\overline{\gamma_{\lambda(\kappa)}}\cap\partial\D\subset\mathring{\Gamma^\ch}.$ Therefore, conditionally on the loop soup, by Lemma~\ref{lem:excarerestriction}, the probabilities of $E_1,$ and of $E_2$ conditionally on $\be^{\intens}_{\mathbb{T}_+,\mathbb{T}_+,\D},$ can both be computed using \eqref{eq:onesidedrestriction} (with a random set to avoid), and in particular we see that these probabilities are strictly positive. Moreover, the probability of $E_3$ conditionally on the loop soup and $\be^{\intens}_{\mathbb{T}_+,\mathbb{T}_+,\D}$ is positive since the restriction of the excursion measure $\mu$ to excursions starting on a closed arc and ending on a disjoint closed arc is finite, see below (5.10) in \cite{lawler2005conformally}. Since the events $E_2$ and $E_3$ are independent conditionally on $\be^{\intens}_{\mathbb{T}_+,\mathbb{T}_+,\D}$ and the loop soup, we are done.
\end{proof}

We now turn to the subcritical case. As we now explain, when $u\geq(8-\kappa)\pi/16$ it is easy to construct from Lemmas~\ref{l.brownianexcursionandlooprestriction} and \ref{lem:SLE-ka-r} a curve in $\be^{\intens,\lambda(\kappa)}$ which surrounds $B(r).$ 

\begin{lemma}\label{lem:2}
Suppose $\kappa\in{[8/3,4]}$ and $\intens\geq(8-\kappa)\pi/16.$ Then $\Pm\big(B(r)\stackrel{\CV^{\intens,\lambda(\kappa)}}{\longleftrightarrow}\partial\D\big)=0$ for all $r\in{[0,1)}.$ 
\end{lemma}
\begin{proof}
For $k\in{\{0,1\}},$ let $\Gamma_k=\{e^{i\theta}:\,\theta\in{[-\pi/4,\, 5\pi/4]+k\pi}\}$ and $D_k=\{tx, \, t\in{(r,1)}, \, x\in{\Gamma_i}\}.$ By Lemmas~\ref{l.brownianexcursionandlooprestriction} and \ref{lem:SLE-ka-r}, the chord $\partial_{\Gamma_k}\be_{\Gamma_k,\Gamma_k,D_k}^{\intens,\lambda(\kappa)}$ almost surely does not intersect $\Gamma_k$ except at the start and end points, and it does not intersect $B(r)$ by definition of $D_k.$ Therefore, it therefore separates $B(r)$ from $\mathring{\Gamma_k},$ and since $\mathring{\Gamma_0}\cup\mathring{\Gamma_1}=\partial\D,$ we are done. 
\end{proof}

Note that while our proof of Lemma~\ref{lem:2} exploits the restriction property from Lemma~\ref{lem:resexc}, the proof of Lemma~\ref{lem:1} does not use it.

\begin{proof}[Proof of Theorem~\ref{t.mainthm}] The proof follows immediately from Lemma~\ref{lem:1} and Lemma~\ref{lem:2} noting that $\Pm\big(B(r)\stackrel{\CV^{\intens,\lambda(\kappa)}}{\longleftrightarrow}\partial\D\big)\geq\Pm\big(0\stackrel{\CV^{\intens,\lambda(\kappa)}}{\longleftrightarrow}\partial\D\big).$
\end{proof}

\section{Appendix: Coupling of random walk and Brownian motion}
\label{sec:KMT}
\renewcommand*{\thetheorem}{B.\arabic{theorem}}
\renewcommand{\theequation}{B.\arabic{equation}}

In order to obtain our approximation result of Brownian excursions by random walk excursions, see Theorem~\ref{the:couplingdiscontexc1}, one first needs to approximate one Brownian motion by one random walk. To do this we shall utilize the KMT coupling \cite{MR375412}, and refer to \cite[Theorem~7.6.1]{MR2677157} for the version that we use here. For a random walk $X,$ we define $X_t$ for non-integer $t$ by linear interpolation. 
\begin{theorem}
\label{the:KMT}
There exists $c<\infty$ and a coupling $\Pm$ between a Brownian motion $B$ on $\R^2$ and a random walk $Y$ on $\Z^2$ both starting at $0$ satisfying for all $n\in\N$
\begin{equation*}
    \Pm \left(\max_{0 \leq t \leq 2n^3 } \left|\frac{1}{\sqrt{2}} B_t-Y_t  \right| > c\log(n) \right) \leq \frac{c}{n}.
\end{equation*}
\end{theorem}

We now turn Theorem~\ref{the:KMT} into a coupling between Brownian motions on $\D,$ that is killed on hitting $\partial\D,$ and random walks on $\D_n,$ that is killed on hitting $\partial\D_n.$

\begin{corollary}
\label{cor:KMTonD}
There exists $c<\infty$ such that, for all $n\in\N,$ $z\in{\D}$ and $x\in{\D_n},$ there exists a coupling $\Pm_{z,x}^{(n)}$ between a two-dimensional Brownian motion $Z$ on $\D$ starting in $z$ and a random walk $X^{(n)}$ on $\D_n$ starting in $x$ satisfying
\begin{equation*}
    \Pm_{z,x}^{(n)}\left(\sup_{t\in{[0,\overline{\tau}_{\partial\D}]}}|Z_t-\widehat{X}_t^{(n)}|>|z-x|+ \frac{c\log(n)}{n}\right)\leq \frac{c}{n},
\end{equation*}
where $\widehat{X}_t^{(n)}:=X_{2tn^2}^{(n)}$ for all $t\geq0$ and $\overline{\tau}_{\partial\D}:=\inf\{t\geq0:|\widehat{X}_t^{(n)}|\vee|Z_t|=\infty\}$ is the first time at which either $\widehat{X}^{(n)}$ or $Z$ are killed, with the convention that, after being killed, $X$ and $Z$ are both equal to a cemetery point $\Delta$ with $|\Delta|=\infty.$
\end{corollary}
\begin{proof}
    Note that, under the coupling from Theorem~\ref{the:KMT}, $B^{(n)}:=(\frac1{\sqrt{2}n}B_{2tn^2}+z)_{t\geq0}$ is a Brownian motion starting in $z,$  $Y^{(n)}:=(\frac1nY_k+x)_{k\geq0}$ is a two-dimensional random walk on $\frac1n\Z^2$ starting in $x,$ and
\begin{equation*}
    \Pm\left(\sup_{t\in{[0,n]}}|B^{(n)}_t-\widehat{Y}_t^{(n)}|>|x-y|+ \frac{c\log(n)}{n}\right)\leq \frac{c}{n},
\end{equation*}
where $\widehat{Y}_t^{(n)}:=Y^{(n)}_{2tn^2}$ for all $t\geq0.$ Letting $\overline{\tau}_{\partial\D}:=\inf\{t\geq0:|B^{(n)}_t|\vee|\widehat{Y}_t^{(n)}|\geq1\},$ we moreover have
\begin{equation*}
\Pm(\overline{\tau}_{\partial\D}\geq n)\leq \Pm(|B^{(n)}_n|\leq1)\leq\exp(-cn).
\end{equation*}
We can conclude by defining under some probability $\Pm_{z,x}^{(n)}$ the processes $(Z,X^{(n)})$  with the same law as $(B^{(n)},Y^{(n)})$ killed respectively on hitting $\partial \D$ and $\partial \D_n$ under $\Pm.$
\end{proof}

In order to couple discrete and continuous excursions, we need to use the coupling  from Corollary~\ref{cor:KMTonD} until the last exit time $\overline{L}_{K}:=\sup\{t\geq0:\widehat{X}_t^{(n)}\in{K}\text{ or }Z_t\in{K}\}$ of a set $K\Subset\D$ by both $\widehat{X}^{(n)}$ and $Z,$ with the convention $\sup\varnothing=0.$ Recall also the convention $Z_t=\Delta$ for all $t$ after $Z$ has been killed with $|\Delta|=\infty,$ and similarly for $X^{(n)}.$

\begin{lemma}
\label{lem:KMTuntilboundary}
Under the coupling from Corollary~\ref{cor:KMTonD}, there exists $c>0$ and $s_0>0$ such that for all $K\Subset\D$ $s\geq s_0,$ $z\in{\D},$ $x\in{\D_n}$ with $|z-x|\leq \frac{s\log(n)}{2n},$ 
\begin{equation}
\label{eq:KMTuntilboundary}
    \Pm_{z,x}^{(n)}\left(\sup_{t\in{[0,\overline{L}_{K}]}}|Z_t-\widehat{X}_t^{(n)}|> \frac{s\log(n)}{n}\right)\leq \frac{cs\log(n)}{n(1-r)},
\end{equation}
where $r=\sup\{|z|:\,z\in{K}\}.$
\end{lemma}
\begin{proof}
Abbreviate $\eps_n=s\log(n)/n,$ and note that we can assume without loss of generality that $1-r\geq \eps_n.$ By Corollary~\ref{cor:KMTonD}, we have for all $s$ large enough
\begin{equation*}
    \Pm_{z,x}^{(n)}(\overline{L}_{K}\geq \overline{\tau}_{\partial\D})\leq \Pm_z\left(L_{K}\geq \tau_{B(1-\eps_n)^\ch}\right)+ \Pm_x^{(n)}\left(L_{K_n}^{(n)}\geq \tau_{B_n(1-\eps_n)^\ch}^{(n)}\right)+\frac{c}{n}.
\end{equation*}
 It moreover follows from \eqref{eq:hittingBM} and the strong Markov property that
\begin{equation*}
    \Pm_z\left(L_{B(r)}\geq \tau_{ B(1-\eps_n)^\ch}\right)\leq\frac{\log(1-\eps_n)}{\log(r)}\leq \frac{cs\log(n)}{n(1-r)}
\end{equation*}
since $1-x \leq \log(1/x), x \in (0,1)$ and $|\log(1-x)|<2x$ for $x$ small enough. Using a similar formula for the random walk, see for instance \cite[Lemma~2.1]{MR3737923}, we have similarly
\begin{equation*}
    \Pm_x^{(n)}\left(L_{B_n(r)}^{(n)}\geq \tau_{B_n(1-\eps_n)^\ch}^{(n)}\right)\leq\frac{\log(1-\eps_n)+c/(nr)}{\log(r)}\leq \frac{cs\log(n)}{n(1-r)}
\end{equation*}
for all $r\geq\frac12.$ Since the above probability is increasing in $r$ and $K\subset B(r),$
 we can conclude.
\end{proof}

Finally, to establish a coupling between the normalized equilibrium measures, we need a bound on the distance between the last exit time of a ball for $Z$ and for $X^{(n)},$ which bears some similarities with \cite[Proposition~7.7.1]{MR2677157}. Recall the definition of $L_B$ from above \eqref{e.eqmeas} and $L_{B_n}^{(n)}$ from \eqref{eq:deftauLn}.

\begin{lemma}\label{l.gtype}
Under the coupling from Corollary~\ref{cor:KMTonD}, there exists $c>0$ and $s_0<\infty$ such that for all $r\in{(1/2,1)}$ and $s\geq s_0,$ 
\begin{equation}
\label{e:boundlastexit}
\Pm_{0,0}^{(n)}\left(\left| {X}^{(n)}_{{L}_{B_n}^{(n)}}-Z_{L_{B }}\right|\geq \frac{s \log(n)}{n}  \right) \leq \frac{c }{s }+\frac{c\log(n)}{n(1-r)},
\end{equation}
where $B=B(r)$ and $B_n=B\cap\D_n.$
\end{lemma}
\begin{proof}
First recall that by Lemma~\ref{lem:KMTuntilboundary} we can couple a time-changed random walk $X^{(n)}$ on $\D_n$ and a Brownian motion $Z$ on $\D$ such that 
\begin{equation}\label{e.lekmt}
    |\widehat{X}_t^{(n)}-Z_t|\leq  \frac{c\log(n)}{n} , \forall t \leq \overline{L}_{B(r)}=\sup\{t\geq0:|\widehat{X}_t^{(n)}|\wedge|Z_t|\leq r\},
\end{equation}
with probability at least $1-  \frac{c\log(n)}{n(1-r)}.$
Let us denote by $\widehat{L}_{B_n}^{(n)}:=\sup\{t\geq0:\widehat{X}^{(n)}_t\in{B_n}\}$ the last exit time of $B_n$ by $\widehat{X}^{(n)},$ and then $\widehat{X}^{(n)}_{\widehat{L}_{B_n}^{(n)}}=X^{(n)}_{L_{B_n}^{(n)}}.$ Let
\begin{equation}
\label{e:defrhoL}
\rho_n^{\pm}:=r\pm\frac{3c\log(n)}{n}\text{ and } \gamma_n^{\pm} :=L_{B(\rho_n^{\pm})}.
\end{equation}
Note that $\widehat{L}_{B_n}^{(n)} \in [\gamma_n^-,\gamma_n^+]$ on the event \eqref{e.lekmt}, and $L_B\in [\gamma_n^-,\gamma_n^+].$ In particular, on the event \eqref{e.lekmt},
\begin{align*}
|{X}^{(n)}_{{L}_{B_n}^{(n)}}-Z_{L_{B}}| &\leq 
\sup_{\gamma_n^- \leq t \leq \gamma_n^+ } |Z_t-Z_{L_{B }}|+|Z_{\widehat{L}_{B_n}^{(n)}}-\widehat{X}^{(n)}_{\widehat{L}_{B_n}^{(n)}}|
\\&\leq 2 \sup_{\gamma_n^- \leq t \leq \gamma_n^+  }|Z_t-Z_{\gamma_n^+}|+\frac{c\log(n)}{n}.
\end{align*}
Hence
\begin{equation}\label{e:bmgambler}
\begin{split}
\Pm_{0,0}^{(n)}&\left( \left| {X}^{(n)}_{{L}_{B_n}^{(n)}}-Z_{L_{B}}\right| \geq \frac{(2s+c) \log(n)}{n}  \right) \\ 
&\leq \Pm_{0}\left( \sup_{\gamma_n^- \leq t \leq \gamma_n^+} | Z_t-Z_{\gamma_n^+}|\geq \frac{s \log(n)}{n}  \right) +\frac{c\log(n)}{n(1-r)}.
\end{split}
\end{equation}

According to \cite[p.74]{MR521533}, conditionally on $Z_{\gamma_n^-},$ the law of $(Z_t)_{t\geq \gamma_n^-}$ is independent of $(Z_t)_{t\leq \gamma_n^-}.$ Since $\sigma_{\rho_n^-},$ the uniform measure on $\partial B(\rho_n^-),$ is the law of $Z_{\gamma_n^-}$ starting from either $0$ or $\sigma_{\rho_n^+}$ by rotational invariance, the probability on the last line of \eqref{e:bmgambler} is equal to
\begin{align*}
&\Pm_{\sigma_{\rho_n^+}}\left( \sup_{\gamma_n^- \leq t \leq \gamma_n^+} | Z_t-Z_{\gamma_n^+}\,|\,\geq \frac{s \log(n)}{n}\,\Big|\,\gamma_n^->0\right)
\\&= \Pm_{\sigma_{\rho_n^+}} \left( \sup_{0 \leq t \leq  \tau_{B(\rho_n^-)} } | Z_t-Z_{0}|\geq \frac{s \log(n)}{n} \,\Big|\, \tau_{B(\rho_n^-)}<\tau_{\partial \D } \right),
\end{align*}
where the last equality follows from invariance under time-reversal of the Brownian motion, see for instance \cite[Theorem~24.18]{kallenberg_2002}.
Therefore, using again rotational invariance and changing notation, which makes the rest of the argument more clear, we obtain
\begin{equation}
\label{e:proofboundonexit0}
    \Pm_0\left( \sup_{\gamma \leq t \leq \bar{\gamma}} | Z_t-Z_{\bar{\gamma}}|\geq \frac{s \log(n)}{n} \right)=\Pm_{{\rho_n^+}}\left( Z[0,\tau_{B(\rho_n^-)}] \not\subset B\big(\rho_n^+,\frac{s \log(n)}{n}\big) \,\Big |\, \tau_{B(\rho_n^-)} < \tau_{\partial \D} \right).
\end{equation}
First note that by \eqref{eq:hittingBM},
\begin{equation}
\label{e:proofboundonexit3}
\Pm_{\rho^+_n}(\tau_{B(\rho_n^-)} < \tau_{\partial \D})=\frac{|{\log( \rho^+_n)}|}{{|\log( \rho^-_n)}|}\geq c
\end{equation}
if $\log(n)/n\leq c|\log(r)|,$ which we can assume without loss of generality since $1-r\leq |\log(r)|.$ We now bound the right-hand side of \eqref{e:proofboundonexit0} without the conditioning. Using conformal invariance of Brownian motion under $z\mapsto \log(z),$ see for instance \cite[Theorem~2.2]{lawler2005conformally}, we moreover have
\begin{equation}
    \label{e:proofboundonexit00}
\begin{split}
\Pm_{{\rho_n^+}}&\left( Z[0,\tau_{B(\rho_n^-)}] \not\subset B(\rho_n^+,\frac{s \log(n)}{n}) , \tau_{B(\rho_n^-)} < \tau_{\partial \D} \right) \\
&\leq  \Pm_{\log(\rho_n^+)} \left( W[0,\tau_{ \log( \rho_n^- )}'] \not\subset \log\Big(B\big(\rho_n^+, \frac{s\log(n)}{n}\big)\Big) \right)
\end{split}
\end{equation}
where $W$ is a Brownian motion started at $\log(\rho_n^+)$ and killed upon hitting $\{z\in{\mathbb{C}}:\Re(z)<0,\Im(z)\in{(-\pi,\pi)}\}^{\ch},$  and $\tau_a'$ is the minimum between the killing time of $W$, and the first hitting time of $\left\{ \Re(z) = a\right\}$ by $W$ for each $a<0.$ Now using the inequality $|e^x-e^y|\leq C|x-y|$ for all $x,y\in{\{z:|z|\leq |\log(r)|+1\}},$ one can find a constant $c>0$ such that, 
\begin{equation}
\label{e:inclusionballs}
B\left(\log(\rho_n^+), \eps_n \right) \subset \log \left( B\Big(\rho_n^+, \frac{s \log(n)}{n}\Big)\right),\text{ where }\eps_n=\frac{cs\log(n)}{n}.
\end{equation}
Hence we only need to estimate the probability that
\(
\{ W[0,\tau_{\log( \rho_n^- )}'] \subset B(W_0, \eps_n) \},
\) conditionally on $\tau'_{\log(\rho_n^-)}<\tau'_0,$
and the remainder part of the proof is to obtain a lower bound on this probability. We begin with some elementary geometrical observations. First by \eqref{e:defrhoL} the distance between the two points $\log(\rho_n^\pm)$ is $\log(\rho_n^+/\rho_n^-)$ and satisfies
\begin{equation}
\label{e:boundrho}
\log(\rho_n^+/\rho_n^-) \leq \frac{c\log(n)}{n}.
\end{equation}
Thus for $s$ large enough, the ball $B\left(\log(\rho_n^+), \eps_n\right)$ intersects the line 
\(
\{\Re (z) = \log( \rho_n^-) \},
\)
and, letting $h$ denote the distance from the point $\log(\rho_n^+)+\eps_n/2$ and one of the points $\{\Re(z)= \log( \rho_n^+)+\eps_n/2 \} \cap B\left(\log(\rho_n^+), \eps_n\right),$ and $h'$ the distance between $\log( \rho_n^-)$ and one of the points $\{\Re(z)= \log( \rho_n^{-} ) \} \cap B\left(\log(\rho_n^+), \eps_n\right)$ we have for $s$ large enough
\begin{align}
\label{e:defh*}
h'= \sqrt{\eps_n^2 -\log\left( \frac{\rho_n^+}{\rho_n^-} \right)^2  
} \geq \frac{\sqrt{3}}{2}\eps_n=  h.
\end{align}

Now we can estimate the probability as follows. For each $t$ before $W$ is killed on $\{\Im(z)\in{[-\pi,\pi]^\ch}\},$ we can write $W_t := Y^1_t + \im Y^2_t,$ where $Y^1$ and $Y^2$ are two independent one-dimensional Brownian motions, and then, assuming without loss of generality that $s\log(n)/n$ is small enough so that $h<\pi,$ we have under $\Pm_{\log(\rho_n^+)}$
\begin{align*}
&\left\{ \tau_{\log(\rho^-_n)}(Y^1)< \tau_{ \log(\rho_n^+)+\eps_n/2   }(Y^1) \right\} \cap \{ \tau_{\log(\rho^-_n)}(Y^1)< \tau_{h}(Y^2)\wedge \tau_{-h}(Y^2)  \}  \\
& \subset
\left\{ W[0, \tau_{ \log( \rho_n^- )}'] \subset B\left(W_0, \eps_n\right)\right\}.
\end{align*}
Therefore,
\begin{equation}
\label{e:proofboundonexit1}
\begin{split}
\Pm_{\log( \rho^+_n)} \left(  W[0, \tau_{ \log( \rho^-_n )}'] \not \subset B\left(W_0, \eps_n\right)\right) &\leq \Pm_{\log( \rho^+_n)} \left(\tau_{\log(\rho^-_n)}(Y^1)> \tau_{ \log(\rho_n^+)+\frac{\eps_n}{2}   }(Y^1)  \right) \\
&\ + \Pm_{\log( \rho^+_n)} \left( \tau_{\log(\rho^-_n)}(Y^1)> \tau_{h}(Y^2)\wedge \tau_{-h}(Y^2)  \right).
\end{split}
\end{equation}
Using the well-known formula for one-dimensional hitting times, see for instance Part II.1, Equation 2.1.2 in  \cite{MR1912205}, we see that by \eqref{e:inclusionballs} and \eqref{e:boundrho} 
\begin{equation}
\label{e:proofboundonexit2}
\begin{split}
    &\Pm_{\log( \rho^+_n)} \left(\tau_{\log(\rho^-_n)}(Y^1)> \tau_{\log(\rho_n^+)+ \frac{\eps_n}{2}   }(Y^1)  \right)= \frac{|\log(\rho_n^-/\rho_n^+)|}{|\log(\rho_n^-/\rho_n^+)|+ \eps_n/2} \leq\frac{c}{s}
\end{split}
\end{equation} 
and by \eqref{e:defrhoL}
The second term in \eqref{e:proofboundonexit1} can be bounded by
\begin{equation}\label{e:bmhit1}
\begin{split}
&\Pm_{\log( \rho^+_n)} \left( \tau_{\log(\rho^-_n)}(Y^1)> \tau_{h}(Y^2)\wedge \tau_{-h}(Y^2)  \right)\leq 2\Pm_{0} \left( \tau_{\log(\rho^+_n/\rho^-_n)}(Y^1)> \tau_{h}(Y^2)\right)
\\&\quad\leq 4 \E_0\left[1-\Phi\left( \frac{h}{\sqrt{\tau_{\log(\rho_n^+/\rho_n^-)}(Y^1)}} \right)\right]
\leq 4 \E_0\left[1-\Phi\left( \frac{cs}{\sqrt{\tau_{1}(Y^1)}} \right)\right],
\end{split}
\end{equation}
where $\Phi(x)$ denotes the CDF of a standard normal random variable, we used the reflection principle in the second inequality, and scaling invariance as well as \eqref{e:boundrho} and \eqref{e:defh*} in the last inequality. Using that $1-\Phi(y) \leq \frac{1}{\sqrt{2\pi}y}\e^{-y^2/2}$ for all $y>0,$ as well as a formula for the density of $\tau_{1},$ see for instance \cite[Part II.1, Equation 2.0.2]{MR1912205}, we moreover have
\begin{equation}
    \label{e:bXYsmall}
    \E_0\left[\left(1-\Phi\left( \frac{cs}{\sqrt{\tau_{1}(Y^1)}}\right)\right)\I_{\tau_{1}(Y^1)<s^{2}} \right]\leq \int_0^{s^{2}}\frac{c}{st}e^{-\frac{cs^{2}}{t}}\mathrm{d}t=\frac{c'}{s},
\end{equation}
where the last equality follows from the change of variable $t\mapsto s^{-2}t.$ Combining \eqref{e:proofboundonexit3}, \eqref{e:proofboundonexit1}, \eqref{e:proofboundonexit2},  \eqref{e:bmhit1}, \eqref{e:bXYsmall} together with the bound $\Pm_0(\tau_1(Y^1)\geq s^{2})=2\Phi(1/s)-1\leq C/s$ we thus obtain 

\[
\Pm_{\log(\rho_n^+)} \left( W[0,\tau_{ \log( \rho_n^- )}'] \not \subset B\big(\log(\rho_n^+), \eps_n\big)\, \middle|\, \tau'_{ \log( \rho_n^- )}<\tau_0' \right) \leq  \frac{c}{s},
\]
which, in view of \eqref{e:bmgambler}, \eqref{e:proofboundonexit0}, \eqref{e:proofboundonexit00} and \eqref{e:inclusionballs}, completes the proof.
\end{proof}

\section{Appendix: Convergence of capacities}
\label{app:convcap}
\renewcommand*{\thetheorem}{C.\arabic{theorem}}
\renewcommand{\theequation}{C.\arabic{equation}}

In this section, we prove convergence of the discrete capacity to the continuous capacity of general connected compacts $K,$ which generalizes Lemma~\ref{l.capconv}, and is used in the proof of Lemma~\ref{Lem:bigcaponboundary}. Recall that for $K\subset \D$ we write $K_n:= \D_n \cap K.$

\begin{proposition}
\label{t.gencapconv}
There exists $c<\infty$ such that for all connected sets $K\Subset \D$ and $n\in{\N}$ satisfying $K\subset B(K_n,2/n),$ letting $r= \sup_{x \in K} |x|,$ we have
\begin{equation}
\label{e.gencapconv}
\left| \capac(K)-\capac^{(n)}(K_n)\right| \leq 
  \frac{c}{\sqrt{1-r}} \left(\left(1+\mathrm{cap}(K)+\mathrm{cap}^{(n)}(K_n)\right)^2\frac{\log(n)}{n} \right)^{1/3}
\end{equation}
\end{proposition}

\begin{proof}
Write $B=B(r\vee(1/2))$ and let $B_n=B\cap\D_n.$ By the sweeping identity \eqref{e.consistency} and the corresponding discrete identity, see for instance \cite[(1.59)]{MR2932978},
we have
\begin{align*}
\frac{\capac(K)}{\capac(B)}-\frac{ \capac^{(n)}(K_n)}{\capac^{(n)}(B_n)}&=
   \Pm_{\overline{e}_B} \left( \tau_K<\infty \right)- \Pm_{\overline{e}_{B_n}^{(n)}}^{(n)} \left( \tau_{K_n}^{(n)} < \infty \right)
   \\&=\E_{\Q_r}\left[\Pm_{E_B}(\tau_K<\infty)-\Pm_{E_B^{(n)}}^{(n)}\left(\tau_{K_n}^{(n)}<\infty\right)\right],
\end{align*}
where $\Q_r$ denotes the coupling from Lemma~\ref{l.EQcoup}. Using Lemma~\ref{l.capconv} we thus have
\begin{equation}
\label{e.proofgenK1}
\begin{split}
\left|\frac{\capac(K)-\capac^{(n)}(K_n)}{\capac(B)} \right|&\leq \left|\frac{\capac(K)}{\capac(B)}-\frac{ \capac^{(n)}(K_n)}{\capac^{(n)}(B_n)}\right|+\capac^{(n)}(K_n)\left|\frac{1}{\mathrm{cap}(B)}-\frac{1}{\mathrm{cap}^{(n)}(B_n)}\right|
\\&\leq\E_{\Q_r}\left[\left|\Pm_{E_B}(\tau_K<\infty)-\Pm_{E_{B_n}^{(n)}}^{(n)}\left(\tau_{K_n}^{(n)}<\infty\right)\right|\right]+\frac{c\,\mathrm{cap}^{(n)}(K_n)}{n}.
\end{split}
\end{equation}
Now, using \eqref{e.boundapproequi} and recalling the coupling $\Pm_{z,x}^{(n)}$ from Lemma~\ref{lem:KMTuntilboundary}, we write for all $s>0$
\begin{equation}
    \label{e.proofgenK2}
\begin{split}
   & \E_{\Q_r}\left[\left|\Pm_{E_B}(\tau_K<\infty)-\Pm_{E_{B_n}^{(n)}}^{(n)}\left(\tau_{K_n}^{(n)}<\infty\right)\right|\right] \\& \leq \frac{c}{s}+ \frac{c\log(n)}{n(1-r)}
   + \E_{\Q_r}\left[ \Pm_{E_B,E_{B_n}^{(n)}}^{(n)} \big( \{\tau_K< \infty\}\Delta\{ \tau_{K_n}^{(n)}<\infty\}\big)\I\left\{|E_B-E_{B_n}^{(n)}|\leq \frac{s\log(n)}{n}\right\}  \right]
\end{split}
\end{equation}

We now bound the expectation on the right hand side of \eqref{e.proofgenK2}. Let $\overline{\tau}_K:=\inf\{t\geq0:\widehat{X}_t^{(n)}\in{K}\text{ or }Z_t\in{K}\}$ be the first time either $\widehat{X}^{(n)}$ or $Z$ hit $K,$ with the convention $\inf\varnothing=0,$ and note that $\overline{\tau}_K\leq\overline{L}_K,$ see above Lemma~\ref{lem:KMTuntilboundary}. Moreover due to our assumption on $K,$ we have $K_n\subset K\subset B(K_n,2/n).$ Letting 
\[
A = \left
\{ \sup_{0 \leq t \leq \overline{\tau}_K } | Z_t-\widehat{X}_t^{(n)}| \leq \frac{2s \log(n)}{n}  \right\},
\] 
 we then have for all $z,x$ with $|z-x|\leq s\log(n)/n$ and all $s\geq2$
\begin{align*}
\Pm_{z,x}^{(n)} 
    &\left( \tau_K< \infty, \tau_{K_n}^{(n)} = \infty, A \right) 
    \leq \Pm_{x}^{(n)} 
        \left( 
        \tau_{B_n(K_n,\eps_n)}^{(n)}< \infty, \tau_{K_n}^{(n)}  = \infty
         \right),  \\
\Pm_{z,x}^{(n)} 
    &\left( \tau_K= \infty, \tau_{K_n}^{(n)} < \infty, A \right) \leq \Pm_{z} 
            \left( 
            \tau_{B(K,\eps_n)}< \infty, \tau_{K}  = \infty 
            \right),
\end{align*}
where $\eps_n=3s\log(n)/n.$ Therefore, we have
\begin{equation}
    \label{e.proofgenK3}
\begin{split}
    & \E_{\Q_r}\left[ \Pm_{E_B,E_{B_n}^{(n)}} \left( \{\tau_K< \infty\}\Delta\{ \tau_{K_n}^{(n)}<\infty\}\right)\BI\left\{|E_B-E_{B_n}^{(n)}|\leq \frac{s\log(n)}{n}\right\}  \right] \\
    &\leq \Pm_{\overline{e}_B} (\tau_{B(K,\eps_n)}< \infty, \tau_{K}=\infty)+\Pm^{(n)}_{\overline{e}^{(n)}_{B_n}} \left( \tau_{B_n(K_n,\eps_n)}^{(n)}<\infty,\tau_{K_n}^{(n)}= \infty  \right)
    \\&\quad+  \sup_{x\in{\D_n},z\in{\D}:|x-z|\leq \frac{s\log(n)}{n}}\Pm_{z,x}^{(n)}(A^\ch).
\end{split}
\end{equation}
Using the Markov property and the sweeping identity  \eqref{e.consistency} we can estimate these quantities as follows
\begin{align*}
    \Pm_{\overline{e}_B} (\tau_{B(K,\eps_n)}< \infty, \tau_{K}=\infty) &= \E_{\overline{e}_B} \left[ \Pm_{Z_{\tau_{B(K,\eps_n)}}} \left(\tau_K= \infty\right) \I\{\tau_{B(K,\eps_n)} < \infty\} \right]
    \\&=\frac{\Pm_{e_{B(K,\eps_n)}} (\tau_K= \infty)}{\mathrm{cap}(B)}.
\end{align*}
Moreover, by the Beurling estimate, see Lemma~\ref{lem:beurling}, we have
\begin{equation}
\label{e.Petau=infty}
\Pm_{e_{B(K,\eps_n)}} (\tau_K= \infty)
    \leq \Pm_{{e}_{B(K,\eps_n)}} (\tau_{K}> \tau_{\partial B(Z_0,1-r) })
    \leq\sqrt{\frac{cs\log(n)}{n(1-r)}} \capac(B(K,\eps_n)).
\end{equation}
Note that by the sweeping identity \eqref{e.consistency}, the left-hand side of \eqref{e.Petau=infty} is actually equal to $\mathrm{cap}(B(K,\eps_n))-\mathrm{cap}(K),$ and in particular for $s\log(n)/(n(1-r))$ small enough, it implies $\mathrm{cap}(B(K,\eps_n))\leq 2\mathrm{cap}(K).$ Therefore using \eqref{capball} we obtain
\begin{equation}
    \label{e.boundKsKcont}
         \Pm_{\overline{e}_B} (\tau_{B(K,\eps_n)}< \infty, \tau_{K}=\infty)\leq \sqrt{\frac{cs(1-r)\log(n)}{n}}\capac(K) .
\end{equation}
Similarly, using the discrete sweeping identity, Beurling estimate \eqref{eq:disbeurling}, and \eqref{eq:approcap}, one can easily prove
\begin{equation}
    \label{e.boundKsKdis}
     \Pm_{\overline{e}_{B_n}^{(n)}}^{(n)} (\tau_{B_n(K_n,\eps_n)}< \infty, \tau_{K_n}^{(n)}=\infty)\leq \sqrt{\frac{cs(1-r)\log(n)}{n}} \capac^{(n)}(K_n).
\end{equation}
Combining  \eqref{eq:KMTuntilboundary}, \eqref{e.proofgenK2}, \eqref{e.proofgenK3}, \eqref{e.boundKsKcont} and \eqref{e.boundKsKdis}, we obtain for all $s$ satisfying $s\geq s_0$ and $s\log(n)/(n(1-r))$ small enough,
\begin{equation}
\label{eq:proofgenKf}
\begin{split}
    &\E_{\Q_r}\left[\left|\Pm_{e_B}(\tau_K<\infty)-\Pm_{e_{B_n}^{(n)}}\left(\tau_{K_n}^{(n)}<\infty\right)\right|\right]
     \\&\leq \frac{c}{s}+ \frac{cs\log(n)}{n(1-r)}+\left(\mathrm{cap}(K)+\mathrm{cap}^{(n)}(K_n)\right)\sqrt{\frac{cs(1-r)\log(n)}{n}} 
\end{split}
\end{equation}
and, taking $s=\big(cn/(\log(n))\big)^{1/3}\left(1+\mathrm{cap}(K)+\mathrm{cap}^{(n)}(K_n)\right)^{-2/3}\sqrt{1-r}$ for some large enough constant $c,$ which we can assume without loss of generality satisfies $s\geq s_0$ and $s\log(n)/(n(1-r))$ is small enough, we can conclude by \eqref{capball} and \eqref{e.proofgenK1}.
\end{proof}

Note that the choice of $s$ below \eqref{eq:proofgenKf} is not optimal, but the optimal choice depends on the relation between $r,$ $n$ and $\mathrm{cap}(K).$ Thus for certain values of these parameters the bound \eqref{e.gencapconv} could be improved, but we anyway do not believe  that the bounds obtained via this method would be optimal.

\bibliography{references} 
\end{document}